\documentclass[11pt,letterpaper]{amsart}
\usepackage{amsmath,amssymb,amsthm,bm}
\usepackage{tikz,caption,subcaption}
\usepackage{genyoungtabtikz}
\usepackage{ytableau,tikz,varwidth,cases}
\usetikzlibrary{calc}

\usetikzlibrary{calc}
\usepackage{ytableau}
\usepackage{dsfont}
\usepackage{multicol}
\usepackage{here}
  \usepackage{soul}
\setstcolor{green}

\usepackage[final]{hyperref}   
\hypersetup{
    linktoc=page,	
    linkcolor=blue,          
    citecolor=green,        
    filecolor=blue,      
    urlcolor=cyan,
   colorlinks=true      }     

\def\bbbr{\mathbb R}

\def\bbbz{\mathbb Z}
\def\bbbn{\mathbb N}

\def\cch{\mathcal H}
\def\ccsc{\mathcal {SC}}
\def\ccdd{\mathcal {DD}}
\def\ccp{\mathcal P}

\DeclareMathOperator{\core}{core}
\DeclareMathOperator{\quot}{quot}

\def\det{{\rm det}}

\def\sgn{{\rm sgn}}

\newtheorem{df}{Definition}[section]
\newtheorem{prop}[df]{Proposition}		

\newtheorem{thm}[df]{Theorem}
\newtheorem{lm}[df]{Lemma} 
\newtheorem{cor}[df]{Corollary}

\newtheorem{rk}[df]{Remark}

\theoremstyle{remark}

\DeclareMathOperator{\Id}{Id}

\DeclareMathOperator{\Sp}{Sp}
\DeclareMathOperator{\GL}{GL}
\DeclareMathOperator{\OO}{SO}
\renewcommand{\OE}{\mathrm{O}}
\DeclareMathOperator{\I}{I}

\newcommand{\Z}{\mathbb{Z}}
\newcommand{\N}{\mathbb{N}}

\newcommand{\os}[2]{\overset{#1}{#2}}

\renewcommand{\sp}{\mathrm{sp}}
\DeclareMathOperator{\oo}{so}
\renewcommand{\oe}{\mathrm{o}^{\text{even}}}

\usepackage{mathtools, stmaryrd}
\usepackage{xparse} 
\DeclarePairedDelimiterX{\Iintv}[1]{\llbracket}{\rrbracket}{\iintvargs{#1}}
\NewDocumentCommand{\iintvargs}{>{\SplitArgument{1}{,}}m}
{\iintvargsaux#1} %
\NewDocumentCommand{\iintvargsaux}{mm} {#1\mkern1.5mu..\mkern1.5mu#2}

\usepackage[margin=1.5in]{geometry}
\numberwithin{equation}{section}
\numberwithin{df}{section}

\begin{document}

\pagenumbering{arabic} 
\title[From Macdonald identities to Nekrasov--Okounkov type formulas]{Some combinatorial interpretations of the Macdonald identities for affine root systems and Nekrasov--Okounkov type formulas}


\author{David Wahiche}\address{Univ. de Tours, UMR CNRS 7013, Institut Denis Poisson, France}
\email{wahiche@univ-tours.fr}





\maketitle


\begin{abstract}
We explore some connections between vectors of integers and integer partitions seen as bi-infinite words. On the one hand, this methodology enables us to obtain enumerations connecting products of hook lengths and vectors of integers.
On the other hand, this yields a combinatorial interpretation of the Macdonald identities for affine root systems of the $7$ infinite families in terms of Schur functions, symplectic and special orthogonal Schur functions with respect to the type of the considered root system. From these results, we are able to derive $q$-Nekrasov--Okounkov formulas associated to each type. The latter for limit cases of $q$ yield Nekrasov--Okounkov type formulas corresponding to all the specializations given by Macdonald. When $q$ goes to $1$, one can derive combinatorial developments of Euler product, answering an open problem from Han.

\end{abstract}



\section{Introduction and notations}\label{chap2:partitions}

Formulas involving hook lengths abound in combinatorics and representation theory. One illustrative example is the hook-length formula discovered in 1954 by Frame, Robinson and Thrall \cite{FRT}. It states the equality between the number $f^\lambda$ of standard Young tableaux of shape $\lambda$ and size $n$, and the number of permutations of $\lbrace 1,\dots,n\rbrace$ divided by the product of the elements of the hook lengths multiset $\cch(\lambda)$ of $\lambda$, namely:
\begin{equation}\label{chap2:eqSYT}
f^\lambda=\frac{n!}{\displaystyle\prod_{h\in\cch(\lambda)}h}\cdot
\end{equation}
Hook lengths formulas such as \eqref{chap2:eqSYT} are useful enumerative tools bridging combinatorics with other fields such as representation theory, probability, gauge theory or algebraic geometry. A much more recent identity is the Nekrasov--Okounkov formula.  It was discovered independently by Nekrasov and Okounkov in their work on random partitions and Seiberg--Witten theory \cite{NO}, and by Westbury \cite{WeW} in his work on universal characters for $\mathfrak{sl}_n$. This formula is commonly stated as follows:
\begin{equation}\label{NOdebut}
\sum_{\lambda\in\ccp}T^{\lvert \lambda\rvert}\prod_{h\in\cch(\lambda)}\left(1-\frac{z}{h^2}\right)=\prod_{k\geq 1}\left(1-T^k\right)^{z-1},
\end{equation}
where $z$ is a fixed complex number.
This identity was later obtained independently by Han \cite{Ha}, based on one of the identities for affine type $A^{(1)}_{t-1}$~in \cite[Appendix $1$]{Mac} and a polynomiality argument. The proof lead the author to call \eqref{NOdebut} in \cite{Hancon} as an \textit{indiscretization} analogue of the Macdonald identity for affine type $A^{(1)}_{t-1}$. 
 The existence of similar identities for other types is a natural question that already arises in Han's paper \cite[Problem $6.4$]{Hancon}. We give a more general answer for all infinite affine root systems and the identities tackling the question from Han can be found in Subsection \ref{sec:macident}.
The approach detailed here should actually lead to an interpretation of Macdonald identities for exceptional types as hook length formulas, however in these cases the ranks are bounded, so one cannot lift these identities to an \textit{indiscretization} analogue as Han calls it. Therefore we do not deal with the exceptional affine types in this paper.
Before going any further, let us introduce some notations. For formal variables, the $T$-Pochhammer symbols are defined as $(a;T)_0=1$ and for any integer $n\geq 1$, as
\begin{flalign*}
&&(a;T)_n&:= (1-a)(1-aT) \ldots (1-aT^{n-1}),&&\\
&& \displaystyle (a;T)_\infty &:= \prod_{j\geq 0} (1-aT^j),\\
\text{and}  &&  \displaystyle (a_1,\dots,a_n;T)_\infty &:=(a_1;T)_\infty \dots (a_n;T)_\infty .
\end{flalign*}
We also introduce the three following shorthand notations:
\begin{align*}
(Ta_1^{\pm};T)_\infty &:=(Ta_1,Ta_1^{-1};T)_\infty ,\\
(Ta_1^{\pm 2};T)_\infty &:=(Ta_1^2,Ta_1^{-2};T)_\infty ,\\
(Ta_1^{\pm}a_2^{\pm};T)_\infty &:=(Ta_1^{-1}a_2,Ta_1a_2^{-1},Ta_1a_2,Ta_1^{-1}a_2^{-1};T)_\infty .
\end{align*}

A \textit{partition} $\lambda$ of a positive integer $n$ is a non-increasing sequence of positive integers $\lambda=(\lambda_1,\lambda_2,\dots,\lambda_\ell)$ such that $\lvert \lambda \rvert := \lambda_1+\lambda_2+\dots+\lambda_\ell = n$. The $\lambda_i$'s are the \textit{parts} of $\lambda$, the number $\ell$ of parts being the \textit{length} of $\lambda$, denoted by $\ell(\lambda)$. The well-known generating series for $\ccp$ can also be derived from \eqref{NOdebut} with $z=0$:
\begin{equation}\label{gspartitions}
\sum_{\lambda\in\ccp}T^{\vert \lambda\vert}=\left(T;T\right)_\infty^{-1}.
\end{equation}

Each partition can be represented by its Ferrers diagram, which consists of a finite collection of boxes arranged in left-justified rows, with the row lengths in non-increasing order. The \textit{Durfee square} of $\lambda$ is the maximal square fitting in the Ferrers diagram. Its diagonal, denoted by $\Delta$, will be called the main diagonal of $\lambda$. Its size will be denoted $d=d_\lambda:=\max(s | \lambda_s\geq s)$. As an example, in Figure~\ref{fig:ferrers}, the Durfee square of $\lambda=(4,3,3,2)$, which is a partition of $12$ of length $4$, is colored in red.
 The partition $\lambda'=(\lambda_1',\lambda_2',\dots,\lambda_{\lambda_1}')$ is the \textit{conjugate} of $\lambda$, where $\lambda_j'$ denotes the number of boxes in the column $j$.

%

\begin{figure}[H]
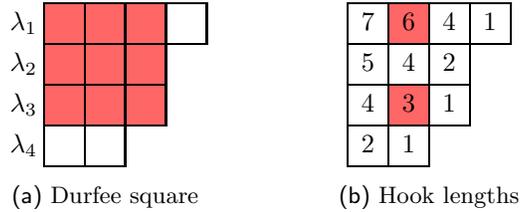

\centering
\begin{subfigure}[t]{.3\textwidth}
\centering
%
\begin{ytableau}
  \none[\lambda_1]  &  *(red!60)  &*(red!60)  & *(red!60) & \\
       \none[\lambda_2] & *(red!60) &*(red!60) &*(red!60) & \none \\
\none[\lambda_3] & *(red!60) &*(red!60)  &*(red!60) &\none \\
\none[\lambda_4] &  &  &\none & \none 
\end{ytableau}
\caption{Durfee square}
\label{fig:ferrers}
\end{subfigure}
\begin{subfigure}[t]{.3\textwidth}
\centering
\begin{ytableau}
  7  &*(red!60) 6 & 4 & 1\\
 5 &4 &2 & \none \\
 4 &*(red!60)3  &1 &\none \\
 2 & 1 &\none & \none \\
\end{ytableau}
\caption{Hook lengths}
\label{fig:hooks}
\end{subfigure}
\caption{Ferrers diagram and some partition statistics.}
\label{fig:fig1}
\end{figure}

For each box $v$ in the Ferrers diagram of a partition $\lambda$ (for short we will say for each box $v$ in $\lambda$), one defines the \textit{arm-length} (respectively \textit{leg-length}) as the number of boxes in the same row (respectively in the same column) as $v$ strictly to the right of (respectively strictly below) the box $v$. The \textit{hook length} of $v$, denoted by $h_v(\lambda)$ or $h_v$, is the number of boxes $u$ such that either $u=v$, or $u$ lies strictly below (respectively to the right) of $v$ in the same column (respectively row).
The \textit{hook lengths multiset} of $\lambda$, denoted by $\mathcal{H}(\lambda)$, is the multiset of all hook lengths of $\lambda$. For any positive integer $t$, the multiset of all hook lengths that are congruent to $0 \pmod t$ is denoted by $\mathcal{H}_t(\lambda)$. Notice that $\mathcal{H}(\lambda)=\mathcal{H}_1(\lambda)$. A partition $\omega$ is a \textit{$t$-core} if $\cch_t(\omega)=\emptyset$. For any $\mathcal{A}\subset\ccp$, we denote by $\mathcal{A}_{(t)}$ the subset of elements of $\mathcal{A}$ that are $t$-cores. For example, the only $2$-cores are the ``staircase'' partitions $(k,k-1,\dots,1)$, where $k$ is any positive integer. In Figure \ref{fig:hooks}, the hook lengths of all boxes for the partition $\lambda=(4,3,3,2)$ are written in their corresponding boxes, the boxes associated with $\mathcal{H}_3(\lambda)$ are shaded in red. The hook-length multiset $\mathcal{H}(\lambda)$ is here given by $\lbrace 2,1,4,3,1,5,4,2,7,6,4,1\rbrace$ and $\mathcal{H}_3(\lambda)=\lbrace 3,6\rbrace$.
Formula \eqref{NOdebut} was first generalized with two parameters in \cite[Equation $(3.7)$]{IKS}. Later, simultaneously both Rains--Warnaar \cite{RW}, by using refined skew Cauchy-type identities for Macdonald polynomials, and Carlsson--Rodriguez Villegas \cite{CRV}, by means of vertex operators and the plethystic exponential, discovered a $(q,t)$-extension of the Nekrasov--Okounkov formula to settle a conjecture from Hausel--Rodriguez Villegas \cite[Conjecture $4.3.2$]{HRV} in their study of Mixed Hodge polynomials of character varieties with genus $1$. 
\begin{thm}\cite[Theorem $1.0.2$]{CRV},\cite[Theorem $1.3$]{RW}\label{qtNO}
We have
\begin{multline*}
\sum_{\lambda\in\ccp}T^{\lvert \lambda\rvert}\prod_{s\in\lambda}\frac{(1-uq^{a(s)+1}t^{l(s)})(1-u^{-1}q^{a(s)}t^{l(s)+1})}{(1-q^{a(s)+1}t^{l(s)})(1-q^{a(s)}t^{l(s)+1})}\\
=\prod_{i,j,k\geq 1}\frac{(1-uq^it^{j-1}T^k)(1-u^{-1}q^{i-1}t^{j}T^k)}{(1-q^{i-1}t^{j-1}T^k)(1-q^{i}t^jT^k)}.
\end{multline*}
\end{thm}

As mentioned in \cite{CRV,RW}, the special case $q=t$ is a reformulation of a result due to Dehaye--Han \cite{HD}, using a specialization of the Macdonald identity for type $A^{(1)}_{t-1}$, and Iqbal--Nazir--Raza--Saleem \cite{INRS}, using the refined topological vertex. It can be written as follows:
\begin{equation}\label{Hande}
\sum_{\lambda\in\ccp}T^{\lvert \lambda\rvert}\prod_{h\in\cch(\lambda)}\frac{(1-uq^h)(1-u^{-1}q^h)}{(1-q^h)^2}=\prod_{k,r\geq 1}\frac{(1-uq^rT^k)^r(1-u^{-1}q^rT^k)^r}{(1-q^{r-1}T^k)^r(1-q^{r+1}T^k)^r}.
\end{equation}

Note that taking $u=q^n$  for a strictly positive integer $n$ and letting $q\rightarrow 1$ in \eqref{Hande} yields \eqref{NOdebut} for $z=n^2$. Note that for any positive integer $k$, the coefficients in $T^k$ on both sides of \eqref{NOdebut} are polynomials in $z$. It is obvious on the left-hand side whereas for the right-hand side, one uses, as in \cite[$(3.10)$]{Han08}, the following relation:
$$ \left(T;T\right)_\infty^{z-1}=\exp\left(-(z-1)\sum_{k\geq 1}\frac{x^k}{k(1-x^k)}\right).$$
Therefore we can use the polynomiality argument which tells us that the identity holds for any complex number $z$.

The methods used by Han to prove \eqref{NOdebut}, and Dehaye--Han for \eqref{Hande}, both start from the type $A^{(1)}_{t-1}$ Macdonald's identity. Nevertheless one needs to keep track of the variables to derive \eqref{Hande}, whereas \eqref{NOdebut} is obtained by specializing all the variables to $1$. In order to  do this for our purpose, one needs some reformulation dispensing of the more abstract sum over the affine Weyl group, as done for instance in the work of Gustafson through his $_6\psi_6$ identities on root systems \cite{Gustafson}. In our case, it is convenient to use the formulation given by Rosengren--Schlosser \cite{RS} (see also Stanton's reformulation in \cite{Stanton}). The next step in \cite{HD} for proving \eqref{Hande} needs new combinatorial notions such as exploded tableaux and a $V_{t}$-coding. In this paper, we extend this last notion to $V_{g,t}$-coding (see Definition \ref{def:vcoding}) adapted from Garvan--Kim--Stanton \cite{GKS}, which comes naturally from the study of the Littlewood decomposition described in Section \ref{sec:lit}.

A natural question is then what happens when one tries to derive Nekrasov--Okounkov formulas for other affine types. In the process of finding such interpretations, some subsets of partitions appear and are introduced hereafter. A partition $\lambda$ is \textit{self-conjugate} if its Ferrers diagram is symmetric along the main diagonal. Let $\mathcal{SC}$ be the set of self-conjugate partitions. The set of doubled distinct partitions, denoted by $\ccdd$, is that of all partitions $\lambda$ of Durfee square size $d$ such that $\lambda_i=\lambda_i'+1$ for all $i\in\lbrace 1,\dots,d\rbrace$. We also define the set of conjugates of doubled distinct partitions $\ccdd':=\lbrace \lambda'\mid \lambda\in\ccdd\rbrace$. For instance, in Figure \ref{fig:varsc} $\lambda=(5,3,3,1,1)\in\ccsc$ has its main diagonal $\Delta$ shaded in green while in Figure \ref{fig:vardd} $\lambda=(6,4,4,1,1)\in\ccdd$ has its main diagonal shaded in green as for the strip shaded in yellow, it corresponds to the boxes added to a self-conjugate partition to obtain a doubled distinct partition. The conjugate of a doubled distinct partition is also illustrated in Figure \ref{fig:varddprime}. The subsets $\ccsc$ and $\ccdd$ appeared within the works of Littlewood \cite{Little}, Macdonald \cite[p.78--79]{Macbook} and of  Koike--Terada \cite{KoikeTerada} in their study of Young-diagrammatic representation of the classical Lie groups $B_n$, $C_n$ and $D_n$, or in the work of King \cite{King} and more recently within the works of Ayyer--Kumari \cite{AyyerKumari} and Albion \cite{Albion}. They are of particular interest in the papers by P\'etr\'eolle \cite{Pe,MP} where two Nekrasov--Okounkov type formulas for affine root systems of types $C^{(1)}$ and $A^{(2)}_{odd}$ are derived. For instance, in~~\cite{Pe,MP}, the author proves the following $\ccsc$ Nekrasov--Okounkov type formula, similar to~\eqref{NOdebut}, giving a partial answer to \cite[Problem $6.4$]{Hancon}, and which stands, for any complex number $z$, as: 
\begin{equation}\label{eq:nospecc}
\sum_{\lambda\in\ccdd}(-1)^{d_\lambda}T^{\lvert \lambda\rvert/2}\prod_{s\in\lambda}\left(1-\frac{(2z+2)\varepsilon_s}{h_s}\right)=\left(T;T\right)_{\infty}^{2z^2+z}.
\end{equation}
Here for a box $s$ of $\lambda$, \textit{$\varepsilon_s$} is defined as $-1$ if $s$ is strictly below the main diagonal of the Ferrers diagram of $\lambda$ and as $1$ otherwise, as depicted in Figure \ref{fig:var}. This signed statistic already appears algebraically within the work of King \cite{King} and combinatorially within the work of Pétréolle \cite{MP}.

A \emph{rim hook} (or border strip, or ribbon) is a connected skew shape containing no $2\times2$ square. The length of a rim hook is the number of boxes in it, and its height is one less than its number of rows. By convention, the height of an empty rim hook is zero. The Littlewood decomposition is a classical bijection in partition theory (see for instance \cite[Theorem 2.7.17]{JK}). Roughly speaking, for any fixed positive integer $t$, this transforms $\lambda\in\ccp$ into two components: the $t$-core $\omega$, which is obtained by removing sequentially all the rim hooks of length $t$ of $\lambda$, and the $t$-quotient $\underline{\nu}$ (see Section \ref{sec:lit} for precise definitions and properties):
$$
\lambda\in\ccp\mapsto \left(\omega,\underline{\nu}\right)\in\ccp_{(t)}\times \ccp^t.
$$

\begin{figure}[H]
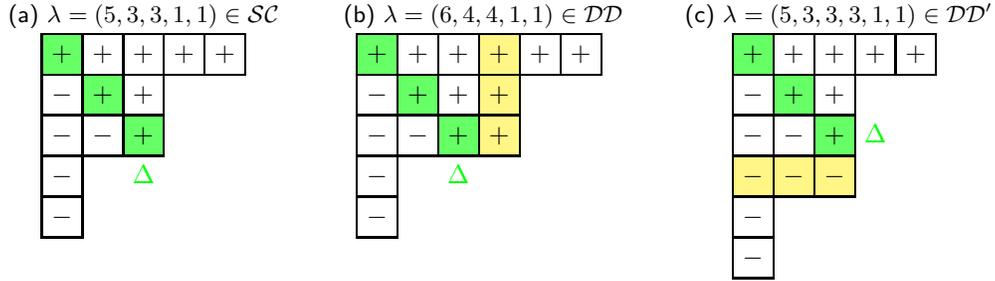

\centering
\begin{subfigure}[t]{.3\textwidth}
\caption{$\lambda=(5,3,3,1,1)\in\ccsc$}

\centering

\begin{ytableau}
 *(green!60)+  &+ & + &+ & +\\
 - &*(green!60)+ &+ & \none \\
 - &-  &*(green!60)+&\none \\
  -&\none &\none[\textcolor{green}{\Delta}] &\none  \\
 - &\none &\none &\none \\
\end{ytableau}

\label{fig:varsc}
\end{subfigure}
\begin{subfigure}[t]{.33\textwidth}
\caption{$\lambda=(6,4,4,1,1)\in\ccdd$}

\centering
\begin{ytableau}
 *(green!60)+  &+ & + &*(yellow!60)+&+ & +\\
 - &*(green!60)+ &+ &*(yellow!60)+ & \none \\
 - &-  &*(green!60)+&*(yellow!60)+ &\none \\
  -&\none &\none[\textcolor{green}{\Delta}] &\none & \none \\
 - &\none &\none &\none & \none \\
\end{ytableau}
\label{fig:vardd}
\end{subfigure}
\begin{subfigure}[t]{.33\textwidth}
\caption{$\lambda=(5,3,3,3,1,1)\in\ccdd'$}
\centering
\begin{ytableau}
 *(green!60)+  &+ & + &+&+ \\
 - &*(green!60)+ &+ & \none \\
 - &-  &*(green!60)+&\none[\textcolor{green}{\Delta}] \\
 *(yellow!60)-&*(yellow!60)-&*(yellow!60)-&\none\\
  -&\none &\none &\none & \none \\
 - &\none &\none &\none & \none \\
\end{ytableau}
\label{fig:varddprime}
\end{subfigure}
\caption{A self-conjugate partition, a doubled distinct partition and its conjugate filled with $\varepsilon$.}
\label{fig:var}
\end{figure}


Let $n$ be a positive integer and let $\mathbf{x}=(x_1,\dots,x_n)$ be a set of any independent variables.
Recall the Weyl denominator formula for a reduced root system $R$ (see for instance p.185 of \cite{Bourbaki})
\begin{equation}\label{chap2:macdodenom}
\sum_{\sigma\in W}\sgn(\sigma)e^{\sigma(\rho)-\rho}=\prod_{\alpha\in R_+}\left(1-e^{-\alpha}\right),
\end{equation}
where $W$ is the Weyl group associated with $R$, $R_+$ is the set of positive roots, $e^\alpha$ is a formal exponential and $\rho= \sum_{\alpha\in R_+}\alpha/2$. 

 In type $A_{n-1}$, it corresponds to the \textit{Vandermonde determinant} $\Delta_A(\textbf{x})$ following notations of Rosengren--Schlosser in \cite{RS}:
\begin{equation}\label{chap2:delta}
\Delta_A(\mathbf{x}):=\underset{1\le i,j\le n}{\det}\left(x_i^{j-1}\right)=\prod_{1\le i<j\le n}(x_j-x_i).
\end{equation}
Following \cite{RS}, the Weyl denominator formula in type $B_n$ is:
\begin{align}\label{def:deltab}
\Delta_B(\mathbf{x}):&=\underset{1\le i,j\le n}{\det}\left(x_i^{j-n}-x_i^{-(j-n)+1}\right)\notag\\
&=\prod_{1\le i \le n}x_i^{1-n}(1-x_i)\prod_{1\le i<j\le n}(x_j-x_i)(1-x_ix_j),
\end{align}
and in type $C_n$ it is:
\begin{align}\label{def:deltac}
\Delta_C(\mathbf{x}):&=\underset{1\le i,j\le n}{\det}\left(x_i^{j-n-1}-x_i^{-(j-n-1)}\right)\notag
\\&=\prod_{1\le i \le n}x_i^{-n}(1-x_i^2)\prod_{1\le i<j\le n}(x_j-x_i)(1-x_ix_j),
\end{align}
while in type $D_n$ it is:
\begin{align}\label{def:deltad}
\Delta_D(\mathbf{x})&:=\underset{1\le i,j\le n}{\det}\left(x_i^{j-n}+x_i^{-(j-n)}\right)=2\prod_{1\le i \le n}x_i^{1-n}\prod_{1\le i<j\le n}(x_j-x_i)(1-x_ix_j).
\end{align}

 Let $\bbbz[\mathbf{x}]:=\bbbz[x_1,\dots,x_n]$ be the ring of polynomials in $x_1,\dots,x_n$ with integer coefficients. A polynomial $P(\mathbf{x})$ in $\bbbz[\mathbf{x}]$ is \textit{symmetric} if it is invariant under permutations of its variables, i.e. for every $\sigma\in S_n$:
$$P(x_1,\dots,x_n)=P\left(x_{\sigma(1)},\dots,x_{\sigma(n)}\right).$$

The set of all symmetric polynomials in $n$ variables, denoted by $\Lambda_n$, is a $\bbbz$-module over $\bbbz[\textbf{x}]$. For a partition $\lambda=(\lambda_1,\dots,\lambda_n)$, the \emph{Schur polynomial} is given by
\begin{equation}
\label{gldef}
s_\lambda(\mathbf{x})=\frac{\underset{1\le i,j\le n}{\det}\left(x_i^{\lambda_j+n-j}\right)}
{\underset{1\le i,j\le n}{\det}\left(x_i^{n-j}\right)}.
\end{equation}

Schur polynomials have been an object of extensive studies; they can be thought of an orthonormal basis of $\Lambda_n$ for the Hall inner product (see for instance \cite{Macbook}). By the alternating property of the determinant, a Schur polynomial is invariant by any $\sigma\in S_n$ as a quotient of determinants. One can also think of a Schur polynomial as a \emph{character} of the polynomial representations of $\GL_n$, or as a sum over semi-standard Young tableaux (see \cite{FH}).

If one thinks of the Schur polynomials as invariant by the action of any element in $S_n$, one can introduce the analogues for other groups of signed permutations (see \cite{AyyerKumari,FH,Kradet1,Kradet2} for more details and background).
We write down the explicit Weyl character formulas for the infinite families of classical groups at the representation indexed by $\lambda = (\lambda_1,\dots,\lambda_n)$.

The \emph{odd orthogonal (type B) character} of the group $\OO(2n+1)$ is given by
\begin{align}
\label{oodef}
\oo_\lambda(\mathbf{x})&:=
\frac{\underset{1\le i,j\le n}{\det}\left(x_i^{\lambda_j+n-j+1/2}-x_i^{-(\lambda_j+n-j+1/2)}\right)}
{\underset{1\le i,j\le n}{\det}\left(x_i^{n-j+1/2}-x_i^{n-j+1/2}\right)} \notag\\
&= \frac{\underset{1\le i,j\le n}{\det}\left(x_i^{\lambda_j+n-j+1}-x_i^{-(\lambda_j+n-j)}\right)}
{\underset{1\le i,j\le n}{\det}\left(x_i^{n-j+1}-x_i^{-(n-j)}\right)}.
\end{align}

The \emph{symplectic (type C) character} of the group $\Sp(2n)$ is given by
\begin{equation}
\label{spdef}
\sp_\lambda(\mathbf{x}):=
\frac{\underset{1\le i,j\le n}{\det}\left(x_i^{\lambda_j+n-j+1}-x_i^{-(\lambda_j+n-j+1)}\right)}
{\underset{1\le i,j\le n}{\det}\left(x_i^{n-j+1}-x_i^{-(n-j+1)}\right)}.
\end{equation}

Lastly, the \emph{even orthogonal (type D) character} of the group $\OE(2n)$ is given by
\begin{equation}
\label{oedef}
\oe_\lambda(\mathbf{x}):=
\frac{2\underset{1\le i,j\le n}{\det}\left(x_i^{\lambda_j+n-j}+x_i^{-(\lambda_j+n-j)}\right)}
{\displaystyle (1+\delta_{\lambda_n,0})
\underset{1\le i,j\le n}{\det}\left(x_i^{n-j}+x_i^{-(n-j)}\right)},
\end{equation}
where $\delta$ is the Kronecker delta. 
There is an extra factor in the denominator because the last column becomes $2$ if $\lambda_n = 0$.


%

The Macdonald identity for affine root systems can be thought of as a generalization of the Weyl denominator formulas: it connects a sum depending on the associated Weyl group $W$, which is a group of (signed) permutations acting on the affine root system, and a parameter $g$, to a product over the positive roots. 

As we will not need it here, we do not give the definitions of (affine) root systems. Nevertheless we refer an interested reader, who wants to read a proper introduction and the definitions of finite and affine root systems, to \cite[Part $I$ Chapter $2$]{Bartphd}. Using the notations of Macdonald \cite{Mac}, there are $7$ infinite reduced affine root systems: $\tilde{A}_{t-1}$,$\tilde{B}_t$, $\tilde{B}^{\vee}_t$, $\tilde{C}_t$, $\tilde{C}^{\vee}_t$, $\tilde{D}_t$ and $\widetilde{BC}_t$. Note that the notations of Kac in his book \cite[p.80]{KacB} slightly differ, $\tilde{X}_t$ becomes $X_t^{(1)}$ for $X$ in $A,B,C,D$, and $\tilde{B}^{\vee}_t$, $\tilde{C}^{\vee}_t$, and $\widetilde{BC}_t$ are respectively denoted by $A_{2t-1}^{(2)},D_{t+1}^{(2)}$ and $A_{2t}^{(2)}$.


Section \ref{chap3:quad} is devoted to the study of the links between some quadratic forms over $\bbbz^t$ that appear in the Macdonald identities and some subsets of partitions whose description is given by the Littlewood decomposition. To do so, we explore an extension of a bijection already studied by Garvan--Kim--Stanton in \cite{GKS} between vector of integers and $t$-cores: we study these extensions to $7$ subsets of partitions, whose weight can be seen as a quadratic form that appears in the Macdonald identities for the $7$ infinite affine root systems. Moreover we derive enumerative results on the products of hook lengths of these subsets of partitions and vectors of integers where the $V_{g,t}$-coding is a key tool (see Definition \ref{def:vcoding}).
 These enumerative results will be necessary first to rewrite Macdonald identities for the $7$ infinite affine root systems in Section \ref{chap4:macdo}, and then to derive $q$-Nekrasov--Okounkov type formulas in Section \ref{sec:NO}. As a first consequence of this combinatorial analysis, we are able to rewrite the sum part of Macdonald identities for all $7$ infinite root systems as a sum over symplectic, special orthogonal and even orthogonal Schur functions. Indeed the type of enumeration used to prove results in Section \ref{chap3:hooks} enables to also connect the signature of a signed permutation of the Weyl group with respect to some characteristics of $\lambda$ such as the size of the Durfee square, the cardinal of the set of boxes of $\lambda$ having hook lengths less than $g$ and the main diagonal $\Delta$ of $\lambda$.
The rewriting of the Macdonald identity \eqref{eq:macdoAfin} for type $A^{(1)}_{t-1}$ gives the following result.
\begin{thm}\label{prop:MacdotypeA}
Set $t\in\mathbb{N}^*$. The Macdonald identity for type $A^{(1)}_{t-1}$ can be rewritten as follows:
\begin{equation}\label{eq:macdoa}
\sum_{\omega\in\ccp_{(t)}} (-1)^{\lvert H_{< t} \rvert}T^{\lvert\omega\rvert}  s_{\mu}({\bf x})\prod_{i=1}^t x_i^{-\ell(\omega)}
=\left(T;T\right)_\infty^{t-1}\prod_{1\le i<j\le t}\left(T(x_ix_j^{-1})^{\pm};T\right)_\infty.
\end{equation}
where $H_{< t}:=\lbrace s\in\omega\mid h_s<t\rbrace$, $\mu\in\ccp$ is such that $\mu_i:=v_i-v_t+i-t$ for all $1\leq i \leq t$ and $\mathbf{v}$ is the $V_{t,t}-$coding corresponding to $\omega$ (see Definition \ref{def:vcoding}).
\end{thm}

 For a positive integer $t$, let us introduce the compact notation:

\begin{equation}\label{eq:kt}
 K_T(t,\mathbf{x}):=\prod_{1\leq i<j\leq t}\left(Tx_i^{\pm} x_j^{\pm};T\right)_\infty.
\end{equation}

We prove similar results as Theorem \ref{prop:MacdotypeA} for the $6$ remaining infinite families of affine root systems. For example, in type $C^{(1)}_{t}$, we get the following result
\begin{thm}\label{prop:MacdotypeC}
Set $t\in\mathbb{N}^*$ such that $t\geq 2$ and $g=2t+2$. The Macdonald identity for type $C^{(1)}_{t}$ can be rewritten as follows:
\begin{equation}\label{eq:macdoc}
\sum_{\omega\in\ccdd_{(g)}} (-1)^{d_{\omega}+\lvert H_{<g,+}\rvert}T^{\lvert\omega\rvert/2}\sp_{\mu}({\bf x})=
\left(T;T\right)_\infty^tK_T(t,{\bf x})
\prod_{i=1}^t \left(Tx_i^{\pm 2};T\right)_\infty ,
\end{equation}
where $H_{<g,+}:=\lbrace s\in\omega, h_s<g, \varepsilon_s=1\rbrace$, $d_\omega$ is the length of the main diagonal of $\omega$, $\mu\in\ccp$ is such that
$\mu_i:=v_i+i-g$ for all $1\leq i \leq t$ and $\mathbf{v}$ is the $V_{g,t}$-coding corresponding to $\omega$ (see Definition \ref{def:vcoding}).
\end{thm}

Table \ref{tableau} at the end of Section \ref{sec:Macdo} summarizes the families of partitions and their corresponding Schur functions associated with the $7$ infinite affine types.

Actually the above theorems admit a uniform statement for all types as it is explained in \cite{LW} formulated in a setting closer to the one exposed in Macdonald \cite{Mac}. Ever since the present work revolves more around some interpretations of Macdonald identities seen as multivariate power series (following the works of \cite{Stanton,RS}), we mention here the reformulation obtained in \cite{LW} for a reader coming with a different point of view.
For the following paragraph, one refers an interested reader to \cite{carter_affine,KacB} for the notations and to \cite{LW} for the proof.
Let $A$ be a generalized Cartan matrix of a classical affine root system and $\os{\ \circ}{A}$ is a Cartan matrix of finite type obtained from $A$ by deleting the row and the column corresponding to $0$.
Let $R_+$ the set of positive affine roots and $\mathring{R}_{+}$ the set of positive (non affine) roots. Let $\mathring{\rho}$ denote the sum of positive (non affine) roots.

Let $W_a$ and $W$ respectively denote the affine Weyl group and the Weyl group. Let $M^{*}$ be the lattice in \cite[Section 20]{carter_affine}, then one has $$W_a\simeq W\ltimes M^{*}.$$ 

Moreover recall that for any classical weight $\gamma\in \mathring{P}$,
$$a_\gamma=\sum_{w\in W}\varepsilon(w)e^{w(\gamma)}.$$
From the Weyl character formula, one can define the virtual characters for any $\gamma\in \mathring{P}$,
$$s_\gamma=\frac{a_{\gamma+\mathring{\rho}}}{a_{\mathring{\rho}}}.$$

The affine grassmannian elements are the elements of minimal length in the cosets of $W_{a}/W$. We shall denote by $W_{a}^{0}$ these affine grassmannian elements. 
Write as usual $W^{\nu}$ for the set of minimal length elements in the cosets of
$W/W_{\nu}$, where $W_{\nu}$ is the stabilizer of $\nu$ under the action of the finite Weyl group and $t_\nu$ for $u\in M^{*}$ as defined in \cite[p.485]{carter_affine}.

	With the previous notations, as shown in \cite{LW}, the Macdonald identity can	be written
	\[
	\prod_{\alpha\in R_{+}\setminus\mathring{R}_{+}}(1-e^{-\alpha
})^{m_{\alpha}}=\sum_{c\in W_{a}^{0}}\varepsilon(c)q^{\left\vert c\right\vert }%
	s_{-\eta^{\vee}\nu+u^{-1}(\mathring{\rho})-\mathring{\rho}}%
	\]
	where $m_{\alpha}$ is the multiplicity of $\alpha$, $\eta^{\vee}$ is the dual Coxeter number, $Q$ is the root lattice, and where we set $c=ut_{\nu}$ with $\nu\in Q\cap(-P_{+})$ and $u\in W^{\nu}$ for
	any element of the affine grassmannian $W_{a}^{0}$. Here $\left\vert c\right\vert$ corresponds to the dual of the so-called atomic length as introduced by Chapelier-Laget and Gerber in \cite{CGT}. Moreover, each weight
	$-\eta^{\vee}\nu+u^{-1}(\mathring{\rho})-\mathring{\rho}$ belongs to $P_{+}$,
	the set of dominants weights for $\mathring{A}$.

%
As a second consequence, these enumerative results allow us to derive $q$-Nekrasov--Okounkov type formulas, such as \eqref{Hande} in type $A^{(1)}_{t-1}$, but also $6$ new results corresponding to the $6$ other infinite types. These formulas are derived from a specialisation of the rewriting of Macdonald identities obtained in Section \ref{sec:Macdo} together with the theorems from Section \ref{chap3:hooks}. For instance, setting $x_i=q^i$ in Theorem \ref{prop:MacdotypeC}, and $\tau(x)=1-q^x$ in Theorem \ref{thm:DDpair}, with a polynomiality argument, we derive the following $q$-Nekrasov--Okounkov formula for type $C^{(1)}_{t}$.
\begin{thm}[A $q$-Nekrasov--Okounkov formula for type $C^{(1)}_t$] \label{NOC}
For formal variables $T$, $q$ and $u$, we have:
\begin{multline}
\sum_{\lambda\in\ccdd} (-u)^{d_\lambda}T^{\lvert \lambda\rvert/2}\prod_{s\in\lambda}\frac{1-u^{-2\varepsilon_s}q^{h_s}}{1-q^{h_s}}\prod_{s\in\Delta}\frac{1+uq^{h_s/2}}{1+u^{-1}q^{h_s/2}}\\=
\prod_{m,r\geq 1}\frac{1+uq^{r-1}T^m}{1+u^{-1}q^{r}T^m}\frac{\left(1-u^{-2}q^{r+2}T^m\right)^{r-\lfloor r/2\rfloor}\left(1-u^2q^{r-1}T^m\right)^{\lceil r/2 \rceil}}{\left(1-q^{r}T^m\right)^{\lceil r/2 \rceil}\left(1-q^{r+1}T^m\right)^{\lceil r/2 \rceil}}.\label{eq:noc}
\end{multline}
\end{thm}

Finally, from these $q$-Nekrasov--Okounkov type formulas, we are able to derive 13 Nekrasov--Okounkov formulas corresponding to all of Macdonald's specializations in \cite[Appendix $1$]{Mac}, where we discuss some subtleties in type $A^{(2)}_{2t}$. For example, regarding the only specialization in type $D^{(1)}_t$, the corresponding Nekrasov--Okounkov formula is as follows:
\begin{align}\label{eq:nospecd}
\sum_{\lambda\in\ccdd'}(-1)^{d_\lambda}T^{\lvert \lambda\rvert/2}\prod_{s\in\lambda}\left(1-\frac{(2z-2)\varepsilon'_s}{h_s}\right)&=\left(T;T\right)_{\infty}^{2z^2-z}\\
&=\sum_{\lambda\in\ccdd}(-1)^{d_\lambda}T^{\lvert \lambda\rvert/2}\prod_{s\in\lambda}\left(1+\frac{(2z-2)\varepsilon_s}{h_s}\right),\notag
\end{align}
where $z$ is any complex number, and $\varepsilon'_s$ equals $1$ if $s$ is strictly above the main diagonal $\Delta$ and $-1$ otherwise. Formulas \eqref{NOdebut} and \eqref{eq:nospecc} are also derived through our approach.


The organization of this paper is as follows. The first section is devoted to the exploration of connections between binary bi-infinite words and statistics of boxes of integer partitions, in particular their hook lengths, which satisfy some constraints under the Littlewood decomposition. This notion enables us to compute explicitly the Littlewood decomposition as described in \cite{HJ}. This way of looking has been used with different names in the literature such as abaci (see for instance \cite{BN}), fermionic viewpoint (see for instance \cite{Johnson}), or Maya diagrams.
Then we show in Section \ref{sec:Macdointer} the connections between Macdonald identities as written in \cite[Proposition $6.1$]{RS} and quadratic forms. In the following section, we exploit the Littlewood decomposition and its interpretation as words to obtain a different way of computing \cite[Bijection $2$]{GKS} which maps a $t$-core partition to a vector of $t$ integers. In Section \ref{chap4:macdo}, we explain how using these techniques allows us not only to explicitly get a connection between the $7$ quadratic forms of Section \ref{sec:Macdointer} and integer partitions, but also to enumerate the products of hook lengths such as in Theorems \ref{thm:HD} and \ref{thm:DDpair} for all these $7$ families of partitions. In particular, note that the results of Section \ref{chap4:sign} 
Finally in Section \ref{sec:NO}, we use these connections to rewrite the Macdonald identities for the $7$ infinite root systems as $q$-Nekrasov--Okounkov formulas and then discuss how one can derive Nekrasov--Okounkov formulas for all the $13$ specializations given by Macdonald.

%
%

\section{Combinatorial properties of the Littlewood decomposition on some subsets of partitions} \label{sec:lit}

This paper in particular aims to provide a transcription of Macdonald identities introduced in Proposition \ref{prop:RS} in terms of integer partitions. From the discussion of the introduction, this requires to establish a connection between vectors of integers and integer partitions. Such a connection can be obtained by considering integer partitions through their bijective correspondence with binary bi-infinite words. This bijection is connected to many tools used within different theories such as beta sets (see for instance \cite{Albion,BN}) or vertex operators.
In this section, we first introduce this correspondence. Then we discuss the Littlewood decomposition, which is the main tool of this paper. As recalled with some history and context in \cite[Section $2.2$]{Albion}, this decomposition has many equivalent descriptions, see for instance \cite{GKS,HJ,WaW}.

We use here the formalism of Han--Ji of \cite{HJ}, as already used by the author in \cite{Wmult,WFpsac}. Let $\partial \lambda$ be the border of the Ferrers diagram of $\lambda$. Each step on $\partial\lambda$ is either horizontal or vertical. Encode the walk along the border from the South-West to the North-East as depicted in Figure \ref{fig:word}: take ``$0$'' for a vertical step and ``$1$'' for a horizontal step. This yields a $0/1$ sequence. The resulting word over the $\lbrace 0,1\rbrace$ alphabet:
\begin{itemize}
\item contains infinitely many ``$0$'''s at the beginning (respectively ``$1$'''s at the end),
\item is indexed by $\bbbz$,
\item and is written $(c_k)_{k\in\mathbb{Z}}$.
\end{itemize}
 
This writing as a sequence is not unique (since for any $k$ sequences $(c_{k+i})_{i\in\mathbb{Z}}$ encode the same partition, note that this is related to the notion of charge in fermionic Fock space, see for instance Section $3$ in \cite{BBNV}), hence the necessity to set the index $0$ uniquely for that encoding to be bijective. To tackle that issue, we set the index $0$ when the number of ``$0$'''s to the right of that index is equal to the number of ``$1$'''s to the left. In other words, the number of horizontal steps along $\partial\lambda$ corresponding to a ``$1$'' of negative index in $(c_k)_{k\in\bbbz}$ must be equal to the number of vertical steps corresponding to ``$0$'''s of nonnegative index in $(c_k)_{k\in\bbbz}$ along $\partial\lambda$. The delimitation between the letter of index $-1$ and that of index $0$ is called the \textit{median} of the word, marked by a $\mid$ symbol. The size of the Durfee square is then equal to the number of ``$1$'''s of negative index. Hence the application denoted by $\psi$ bijectively associates a partition to the word:
\begin{align*}
\psi(\lambda):=(c_k)_{k\in\mathbb{Z}}=\left(\ldots c_{-2}c_{-1}|c_0c_1c_2\ldots\right), \intertext{where $c_k\in\lbrace 0,1\rbrace$ for any $k\in\bbbz$, and such that}
\#\{k\leq-1,c_k=1\}= \#\{k\geq0,c_k=0\}.
\end{align*}

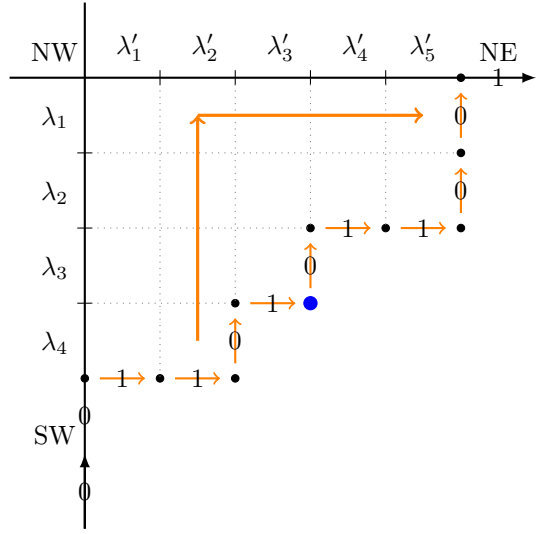
\begin{figure}[h!]
\centering
\begin{tikzpicture}
    [
        dot/.style={circle,draw=black, fill,inner sep=1pt},
    ]

\foreach \x in {0,...,2}{
    \node[dot] at (\x,-4){ };
}

\foreach \x in {.2,1.2}
    \draw[->,thick,orange] (\x,-4) -- (\x+.6,-4);
\foreach \x in {3.2,4.2}
    \draw[->,thick,orange] (\x,-2) -- (\x+.6,-2);
\foreach \y in {1.8,.8}
    \draw[->,thick,orange] (5,-\y) -- (5,-\y+.6);
\draw[->,thick,orange] (2,-3.8) -- (2,-3.8+.6);
\draw[->,thick,orange] (3,-2.8) -- (3,-2.8+.6);
\draw[->,thick,orange] (2.2,-3) -- (2.2+.6,-3);

\node[dot] at (2,-3){};
\foreach \x in {3,...,5}
    \node[dot] at (\x,-2){};
\foreach \x in {1,...,4}
    \draw (\x,-.1) -- node[above,xshift=-0.4cm,yshift=1mm] {$\lambda_\x'$} (\x,+.1);

\node[above,xshift=0.5cm,yshift=1mm] at (4,0) {$\lambda_5'$};
\node[above,xshift=0.5cm,yshift=1mm] at (5,0) {NE};
\node[above,xshift=-4mm,yshift=1mm] at (0,0) {NW};

\foreach \y in {1,...,3}
    \draw (.1,-\y) -- node[above,xshift=-4mm,yshift=0.2cm] {$\lambda_\y$} (-.1,-\y);
\node[above,xshift=-4mm,yshift=2mm] at (0,-4) {$\lambda_4$};
\node[above,xshift=-4mm] at (0,-5) {SW};
\node at (0,-4.5) {0};  
\node at (0,-5.5) {0};
\node at (2,-3.5) {0};
\node at (3,-2.5) {0};
\node at (5,-1.5) {0};
\node at (5,-0.5) {0};

\node at (0.5,-4) {1};
\node at (1.5,-4) {1};   
\node at (4.5,-2) {1};
\node at (3.5,-2) {1};   
\node at (5.5,-0) {1};
\node at (2.5,-3) {1};   

\draw[dotted,gray] (0,-1)--(5,-1);
\draw[dotted,gray] (0,-2)--(3,-2);
\draw[dotted,gray] (0,-3)--(2,-3);

\draw[dotted,gray] (1,0)--(1,-4);
\draw[dotted,gray] (2,0)--(2,-3);
\draw[dotted,gray] (3,0)--(3,-2);
\draw[dotted,gray] (4,0)--(4,-2);

\draw[->,very thick, orange](1.5,-3.5)--(1.5,-0.5);
\draw[->,very thick, orange](1.5,-0.5)--(4.5,-0.5);

\node[dot] at (5,-1){};  
\node[dot] at (5,0){};
\draw[->,thick,-latex] (0,-6) -- (0,-5);
\draw[thick] (0,-6) -- (0,1);
\draw[->,thick,-latex] (-1,0) -- (6,0);
\node[circle,draw=blue,fill=blue,inner sep=0pt,minimum size=5pt] at (3,-3){};

\end{tikzpicture}
\caption{$\partial\lambda$ and its binary correspondence for $\lambda=(5,5,3,2)$ with a hook.}
\label{fig:word}
\end{figure}

In Figure \ref{fig:word}, note that the hook corresponds to the collection of boxes of $\lambda$ as the set of boxes belonging both to $\lambda$ and to the lattice coming from the south of the diagram in parallel to the ordinate axis and exiting in parallel to the $x$-axis. This remark illustrates the following lemma.
\begin{lm}\label{lem:indices}
The application $\psi$ maps bijectively a box $s$ of hook length $h_s$ of the Ferrers diagram of $\lambda$ to a pair of indices $(i_s,j_s)\in\bbbz^2$ of the word $\psi(\lambda)$ such that
\begin{enumerate}
\item $i_s<j_s$,
\item $c_{i_s}=1$, $c_{j_s}=0$,
\item $j_s-i_s=h_s$,
\item $s$ is a box above the main diagonal in the Ferrers diagram of $\lambda$ if and only if the number of letters ``$1$'' with negative index greater than $i_s$ is lower than the number of letters ``$0$'' with nonnegative index lower than $j_s$.
\end{enumerate}
\end{lm}  
\begin{proof}
The first three points are immediate. We now prove the last point. 
Let $s$ be a box in $\lambda$ and $(i_s,j_s)\in\bbbz^2$ the corresponding indices in $\psi(\lambda)=(c_k)_{k\in\bbbz}$ such that $c_{i_s}=1$ and $c_{j_s}=0$. If $i_s$ and $j_s$ share the same sign, this is equivalent to the fact that the hook defined by the sequence $c_{i_s}\dots c_{j_s}$ begins and ends on the same side of the median of $\psi(\lambda)$.
Then the box $s$ associated with this hook is either below the Durfee square or to its right. Hence $s$ is below when $i_s$ and $j_s$ are negative as we also know that $i_s<j_s$, which concludes the proof of this case. Now we assume that $i_s$ is negative and $j_s$ is nonnegative. The number of letters ``$1$'' with negative index greater than $i_s$ corresponds to the number of horizontal steps before the Durfee square and the number of letters ``$0$'' with nonnegative index lower than $j_s$ corresponds to the number of vertical steps after the Durfee square. By definition of the Durfee square, this concludes the proof.
\end{proof}



\begin{rk}\label{rem:conju}
Let $\lambda$ be a partition and $\psi(\lambda)=(c_k)_{k\in\bbbz}$ be its corresponding word. Let $\lambda'$ be the conjugate of $\lambda$ and $\psi(\lambda')=(c'_i)_{i\in\bbbz}$. We have
$$\forall k\in\bbbz, c'_k=1-c_{-k-1}.$$
\end{rk}

Given the properties of symmetries of self-conjugate and doubled distinct partitions, they admit a similar characterization:
\begin{equation}\label{eq:motdd}
\lambda\in\ccdd \iff \psi(\lambda)=(c_k)_{k\in\bbbz} \mid c_0=1 \text{ and } \forall k \in \bbbn^*, c_{-k}=1-c_k,
\end{equation}

and
\begin{equation}\label{eq:motsc}
\lambda\in\ccsc \iff \psi(\lambda)=(c_k)_{k\in\bbbz} \mid  \forall k \in \bbbn, c_{-k-1}=1-c_k.
\end{equation}
From Remark \ref{rem:conju} and \eqref{eq:motdd}, the set $\ccdd'$ also admits a similar characterization:
\begin{equation}\label{eq:motddprime}
\lambda\in\ccdd' \iff \psi(\lambda)=(c_k)_{k\in\bbbz} \mid c_{-1}=0 \text{ and } \forall k \in \bbbn, c_{-k-2}=1-c_k.
\end{equation}

From \eqref{eq:motdd}--\eqref{eq:motddprime} and Lemma \ref{lem:indices}$(4)$, we derive the following result.
\begin{cor}\label{cor:position}
Set $\lambda\in  \ccsc\cup\ccdd\cup\ccdd'$ and $\psi(\lambda)$ its corresponding word. Let $s$ be a box of the Ferrers diagram of $\lambda$. Let $(i_s,j_s)\in\bbbz^2$ be the indices in $\psi(\lambda)$ associated with $s$. Then $s$ is a box on or above the main diagonal in the Ferrers diagram of $\lambda$ if and only if $i_s\geq 0$ or if $i_s<0\leq j_s$ and
$$\begin{cases}
\lvert i_s\rvert\leq \lvert j_s\rvert \text{ if } \lambda\in \ccdd,\\
\lvert i_s+1\rvert\leq \lvert j_s\rvert \text{ if } \lambda\in \ccsc,\\
\lvert i_s+2\rvert\leq \lvert j_s\rvert \text{ if } \lambda\in \ccdd'.
\end{cases}$$
\end{cor}

Now we recall the following classical map, often called the Littlewood decomposition (see for instance \cite{GKS,HJ}).

\begin{df}\label{defphi}{\em
 Let $t \geq 2$ be an integer and consider:\\
$$\begin{array}{l|rcl}
\Phi_t: & \mathcal{P} & \to & \mathcal{P}_{(t)} \times \mathcal{P}^t \\
& \lambda & \mapsto & (\omega,\nu^{(0)},\ldots,\nu^{(t-1)}),
\end{array}$$
where if we set $\psi(\lambda)=\left(c_i\right)_{i\in\bbbz}$, then for all $k\in\lbrace 0,\dots,t-1\rbrace$, one has \linebreak $\nu^{(k)}:=\psi^{-1}\left(\left(c_{ti+m_i+k}\right)_{i\in\bbbz}\right)$, where $m_i=\#\lbrace i\in\N \mid c_{ti+k}=1\rbrace-\#\lbrace i\in\N^* \mid c_{-ti+k}=1\rbrace$. The tuple $\underline{\nu}=\left(\nu^{(0)},\ldots,\nu^{(t-1)}\right)$ is called the $t$-quotient of $\lambda$ and is denoted by \textit{$\quot_t(\lambda)$}, while $\omega$ is the $t$-core of $\lambda$ denoted by \textit{$\core_t(\lambda)$}.}
\end{df}
Obtaining the $t$-quotient is straightforward from $\psi(\lambda)=\left(c_i\right)_{i\in\bbbz}$: we just look at subwords with indices congruent to the same values modulo $t$. The sequence $10$ within these subwords are replaced iteratively by $01$ until the subwords are all the infinite sequence of ``$0$'''s before the infinite sequence of ``$1$'''s (in fact it consists in removing all rim hooks in $\lambda$ of length congruent to $0\pmod t$). Then $\omega$ is the partition corresponding to the word which has the subwords$\pmod t$ obtained after the removal of the $10$ sequences.
For example, if we take $\lambda = (4,4,3,2) \text{ and } t=3$, then $\psi(\lambda)=\ldots \color{red}{0} \color{blue}{0} \color{green}{1} \color{red}{1} \color{blue}{0}\color{green}{1} \color{black}| \color{red}{0} \color{blue}{1} \color{green}{0} \color{red}{0} \color{blue}{1} \color{green}{1}\color{black}\ldots$
\begin{align*}
\begin{array}{rc|rcl}
\psi\left(\nu^{(0)}\right)=\ldots \color{red} 001 \color{black}| \color{red}001\color{black}\ldots& &\psi\left(w_{0}\right)=\ldots \color{red} 000 \color{black}| \color{red}011\color{black}\ldots,\\
 \psi\left(\nu^{(1)}\right)=\ldots \color{blue} 000 \color{black}| \color{blue}111\color{black}\ldots& \longmapsto& \psi\left(w_{1}\right)=\ldots \color{blue} 000 \color{black}| \color{blue}111\color{black}\ldots , \\
 \psi\left(\nu^{(2)}\right)=\ldots \color{green} 011 \color{black}| \color{green}011\color{black}\ldots & & \psi\left(w_{2}\right)=\ldots \color{green} 001 \color{black}| \color{green}111\color{black}\ldots .
\end{array}\\
\end{align*}
Thus
\begin{center}
$\psi(\omega)=\ldots \color{red}{0} \color{blue}{0} \color{green}{0} \color{red}{0} \color{blue}{0}\color{green}{1} \color{black}| \color{red}{0} \color{blue}{1} \color{green}{1} \color{red}{1} \color{blue}{1} \color{green}{1}\color{black}\ldots $
\end{center}
and
$$
\quot_3(\lambda)=\left(\nu^{(0)},\nu^{(1)},\nu^{(2)}\right)=\left((1,1),\emptyset,(2)\right),\ \core_3(\lambda)= \omega=(1).
$$

The following properties of the Littlewood decomposition are given in \cite{HJ}.

\begin{prop}\cite[Theorem 2.1]{HJ}
\label{Littlewood} Let $t$ be a positive integer. The Littlewood decomposition $\Phi_t$ maps bijectively a partition $\lambda$ to $\left(\omega,\nu^{(0)},\dots,\nu^{(t-1)}\right)$ such that:
\begin{align*}
&(P1)\quad \omega \text{ is a $t$-core and }\nu^{(0)},\dots,\nu^{(t-1)} \text{are partitions},\\
&(P2) \quad |\lambda|=|\omega|+t\sum_{i=0}^{t-1} |\nu^{(i)}|,\\
&(P3)\quad \mathcal{H}_t(\lambda)=t\mathcal{H}(\underline{\nu}),
\text{ where for a multiset $S$, } tS:=\{ts,s\in S\}\text{ and } \mathcal{H}(\underline{\nu}):=\bigcup\limits_{i=0}^{t	-1}\cch(\nu^{(i)}).\notag
\end{align*}
\end{prop}

%
\begin{rm}
Proposition \ref{Littlewood} $(P3)$ in particular implies that $\lambda$ is a $t$-core if and only if $\lambda$ has no hook of length $t$. Indeed if there exists $i\in\lbrace 1,\dots,t\rbrace$ such that $\lvert \nu^{(i)}\rvert\neq \varnothing$, then $1$ always belongs to the multiset $\cch(\nu^{(i)})$ and then there always exists a hook of $\lambda$ whose length is equal to $t$.
\end{rm}

Now we discuss the Littlewood decomposition for $\ccsc$. Let $t$ be a positive integer, take $\lambda\in \ccsc$, and set $\psi(\lambda)=(c_i)_{i \in\mathbb{Z}}\in \lbrace 0,1\rbrace^\bbbz$ and $(\omega,\underline{\nu})=\left(\core_t(\lambda),\quot_t(\lambda)\right)$. Using \eqref{eq:motsc}, one has the equivalence  (see for instance \cite{GKS,Pe}):
\begin{eqnarray}
\lambda\in \ccsc &\iff & \forall i_0 \in \lbrace 0,\ldots,t-1\rbrace,\forall j \in \mathbb{N},c_{i_0+jt}=1-c_{-i_0-jt-1}\notag\\
&\iff & \forall i_0 \in \lbrace 0,\ldots,t-1\rbrace,\forall j \in \mathbb{N},c_{i_0+jt}=1-c_{t-(i_0+1)-t(j-1)}\label{mot}\notag\\
&\iff &\forall i_0 \in \left\lbrace 0,\ldots,t-1 \right\rbrace, \nu^{(i_0)}=\left(\nu^{(t-i_0-1)}\right)'\quad \text{and}\quad \omega \in \ccsc_{(t)}\notag .\end{eqnarray}

Therefore $\lambda$ is uniquely defined if its $t$-core is known as well as the $\left\lfloor t/2\right\rfloor$ first elements of its $t$-quotient, which are partitions without any constraints. It implies that if $t$ is even, there is a one-to-one correspondence between a self-conjugate partition and a pair made of one $\ccsc$ $t$-core and $t/2$ generic partitions. If $t$ is odd, the Littlewood decomposition is a one to one correspondence between a self-conjugate partition and a triple made of one $\ccsc$ $t$-core, $(t-1)/2$ generic partitions and a self-conjugate partition $\mu=\nu^{((t-1)/2)}$.
Hence the analogues of the above propositions when applied to self-conjugate partitions are as follows.

\begin{prop}\cite[Lemma 4.7]{MP}\label{SCLittlewood}
Let $t$ be a positive integer. The Littlewood decomposition $\Phi_t$ maps a self-conjugate partition $\lambda$ to $\left(\omega,\nu^{(0)},\dots,\nu^{(t-1)}\right)=(\omega,\underline{\nu})$ such that:
\begin{align*}
&(SC1)\quad \text{the first component } \omega \text{ is a $\ccsc$ $t$-core and }\nu^{(0)},\dots,\nu^{(t-1)} \text{are partitions},\\
&(SC2) \quad \forall j \in \left\lbrace 0,\dots,\left\lfloor t/2 \right\rfloor-1\right\rbrace, \nu^{(j)}=\left(\nu^{(t-1-j)}\right)',\\
&(SC'2)\quad  \text{if t is odd, } \nu^{\left((t-1)/2\right)}=\left(\nu^{\left((t-1)/2\right)}\right)', \\
&(SC3) \quad 
\lvert\lambda\rvert=\begin{cases}\displaystyle\lvert\omega\rvert+2t\sum_{i=0}^{(t-3)/2} \lvert\nu^{(i)}\rvert+t\lvert \mu\rvert \quad \text{if t is odd},
\\
\displaystyle\lvert\omega\rvert +2t\sum_{i=0}^{t/2-1} \lvert\nu^{(i)}\rvert \quad
 \text{if t is even},
\end{cases}\\
&(SC4)\quad \mathcal{H}_t(\lambda)=t\mathcal{H}(\underline{\nu}).
\end{align*}
\end{prop}

Using similar properties as \eqref{eq:motdd} and \eqref{eq:motddprime}, this way of computing the Littlewood decomposition enables to prove the following similar result for $\ccdd$.
\begin{lm}\cite[Lemma 4.3]{MP}\label{lem:DDLittlewood}
Let $t$ be a positive integer. The Littlewood decomposition $\Phi_t$ maps a doubled distinct partition $\lambda$ to $\left(\omega,\nu^{(0)},\dots,\nu^{(t-1)}\right)=(\omega,\underline{\nu})$ such that:
\begin{align*}
&(DD1)\quad \text{the first component } \omega \text{ is a $\ccdd$ $t$-core and }\nu^{(0)},\dots,\nu^{(t-1)} \text{are partitions},\\
&(DD2) \quad \forall j \in \left\lbrace 1,\dots,\left\lfloor t/2 \right\rfloor\right\rbrace, \nu^{(j)}=\left(\nu^{(t-j)}\right)',\nu^{(0)}\in\ccdd,
\\&\quad \quad\quad \quad \text{and if t is even, } \nu^{\left(t/2\right)}=\left(\nu^{\left(t/2\right)}\right)'\in\ccsc,\\
&(DD3) \quad |\lambda|=\begin{cases}\displaystyle|\omega|+2t\sum_{i=1}^{(t-1)/2} \lvert\nu^{(i)}\rvert +t\lvert \nu^{(0)}\rvert\quad
 \text{if t is odd},\\
\displaystyle|\omega|+2t\sum_{i=1}^{t/2-1} \lvert\nu^{(i)}\rvert+t\lvert \nu^{(0)}\rvert+t\lvert \nu^{\left(t/2\right)}\rvert \quad \text{if t is even},
\end{cases}\\
&(DD4)\quad \mathcal{H}_t(\lambda)=t\mathcal{H}(\underline{\nu}).
\end{align*}
\end{lm}

By \eqref{eq:motddprime} and Lemma \ref{lem:DDLittlewood}, one gets the following similar result for partitions in $\ccdd'$.
\begin{lm}\label{lem:DDprimeLittlewood}
Let $t$ be a positive integer. The Littlewood decomposition $\Phi_t$ maps a conjugate doubled distinct partition $\lambda$ to $\left(\omega,\nu^{(0)},\dots,\nu^{(t-1)}\right)=(\omega,\underline{\nu})$ such that:
\begin{align*}
&(DD'1)\quad \text{the first component } \omega \text{ is a $\ccdd'$ $t$-core and }\nu^{(0)},\dots,\nu^{(t-1)} \text{are partitions},\\
&(DD'2) \quad \forall j \in \left\lbrace 0,\dots,\left\lfloor t/2 \right\rfloor-1\right\rbrace, \nu^{(j)}=\left(\nu^{(t-2-j)}\right)',\nu^{(t-1)}\in\ccdd',
\\&\quad \quad \quad \quad \text{and if t even, } \nu^{\left(t/2-1\right)}=\left(\nu^{\left(t/2-1\right)}\right)'\in\ccsc, \\
&(DD'3) \quad |\lambda|=\begin{cases}\displaystyle|\omega|+2t\sum_{i=0}^{(t-3)/2} \lvert\nu^{(i)}\rvert +t\lvert \nu^{(t-1)}\rvert\quad
 \text{if t is odd},\\
\displaystyle|\omega|+2t\sum_{i=0}^{t/2-2} \lvert\nu^{(i)}\rvert+t\lvert \nu^{(t-1)}\rvert+t\lvert \nu^{\left(t/2-1\right)}\rvert \quad \text{if t is even},
\end{cases}\\
&(DD'4)\quad \mathcal{H}_t(\lambda)=t\mathcal{H}(\underline{\nu}).
\end{align*}
\end{lm}

\section{Macdonald identities as multivariate series}\label{sec:Macdointer}

The goal of this section is to make explicit the connection between Macdonald identities seen as multivariate series on the one hand and some quadratic forms appearing to the exponent of a formal variable $T$. This connection will motivate the study in Section \ref{sec:hook}. Stanton in \cite{Stanton} and Rosengren--Schlosser in \cite{RS} prove the Macdonald identities in all these $7$ types by considering \eqref{chap2:macdodenom} as multivariate power series. To illustrate this process, we give here the examples of the right-hand sides of \eqref{chap2:macdodenom} in types $A^{(1)}_{t-1}$ and $C^{(1)}_{t}$. Using the set of positive roots for $A^{(1)}_{t-1}$, the right-hand side of \eqref{chap2:macdodenom} becomes 
\begin{equation}\label{chap2:macdoA}
\prod_{1\leq i<j \leq t} \prod_{k\geq 0} \left(1-e^{-(k+\varepsilon_i-\varepsilon_j)}\right)\left(1-e^{-(k+1+\varepsilon_j-\varepsilon_i)}\right),
\end{equation}
and for $C^{(1)}_{t}$, it is:
\begin{multline}\label{chap2:macdoC}
\prod_{1\leq i \leq t} \prod_{k\geq 0} \left(1-e^{-(k+2\varepsilon_i)}\right)\left(1-e^{-(k+1-2\varepsilon_i)}\right)
\\\times\prod_{1\leq i<j \leq t} \prod_{k\geq 0} \left(1-e^{-(k+\varepsilon_i+\varepsilon_j)}\right)\left(1-e^{-(k+\varepsilon_i-\varepsilon_j)}\right)
 \left(1-e^{-(k+1+\varepsilon_j+\varepsilon_i)}\right)\left(1-e^{-(k+1-\varepsilon_j-\varepsilon_i)}\right).
\end{multline}

Then setting $T=e^{-1}$ and $x_i=T^{-\varepsilon_i}$, the product \eqref{chap2:macdoA} becomes
$$\prod_{1\leq i<j \leq t}\left(x_ix_j^{-1},x_i^{-1}x_jT;T\right)_\infty,$$
and this rewriting yields for \eqref{chap2:macdoC} 
$$\prod_{i=1}^t \left(x_i^2,Tx_i^{-2};T\right)_\infty\prod_{1\leq i<j \leq t}\left(x_ix_j^{-1},x_ix_j,Tx_i^{-1}x_j,Tx_i^{-1}x_j^{-1};T\right)_\infty.$$

In \cite[Corollary $6.2$]{RS}, Macdonald identities are written with $\theta$ functions. We state here this result in a slightly different way writing the $\theta$ functions as infinite products in which we isolate the terms without any power of $T$. Moreover for types $B^{(1)}_{t}, A^{(2)}_{2t-1}$ and $D^{(1)}_{t}$, Macdonald identities are stated as a sum of the sublattice $\lbrace \mathbf{m}\in\bbbz^t\mid \sum_{i=1}^t m_i\equiv 0 \pmod{2}\rbrace$. As remarked in the proof of \cite[Corollary $6.2$]{RS}, the sum can be replaced as half the sum over $\bbbz^t$. The last modification from the statement of \cite[Corollary $6.2$]{RS} is the multiplication by $2$ on both sides of the identity in type $D_t^{(1)}$ so that the term $\Delta_D(\mathbf(x))$ now appears.
\begin{prop}\cite[Corollary $6.2$]{RS}\label{prop:RS}
The Macdonald identities are
\begin{enumerate}
\item For type $A^{(1)}_{t-1}$ with $t\geq 2$
\begin{multline}\label{eq:macdoAstart}
\sum_{\substack{\mathbf{m}\in\bbbz^t\\m_1+\dots+m_t=0\\\sigma\in S_t}}\sgn(\sigma)\prod_{i=1}^tT^{t\binom{m_i}{2}+(\sigma(i)-1)m_i}x_i^{tm_i+\sigma(i)-1}
\\ =\left(T;T\right)_\infty^{t-1}\Delta_A(\mathbf{x}) \prod_{1\le i<j\le t}\left(T(x_ix_j^{-1})^{\pm};T\right)_\infty,
\end{multline}
where $\Delta_A(\mathbf{x})$ is defined in \eqref{chap2:delta} where $n$ is replaced by $t$.
\item For type $B^{(1)}_{t}$ with $t\geq 3$
\begin{multline}\label{eq:macdoBstart}
\frac{1}{2}\sum_{\substack{\mathbf{m}\in\bbbz^t\\\sigma\in S_t}}\sgn(\sigma)\prod_{i=1}^tT^{(2t-1)\binom{m_i}{2}+(t-1)m_i}x_i^{(2t-1)m_i}
\left((T^{m_i}x_i)^{\sigma(i)-t}-(T^{m_i}x_i)^{t+1-\sigma(i)}\right)\\
=\left(T;T\right)_\infty^{t}\Delta_B(\mathbf{x}) K_T(t,\mathbf{x})\prod_{i=1}^t \left(Tx_i^{\pm};T\right)_\infty,
\end{multline}
where $\Delta_B(\mathbf{x})$ is defined in \eqref{def:deltab} where $n$ is replaced by $t$ and $K_T(t,\mathbf{x})$ in \eqref{eq:kt}.
\item For type $A^{(2)}_{2t-1}$ with $t\geq 3$
\begin{multline*}
\frac{1}{2}\sum_{\substack{\mathbf{m}\in\bbbz^t\\\sigma\in S_t}}\sgn(\sigma)\prod_{i=1}^tT^{2t\binom{m_i}{2}+tm_i}x_i^{2tm_i}
\left((T^{m_i}x_i)^{\sigma(i)-t-1}-(T^{m_i}x_i)^{t+1-\sigma(i)}\right)\\
=\left(T^2;T^2\right)_\infty\left(T;T\right)_\infty^{t-1}\Delta_C(\mathbf{x}) K_T(t,\mathbf{x})\prod_{i=1}^t x_i^{-1}\left(T^2x_i^{\pm 2};T^2\right)_\infty,
\end{multline*}
where $\Delta_C(\mathbf{x})$ is defined in \eqref{def:deltac} where $n$ is replaced by $t$.
\item For type $C^{(1)}_{t}$ with $t\geq 2$
\begin{multline}\label{eq:macdostart}
\sum_{\substack{\mathbf{m}\in\bbbz^t\\\sigma\in S_t}}\sgn(\sigma)\prod_{i=1}^tT^{(2t+2)\binom{m_i}{2}+(t+1)m_i}x_i^{(2t+2)m_i}
\left((T^{m_i}x_i)^{\sigma(i)-t-1}-(T^{m_i}x_i)^{t+1-\sigma(i)}\right)\\
=\left(T;T\right)_\infty^{t}\Delta_C(\mathbf{x}) K_T(t,\mathbf{x})\prod_{i=1}^t x_i^{-1}\left(Tx_i^{\pm 2};T\right)_\infty.
\end{multline}

\item For type $D^{(2)}_{t+1}$ with $t\geq 2$
\begin{multline*}
\sum_{\substack{\mathbf{m}\in\bbbz^t\\\sigma\in S_t}}\sgn(\sigma)\prod_{i=1}^tT^{2t\binom{m_i}{2}+(t-1/2)m_i}x_i^{2tm_i}\left((T^{m_i}x_i)^{\sigma(i)-t}-(T^{m_i}x_i)^{t+1-\sigma(i)}\right)\\=
\left(T^{1/2};T^{1/2}\right)_\infty\left(T;T\right)_\infty^{t-1}\Delta_B(\mathbf{x}) K_T(t,\mathbf{x})\prod_{i=1}^t \left(T^{1/2}x_i^{\pm};T^{1/2}\right)_\infty.
\end{multline*}

\item For type $A^{(2)}_{2t}$ with $t\geq 1$
\begin{multline}\label{eq:macdoBCstart}
\sum_{\substack{\mathbf{m}\in\bbbz^t\\\sigma\in S_t}}\sgn(\sigma)\prod_{i=1}^tT^{(2t+1)\binom{m_i}{2}+tm_i}x_i^{(2t+1)m_i}
\left((T^{m_i}x_i)^{\sigma(i)-t}-(T^{m_i}x_i)^{t+1-\sigma(i)}\right)\\=\left(T;T\right)_\infty^{t}\Delta_B(\mathbf{x}) K_T(t,\mathbf{x})\prod_{i=1}^t\left(Tx_i^{\pm};T\right)_\infty\left(Tx_i^{\pm 2};T^2\right)_\infty.
\end{multline}

\item For type $D^{(1)}_{t}$ with $t\geq 4$
\begin{multline*}
\frac{1}{2}\sum_{\substack{\mathbf{m}\in\bbbz^t\\\sigma\in S_t}}\sgn(\sigma)\prod_{i=1}^tT^{(2t-2)\binom{m_i}{2}+(t-1)m_i}x_i^{(2t-2)m_i}
\left((T^{m_i}x_i)^{\sigma(i)-t}+(T^{m_i}x_i)^{t-\sigma(i)}\right),
\\=\left(T;T\right)_\infty^{t}\Delta_D(\mathbf{x}) K_T(t,\mathbf{x}),
\end{multline*}
where $\Delta_D(\mathbf{x})$ is defined in \eqref{def:deltad} where $n$ is replaced by $t$.

\end{enumerate}

\end{prop}

Hence the Macdonald identities can all be stated as sums over sublattices of $\bbbz^t$. To motivate the combinatorics in Section \ref{sec:hook}, we will rewrite the results from Proposition \ref{prop:RS} as sums over sublattices of $\mathbb{Z}^t$ in $\mathbb{C}[x_1,\dots,x_t,x_1^{-1},\dots,x_t^{-1}][[T]]$ where the coefficients of the powers of $T$ can be written as sum of determinants. In this process, a quadratic form will appear to the exponent of $T$, which can be mapped to some subsets of integer partitions whose definitions can be made explicit through the Littlewood decomposition as detailed in Section \ref{chap3:quad}. Finally we divide both sides by the analogues of the Vandermonde determinant for each types, e.g. the terms $\Delta_X$ on the right-hand sides of equalities of Proposition \ref{prop:RS}.
%

\begin{prop}\label{prop:intermed_rewrite}
The Macdonald identities are
\begin{enumerate}
\item For type $A^{(1)}_{t-1}$ with $t\geq 2$
\begin{multline}\label{eq:macdoAfin}
\sum_{\substack{\mathbf{m}\in\bbbz^t\\m_1+\dots+m_t=0}}T^{t\lVert m\rVert^2/2+\sum_{i=1}^t(i-1)m_i}\frac{\underset{1\le i,j\le t}{\det}\left(x_i^{tm_j+j-1}\right)}{\underset{1\le i,j\le t}{\det}\left(x_i^{j-1}\right)} \\
=\left(T;T\right)_\infty^{t-1} \prod_{1\le i<j\le t}\left(T(x_ix_j^{-1})^{\pm};T\right)_\infty.
\end{multline}
\item For type $B^{(1)}_{t}$ with $t\geq 3$, with $K_T(t,\mathbf{x})$ as defined in \eqref{eq:kt}

\begin{multline}
\sum_{\mathbf{m}\in\mathbb{N}^*\times\bbbz^{t-1}}T^{(2t-1)\lVert m\rVert^2/2+\sum_{i=1}^tm_i(i-t-1/2)}\frac{\underset{1\le i,j\le t}{\det}\left(x_i^{(2t-1)m_j+j-t}-x_i^{-((2t-1)m_j+j-t-1)}\right)}{\underset{1\le i,j\le t}{\det}\left(x_i^{j-t}-x_i^{-(j-t)+1}\right)}\\
=\left(T;T\right)_\infty^{t} K_T(t,\mathbf{x})\prod_{i=1}^t \left(Tx_i^{\pm};T\right)_\infty .\label{eq:misignbterm}
\end{multline}

\item For type $A^{(2)}_{2t-1}$ with $t\geq 3$, with $K_T(t,\mathbf{x})$ as defined in \eqref{eq:kt}
\begin{multline}
\sum_{\mathbf{m}\in\mathbb{N}^*\times\bbbz^{t-1}}T^{(2t)\lVert m\rVert^2/2+\sum_{i=1}^tm_i(i-t-1)}\frac{\underset{1\le i,j\le t}{\det}\left(x_i^{2tm_j+j-t-1}-x_i^{-(2tm_j+j-t-1)}\right)}{\underset{1\le i,j\le n}{\det}\left(x_i^{j-t-1}-x_i^{-(j-t-1)}\right)}\\
=\left(T;T\right)_\infty^{t} K_T(t,\mathbf{x})\prod_{i=1}^t \left(Tx_i^{\pm};T\right)_\infty .\label{eq:misignbvterm}
\end{multline}

\item For type $C^{(1)}_{t}$ with $t\geq 2$,  with $K_T(t,\mathbf{x})$ as defined in \eqref{eq:kt}
\begin{multline}
\sum_{\mathbf{m}\in\bbbz^t}T^{(t+1)\lVert m\rVert^2+\sum_{i=1}^tm_i(i-t-1)}\frac{\underset{1\le i,j\le t}{\det}\left(x_i^{(2t+2)m_j+j-t-1}-x_i^{-((2t+2)m_j+j-t-1)}\right)}{\underset{1\le i,j\le t}{\det}(x_i^{j-t-1}-x_i^{-(j-t-1)})}\\
=\left(T;T\right)_\infty^t K_T(t,{\bf x})\prod_{i=1}^t \left(Tx_i^{\pm2};T\right)_\infty .\label{eq:misign}
\end{multline}

\item For type $D^{(2)}_{t+1}$ with $t\geq 2$,  with $K_T(t,\mathbf{x})$ as defined in \eqref{eq:kt}
\begin{multline}
\sum_{\mathbf{m}\in\bbbz^t}T^{t\lVert m\rVert^2+\sum_{i=1}^tm_i(i-t-1/2)}\frac{\underset{1\le i,j\le t}{\det}\left(x_i^{2tm_j+j-t}-x_i^{-(2tm_j+j-t-1)}\right)}{\underset{1\le i,j\le t}{\det}\left(x_i^{j-t}-x_i^{-(j-t)+1}\right)}\\
=\left(T^{1/2};T^{1/2}\right)_\infty\left(T;T\right)_\infty^{t-1}
K_T(t,{\bf x})\prod_{i=1}^t \left(T^{1/2}x_i^{\pm};T^{1/2}\right)_\infty .\label{eq:misigncv}
\end{multline}

\item For type $A^{(2)}_{2t}$ with $t\geq 1$,  with $K_T(t,\mathbf{x})$ as defined in \eqref{eq:kt}
\begin{multline}
\sum_{\mathbf{m}\in\bbbz^t}T^{(2t+1)/2\lVert m\rVert^2+\sum_{i=1}^tm_i(i-t-1/2)}
\frac{\underset{1\le i,j\le t}{\det}\left(x_i^{(2t+1)m_j+j-t}-x_i^{-((2t+1)m_j+j-t-1)}\right)}{\underset{1\le i,j\le t}{\det}\left(x_i^{j-t}-x_i^{-(j-t)+1}\right)}
\\=\left(T;T\right)_\infty^{t} K_T(t,\mathbf{x})\prod_{i=1}^t\left(Tx_i^{\pm};T\right)_\infty\left(Tx_i^{\pm 2};T^2\right)_\infty.\label{eq:misignbc}
\end{multline}


\item For type $D^{(1)}_{t}$ with $t\geq 4$,  with $K_T(t,\mathbf{x})$ as defined in \eqref{eq:kt}
\begin{multline}
\sum_{\substack{\mathbf{m}\in\bbbz^{t}\\m_1>0,m_t\geq 0}}T^{(2t-2)/2\lVert m\rVert^2+\sum_{i=1}^tm_i(i-t)}
\frac{2\underset{1\le i,j\le t}{\det}\left(x_i^{(2t-2)m_j+j-t}+x_i^{-((2t-2)m_j+j-t)}\right)}{(1+\delta_{m_t,0})\underset{1\le i,j\le t}{\det}\left(x_i^{j-t}+x_i^{-(j-t)}\right)}
\\=\left(T;T\right)_\infty^{t} K_T(t,\mathbf{x}).\label{eq:misignd}
\end{multline}

\end{enumerate}

\end{prop}

\begin{proof}

\noindent First we begin the rewriting of the Macdonald identity \eqref{eq:macdoAstart} for type $A^{(1)}_{t-1}$. Extracting  the powers of $T$ on the left-hand side of \eqref{eq:macdoAstart} yields
$$
\sum_{\substack{\mathbf{m}\in\bbbz^t\\m_1+\dots+m_t=0\\\sigma\in S_t}}\sgn(\sigma)T^{t/2 \sum_{i=1}^tm_i^2+\sum_{i=1}^t(\sigma(i)-1-t/2)m_i}\prod_{i=1}^tx_i^{tm_i+\sigma(i)-1}.$$

Substituting $m_i\mapsto m_{\sigma(i)}$ and setting $\lVert m\rVert^2:=\sum_{i=1}^tm_i^2$, this becomes

$$
\sum_{\substack{\mathbf{m}\in\bbbz^t\\m_1+\dots+m_t=0\\\sigma\in S_t}}\sgn(\sigma)T^{t\lVert m\rVert^2/2+\sum_{i=1}^t(i-1-t/2)m_i}\prod_{i=1}^tx_i^{tm_{\sigma(i)}+\sigma(i)-1}.$$
\noindent Hence, using the fact the sum of all $m_i$ is equal to $0$, one can rewrite \eqref{eq:macdoAstart} as \eqref{eq:macdoAfin}.

We proceed similarly for the Macdonald identity for type $C^{(1)}_{t}$. Expand the product on the left-hand side of \eqref{eq:macdostart} and denote by $\xi\in\lbrace \pm 1\rbrace^t$ the sign coming from the expansion of the terms $$\prod_{i=1}^t T^{(2t+2)\binom{m_i}{2}+(t+1)m_i}x_i^{(2t+2)m_i}\left((T^{m_i}x_i)^{\sigma(i)-t-1}-(x_iT^{m_i})^{t+1-\sigma(i)}\right).$$

\noindent Then setting $m_i\to \xi_i m_i$ and finally noting that $m_i^2$ is invariant by this transformation, the left-hand side of \eqref{eq:macdostart} can be rewritten as follows:

\begin{multline*}
\sum_{\sigma\in S_t}\sgn(\sigma) \sum_{\xi\in\lbrace \pm 1\rbrace^t}\prod_{i=1}^t \xi_i \sum_{\mathbf{m}\in\bbbz^t}\prod_{i=1}^t T^{(t+1)m_i^2+m_i \xi_i(\sigma(i)-t-1)} x_i^{(2t+2)m_i+\xi_i(\sigma(i)-t-1)}\\
=\sum_{\sigma\in S_t}\sgn(\sigma) \sum_{\xi\in\lbrace \pm 1\rbrace^t}\prod_{i=1}^t \xi_i\sum_{\mathbf{m}\in\bbbz^t}\prod_{i=1}^t  T^{(t+1)m_i^2+m_i(\sigma(i)-t-1)} x_i^{(2t+2)\xi_im_i+\xi_i(\sigma(i)-t-1)}.
\end{multline*}

\noindent By extracting the powers of the variable $T$, replacing $m_i$ with $m_{\sigma(i)}$ and inverting the sums, we get:
\begin{multline*}
\sum_{\sigma\in S_t}\sgn(\sigma) \sum_{\xi\in\lbrace \pm 1\rbrace^t}\prod_{i=1}^t \xi_i\sum_{\mathbf{m}\in\bbbz^t}T^{(t+1)\lVert \mathbf{m}\rVert^2+\sum_{i=1}^t m_i(i-t-1)}\prod_{i=1}^t   x_i^{(2t+2)\xi_im_{\sigma(i)}+\xi_i(\sigma(i)-t-1)}\\
=\sum_{\mathbf{m}\in\bbbz^t}T^{(t+1)\lVert \mathbf{m}\rVert^2+\sum_{i=1}^tm_i(i-t-1)}\underset{1\le i,j\le t}{\det}\left(x_i^{(2t+2)m_j+j-t-1}-x_i^{-((2t+2)m_j+j-t-1)}\right).
\end{multline*}

\noindent The right-hand side of \eqref{eq:macdostart} is
\begin{equation*}
\left(T;T\right)_\infty^t \Delta_C(\mathbf{x})K_T(t,{\bf x})\prod_{i=1}^t \left(Tx_i^{\pm 2};T\right)_\infty,
\end{equation*}
%
with $\Delta_C(\mathbf{x})$ as defined in \eqref{def:deltac}.
Hence \eqref{eq:macdostart} becomes the desired expression \eqref{eq:misign}.
%

Now for type $A_{2t}^{(2)}$, one can expand the left-hand side of the Macdonald identity \eqref{eq:macdoBCstart} and denote by $\xi\in\lbrace \pm 1\rbrace^t$ the sign coming from the expansion of the terms $$\prod_{i=1}^t T^{(2t+1)\binom{m_i}{2}+t m_i}x_i^{(2t+1)m_i}\left((T^{m_i}x_i)^{\sigma(i)-t}-(x_iT^{m_i})^{t+1-\sigma(i)}\right).$$

\noindent Then setting $m_i\to \xi_i m_i$ and finally noting that $m_i^2$ is invariant by this transformation, the left-hand side of \eqref{eq:macdoBCstart} can be rewritten as follows:
\begin{multline*}
\sum_{\sigma\in S_t}\sgn(\sigma) \sum_{\xi\in\lbrace \pm 1\rbrace^t}\prod_{i=1}^t \xi_i \sum_{\mathbf{m}\in\bbbz^t}\prod_{i=1}^t T^{(t+1/2)m_i^2+m_i \xi_i(\sigma(i)-t-1/2)} x_i^{(2t+1)m_i+\xi_i(\sigma(i)-t+(\xi_i-1)/2)}\\
=\sum_{\sigma\in S_t}\sgn(\sigma) \sum_{\xi\in\lbrace \pm 1\rbrace^t}\prod_{i=1}^t \xi_i\sum_{\mathbf{m}\in\bbbz^t}\prod_{i=1}^t  T^{(t+1/2)m_i^2+m_i(\sigma(i)-t-1/2)} x_i^{(2t+1)\xi_im_i+\xi_i(\sigma(i)-t+(\xi_i-1)/2))}.
\end{multline*}

By doing the same computations, one derives the modified Macdonald identity \eqref{eq:misigncv} for type $D_{t+1}^{(2)}$.

%


\noindent  Similarly for type $B_t^{(1)}$, one can expand the product on the left-hand side of \eqref{eq:macdoBstart} and denote by $\xi\in\lbrace \pm 1\rbrace^t$ the sign coming from the expansion of the terms $$\prod_{i=1}^t T^{(2t-1)\binom{m_i}{2}+(t-1)m_i}x_i^{(2t-1)m_i}\left((T^{m_i}x_i)^{\sigma(i)-t}-(x_iT^{m_i})^{t+1-\sigma(i)}\right).$$

\noindent Then setting $m_i\to \xi_i m_i$ and finally noting that $m_i^2$ is invariant by this transformation, the left-hand side of \eqref{eq:macdoBstart} can be rewritten as follows:
\begin{multline*}
\frac{1}{2}\sum_{\sigma\in S_t}\sgn(\sigma) \sum_{\xi\in\lbrace \pm 1\rbrace^t}\prod_{i=1}^t \xi_i \sum_{\mathbf{m}\in\bbbz^t}\prod_{i=1}^t T^{(t-1/2)m_i^2+m_i \xi_i(\sigma(i)-t-1/2)} x_i^{(2t-1)m_i+\xi_i(\sigma(i)-t+(\xi_i-1)/2)}\\
=\sum_{\sigma\in S_t}\sgn(\sigma) \sum_{\xi\in\lbrace \pm 1\rbrace^t}\prod_{i=1}^t \xi_i\sum_{\mathbf{m}\in\bbbz^t}\prod_{i=1}^t  T^{(t-1/2)m_i^2+m_i(\sigma(i)-t-1/2)} x_i^{(2t-1)\xi_im_i+\xi_i(\sigma(i)-t+(\xi_i-1)/2))}.
\end{multline*}

Hence \eqref{eq:macdoBstart} can be rewritten as
\begin{multline}
\frac{1}{2}\sum_{\mathbf{m}\in\bbbz^t}T^{(2t-1)\lVert m\rVert^2/2+\sum_{i=1}^tm_i(i-t-1/2)}\\
\times\frac{\underset{1\le i,j\le t}{\det}\left(x_i^{(2t-1)m_j+j-t}-x_i^{-((2t-1)m_j+j-t-1)}\right)}{\underset{1\le i,j\le t}{\det}\left(x_i^{j-t}-x_i^{-(j-t)+1}\right)}
=\left(T;T\right)_\infty^{t} K_T(t,\mathbf{x})\prod_{i=1}^t \left(Tx_i^{\pm};T\right)_\infty .\label{eq:misignb}
\end{multline}
We split the sum on the left-hand side of \eqref{eq:misignb} in two sums. The first is over $\mathbf{m}\in\bbbz^t$ such that $m_1\in\bbbn^*$ and the other is over $\mathbf{m}\in\bbbz^t$ such that $m_1\leq 0$. Moreover one can note that by setting $m'_1=1-m_1$ and $m'_i=m_i$:
\begin{multline*}
\sum_{(m_1,\dots,m_t)\in\mathbb{N}_{\leq 0}\times\bbbz^{t-1}}T^{(2t-1)\lVert m\rVert^2/2+\sum_{i=1}^tm_i(i-t-1/2)}\\\times\frac{\underset{1\le i,j\le t}{\det}\left(x_i^{(2t-1)m_j+j-t}-x_i^{-((2t-1)m_j+j-t-1)}\right)}{\underset{1\le i,j\le t}{\det}\left(x_i^{j-t}-x_i^{-(j-t)+1}\right)}
\\=\sum_{(m'_1,\dots,m'_t)\in\mathbb{N}^*\times\bbbz^{t-1}}T^{(2t-1)\lVert m'\rVert^2/2+\sum_{i=1}^tm'_i(i-t-1/2)}
\\\times \frac{\underset{1\le i,j\le t}{\det}\left(x_i^{(2t-1)m'_j+j-t}-x_i^{-((2t-1)m'_j+j-t-1)}\right)}{\underset{1\le i,j\le t}{\det}\left(x_i^{j-t}-x_i^{-(j-t)+1}\right)}.
\end{multline*}

\noindent Hence \eqref{eq:misignb} can be rewritten as \eqref{eq:misignbterm}.

\noindent One derives the rewriting of Macdonald identity in type $A_{2t-1}^{(2)}$ \eqref{eq:misignbvterm} following the same steps.

\noindent To complete the proof, it remains to investigate the case of type $D_t^{(1)}$. It follows the same steps as for type $B_t^{(1)}$ but in this case, the sum is split into $6$ sublattices of $\mathbb{Z}^t$: $\bbbn^*\times \bbbz^{t-2}\times \bbbn^*$, $\bbbn^*\times \bbbz^{t-2}\times \bbbn_{<0}$, $\bbbn_{\leq 0}\times \bbbz^{t-2}\times \bbbn^*$, $\bbbn_{\leq 0}\times \bbbz^{t-2}\times \bbbn_{<0}$ and 
$\bbbn^*\times \bbbz^{t-2}\times \lbrace 0 \rbrace$, $\bbbn_{\leq 0}\times \bbbz^{t-2}\times \lbrace 0 \rbrace$.
By performing the substitutions $m_1\mapsto 1-m_1$ and $m_t\mapsto -m_t$ when $m_1$,respectively $m_t$, is a nonpositive integer, respectively a negative integer, one derives \eqref{eq:misignd}.
\end{proof}

\section{Hook lengths product of $t$-core partitions}\label{sec:hook}
\subsection{Integer partitions and some quadratic forms}\label{chap3:quad}

Thanks to the properties of the Littlewood decomposition, one might see $t$-cores $\omega$ as partitions whose $t$-quotient in the Littlewood decomposition is empty: $\quot_t(\omega)=\left(\emptyset,\dots,\emptyset\right)$. This is equivalent to say that all subwords $\pmod{t}$ in $\psi(\omega)$ are of the form $\ldots 0011\ldots$, which is an infinite sequence of ``$0$'''s followed by an infinite sequence of ``$1$'''s. For any $i\in\lbrace 0,\dots,t-1\rbrace$ let us define $n_i:=\min\lbrace k \in\bbbz\mid c_{i+kt}=1\rbrace$. Each $n_i$ corresponds to the index of the first ``$1$'' in the subword of $\psi(\omega)$ whose index is congruent to $i \pmod t$. Recall from Section \ref{sec:lit} that the word $\psi(\omega)$ has as many ``$1$'''s of negative index as ``$0$'''s of positive index. That condition is equivalent to require $\sum_{i=0}^{t-1}n_i=0$.
Hence there is a natural bijective map 
$$\phi:\ccp_{(t)}\rightarrow\bbbz^t$$
 such that if we set $\phi(\omega):=\mathbf{n}$ then $\sum_{i=0}^{t-1}n_i=0$.

For example, if we take $\omega = (4,2) \text{ and } t=3$, then 
\begin{align*}
\begin{array}{rc|lc}
& &\psi\left(w_{0}\right)=\ldots \color{red} 000 &\color{black}|\color{red}\underbrace{001}_{n_0=2}1\color{black}\ldots,\\
 \psi\left(\omega\right)=\ldots \color{red}{0} \color{blue}{0} \color{green}{0} \color{red}{0} \color{blue}{1}\color{green}{1} \color{black}| \color{red}{0} \color{blue}{1} \color{green}{1} \color{red}{0} \color{blue}{1} \color{green}{1}\color{red}{1}\color{blue}{1} \color{green}{1}\color{black} \ldots& \longmapsto& \psi\left(w_{1}\right)=\ldots \color{blue} 00\underbrace{1}_{n_1=-1}&\color{black}| \color{blue}111\color{black}\ldots , \\
  & & \psi\left(w_{2}\right)=\ldots \color{green} 00\underbrace{1}_{n_2=-1} &\color{black}| \color{green}111\color{black}\ldots .
\end{array}
\end{align*}

Therefore in the above example, $\phi(\omega)=(2,-1,-1)\in\bbbz^3$.

In \cite{Johnson}, Johnson uses the fermionic description of partitions (which is equivalent to the definition of $\psi$) to prove that this bijection is the one defined by Garvan--Kim--Stanton in \cite[Bijection $2$]{GKS}. We reformulate what Johnson wrote in \cite[Section 2]{Johnson} in terms of indices of words. Let $\lambda$ be a partition and $t$ be a positive integer. The abaci correspond exactly to the $t$-subwords of $\psi(\lambda)$ with fixed residue $\pmod{t}$ and the $n_i$'s, as defined in Theorem \ref{thm:gk}, correspond to the charge of the $i$-th runner on the abaci.
\begin{thm}\cite[Theorem $2.10$]{Johnson}\cite[Bijection $2$]{GKS}\label{thm:gk}
Let $\omega$ be a $t$-core and $\psi(\omega)=\left(c_k\right)_{k\in\bbbz}$ be its corresponding word. The bijection $\phi$ satisfies $\phi(\omega)=\left(n_0,\dots,n_{t-1}\right)$ with $n_i=\min\lbrace k\in\bbbz \mid c_{kt+i}=1\rbrace$ and $\sum_{i=0}^{t-1} n_i=0$. Moreover, we have:
\begin{equation}\label{eq:gk}
\lvert \omega\rvert =\frac{t}{2}\sum_{i=0}^{t-1}n_i^2+\sum_{i=0}^{t-1}in_i.
\end{equation}

\end{thm}

As said before, when studying Macdonald identities for the $7$ infinite types of affine root systems, some quadratic forms arise. For instance, as pointed out by Han \cite{Ha} and Dehaye--Han \cite{HD}, the quadratic form \eqref{eq:gk} appears in the Macdonald identity for type $A^{(1)}_{t-1}$ to the exponent of $T$ in \eqref{eq:macdoAfin}. When it comes to the study of other types, different quadratic forms arise. The goal of this section is to introduce the subsets of partitions linked to these quadratic forms. We will see that all of them can be described through the Littlewood decomposition.


First, when studying the Macdonald identity for type $C^{(1)}_{t}$, the quadratic form to the exponent of $T$ in \eqref{eq:misign} corresponds to half the weight of elements of the subset $\ccdd_{(2t+2)}$.
A restriction of the Littlewood decomposition to $\ccdd_{(2t+2)}$ yields some additional conditions on the associated vector of integers already studied by Garvan--Kim--Stanton in \cite{GKS} but stated here in a slightly different way. Set $\omega\in \ccdd_{(2t+2)}$ and $\phi(\omega)=\mathbf{n}\in\bbbz^{2t+2}$, where $\phi$ is the bijection from Theorem \ref{thm:gk}. The equivalence \eqref{eq:motdd} ensures that the following additional conditions are satisfied:

\begin{equation}\label{eq:ndd2tplus2}
n_0=0, \text{ and }\forall i \in \lbrace 1,\dots,2t+1\rbrace,\, n_i=-n_{2t+2-i}.
\end{equation}

\noindent So in particular, the last condition implies $n_{t+1}=0$ and $\omega$ is bijectively associated with a vector of $t$ integers. Therefore we have the following proposition:

\begin{prop}\label{prop:dd2t+2}\cite[Bijection $4$]{GKS}
Let $\omega\in \ccdd_{(2t+2)}$. Set $\phi(\omega)=(n_i)_{0\leq i\leq 2t+1}$. Then the application mapping $\omega\in\ccdd_{(2t+2)}$ to $(n_i)_{1\leq i\leq t}\in\mathbb{Z}^t$ is a bijection. Moreover we have that:
\begin{align}
\lvert \omega \rvert &=(t+1)\sum_{i=0}^{2t+1}n_i^2+\sum_{i=0}^{2t+1}in_i=2\left((t+1)\sum_{i=1}^{t}n_i^2+\sum_{i=1}^{t}(i-t-1)n_i\right).\label{eq:gkddpair}
\end{align}

\end{prop}


As underlined in \cite{Johnson}, Figure \ref{fig:wordddcore} below illustrates how the word interpretation together with the computation of the Littlewood decomposition for a partition in $\ccdd_{(6)}$ when $t=2$, and $2t+2=6$ provide the vector of integers given in \cite{GKS}. The arrows are sorted in $2t+2=6$ different colors, each of them corresponding to a fixed residue $\pmod{6}$ of the index of the corresponding word of $\omega$.
\begin{figure}[h]
\centering
\begin{tikzpicture}
    [
        dot/.style={circle,draw=black, fill,inner sep=1pt},
    ]
\fill [red!30] (0,0) rectangle (11,-1);
\fill [blue!30] (0,-1) rectangle (1,-10);
\foreach \x in {0,...,1}{
    \node[dot] at (\x,-10){ };
}

\draw[->,very thick,blue!100] (0.2,-10) -- (0.2+.6,-10);
\draw[->,very thick,blue!100] (1.2,-5) -- (1.2+.6,-5);
\draw[->,very thick,blue!100] (4.2,-2) -- (4.2+.6,-2);
\draw[->,very thick,blue!100] (9.2,-1) -- (9.2+.6,-1);

\draw[->,very thick,brown] (1,-9.8) -- (1,-9.8+.6);
\draw[->,very thick,brown] (2,-4.8) -- (2,-4.8+.6);
\draw[->,very thick,brown] (5.2,-2) -- (5.2+.6,-2);
\draw[->,very thick,brown] (10.2,-1) -- (10.2+.6,-1);

\draw[->,very thick,blue!30] (1,-8.8) -- (1,-8.8+.6);
\draw[->,very thick,blue!30] (2,-3.8) -- (2,-3.8+.6);
\draw[->,very thick,blue!30] (6,-1.8) -- (6,-1.8+.6);
\draw[->,very thick,blue!30] (11,-0.8) -- (11,-0.8+.6);

\draw[->,very thick,purple] (1,-7.8) -- (1,-7.8+.6);
\draw[->,very thick,purple] (2.2,-3) -- (2.2+.6,-3);
\draw[->,very thick,purple] (6.2,-1) -- (6.2+.6,-1);

\draw[->,very thick,teal] (1,-6.8) -- (1,-6.8+.6);
\foreach \x in {3.2}
    \draw[->,very thick,teal] (\x,-3) -- (\x+.6,-3);
\draw[->,very thick,teal] (7.2,-1) -- (7.2+.6,-1);

\draw[->,very thick,purple!30] (1,-5.8) -- (1,-5.8+.6);
\draw[->,very thick,purple!30] (4,-2.8) -- (4,-2.8+.6);
\draw[->,very thick,purple!30] (8.2,-1) -- (8.2+.6,-1);

\foreach \y in {5,...,9}
	\node[dot] at (1,-\y){};
\node[dot] at (4,-3){};
\foreach \y in {3,...,5}
    \node[dot] at (2,-\y){};

\foreach \x in {4,...,6}
    \node[dot] at (\x,-2){};
\foreach \x in {6,...,11}
    \node[dot] at (\x,-1){};

\foreach \x in {1,...,11}
    \draw (\x,-.1) -- node[above,xshift=-0.4cm,yshift=1mm] {$\lambda_{\x}'$} (\x,+.1);

\node[above,xshift=0.5cm,yshift=1mm] at (12,0) {NE};
\node[above,xshift=-4mm,yshift=1mm] at (0,0) {NW};

\foreach \y in {1,...,10}
    \draw (.1,-\y) -- node[above,xshift=-4mm,yshift=0.2cm] {$\lambda_{\y}$} (-.1,-\y);
\node[below,xshift=-4mm] at (0,-11) {SW};
\node at (-0.5,-11+0.5){0};
\node at (11+0.5,0){1};

\foreach \y in {6,...,10}
	\node at (1,-\y+0.5){0};

\foreach \y in {4,...,5}
    \node at (2,-\y+0.5){0};

\node at (4,-3+0.5){0};
\node at (6,-2+0.5){0};
\node at (11,-0.5){0};

\node at (0.5,-10) {1};
\node at (1.5,-5) {1};

\foreach \x in {2,...,3}
    \node at (\x+0.5,-3){1};
\foreach \x in {4,...,5}
    \node at (\x+0.5,-2){1};
\foreach \x in {6,...,10}
    \node at (\x+0.5,-1){1};
\node at (0.5,-10.25){\textcolor{blue}{$\shortparallel$}};
\node at (0.5,-10.5){\textcolor{blue}{\phantom{o}$c_{-v_1+g}$}};

\node at (3.5,-3.25){\textcolor{teal}{$\shortparallel$}};
\node at (3.5,-3.5){\textcolor{teal}{$c_{0}$}};

\node at (11.75,-0.55){\textcolor{blue!30}{$=c_{v_1-g}$}};

\node[dot] at (6,-1){};  
\draw[->,thick,-latex] (0,-11) -- (0,-11);
\draw[thick] (0,-11) -- (0,1);
\draw[->,thick,-latex] (-1,0) -- (12,0);
\node[circle,draw=blue,fill=blue,inner sep=0pt,minimum size=5pt] at (3,-3){};

\foreach \y in {6,...,9}
	\draw[dotted,gray] (0,-\y)--(1,-\y);
\draw[dotted,gray] (0,-5)--(2,-5);
\draw[dotted,gray] (0,-4)--(3,-4);
\draw[dotted,gray] (0,-3)--(3,-3);
\draw[dotted,gray] (0,-2)--(5,-2);
\draw[dotted,gray] (0,-1)--(6,-1);

\draw[dotted,gray] (1,-6)--(1,0);
\draw[dotted,gray] (2,-5)--(2,0);
\draw[dotted,gray] (3,-3)--(3,0);
\draw[dotted,gray] (4,-3)--(4,0);
\draw[dotted,gray] (5,-2)--(5,0);
\draw[dotted,gray] (6,-1)--(6,0);
\foreach \x in {7,...,10}
	\draw[dotted,gray] (\x,-1)--(\x,0);



\draw[->,very thick, blue](1.5,-4.5)--(1.5,-2.5);
\draw[->,very thick, purple!30](1.5,-2.5)--(3.5,-2.5);

\end{tikzpicture}
\caption{$\omega=(11,6,4,2,2,1,1,1,1,1)\in\ccdd_{(6)}$ and its binary correspondence. Here $g=2t+2=6$ and $v_1$ is defined in Section \ref{chap3:hooks}.}
\label{fig:wordddcore}
\end{figure}
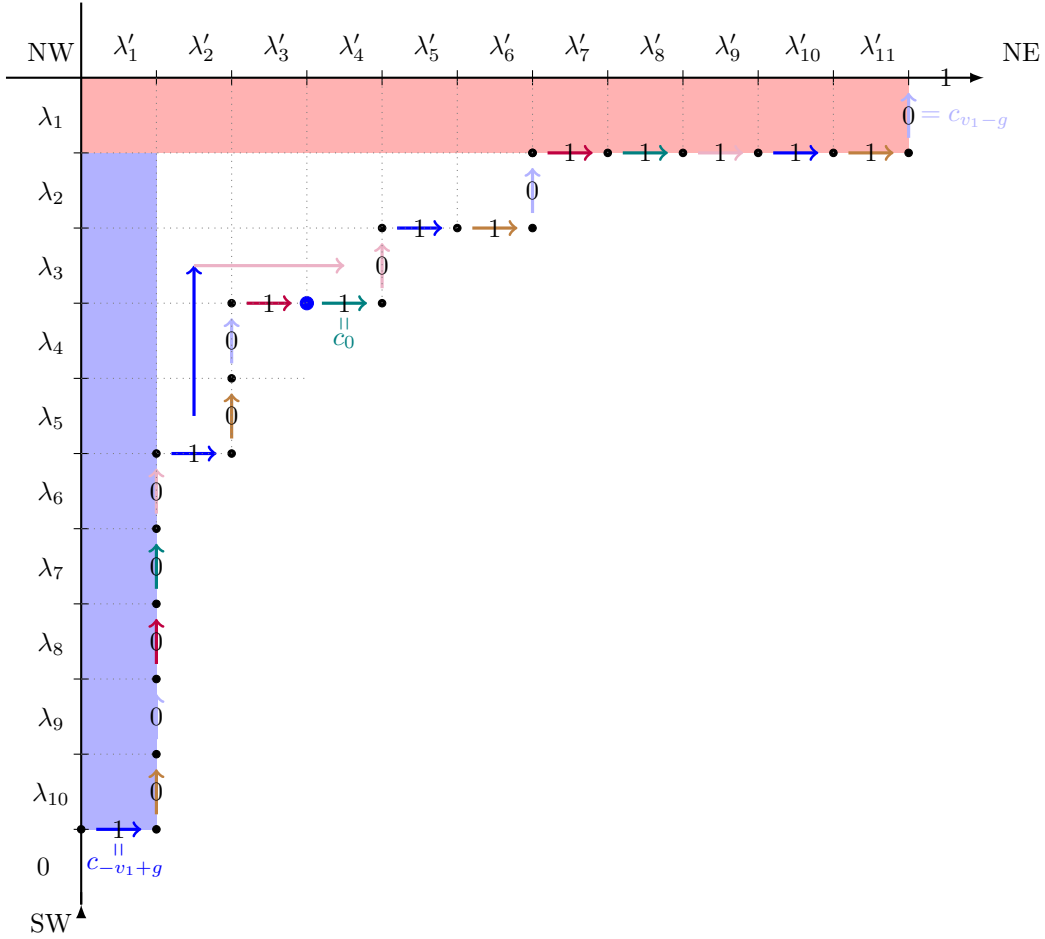
The word corresponding to $\omega$ writes as follows:
\begin{center}
$\psi(\omega)=\cdots \color{blue}{0} \color{brown}{0}  \color{blue!30}{0} \color{purple}{0} \color{teal}{0} \color{purple!30}{0}\color{blue}{1} \color{brown}{0}  \color{blue!30}{0} \color{purple}{0} \color{teal}{0} \color{purple!30}{0} \color{blue}{1} \color{brown}{0}  \color{blue!30}{0} \color{purple}{1} \color{black}|\color{teal}{1} \color{purple!30}{0} \color{blue}{1} \color{brown}{1}  \color{blue!30}{0} \color{purple}{1} \color{teal}{1} \color{purple!30}{1} \color{blue}{1} \color{brown}{1}  \color{blue!30}{0} \color{purple}{1} \color{teal}{1} \color{purple!30}{1}
\color{blue}{1} \color{brown}{1}  \color{blue!30}{1} \color{purple}{1} \color{teal}{1} \color{purple!30}{1}\color{black}\cdots $
\end{center}

By extracting the subwords of fixed residue $\pmod 6$, we obtain:
\begin{align*}
\psi\left(w_{0}\right)=\cdots \color{teal} 000 \color{black}| \color{teal}111\color{black}\cdots\\
\psi\left(w_{1}\right)=\cdots \color{purple!30} 000 \color{black}| \color{purple!30}011\color{black}\cdots\\
\psi\left(w_{2}\right)=\cdots \color{blue} 011 \color{black}| \color{blue}111\color{black}\cdots\\
\psi\left(w_{3}\right)=\cdots \color{brown} 000 \color{black}| \color{brown}111\color{black}\cdots\\
\psi\left(w_{4}\right)=\cdots \color{blue!30} 000 \color{black}| \color{blue!30}001\color{black}\cdots\\
\psi\left(w_{5}\right)=\cdots \color{purple} 001 \color{black}| \color{purple}111\color{black}\cdots
\end{align*}

so that
\begin{center}
$(\color{teal}{n_0},\color{purple!30}{n_1},\color{blue}{n_2},\color{brown}{n_3},\color{blue!30}{n_4},\color{purple}{n_5}\textcolor{black}{)=(}\color{teal}{0},\color{purple!30}{1},\color{blue}{-2},\color{brown}{0},\color{blue!30}{2},\color{purple}{-1} \textcolor{black}{)\in\bbbz^6.}$
\end{center}
The equivalence \eqref{eq:motdd} enables us to keep two components of $\mathbf{n}$:
\begin{center}
$ \omega \in\ccdd_{(6)}\longleftrightarrow (\color{purple!30}{1},\color{blue}{-2}\textcolor{black}{)\in\bbbz^2.}$
\end{center}
%

Similarly the subset $\ccdd_{(2t+1)}$ arises in the Macdonald identity for type $A^{(2)}_{2t}$. Set $\omega\in\ccdd_{(2t+1)}$ and $\phi(\omega)=\mathbf{n}\in\bbbz^{2t+1}$ its image by the bijection of Theorem \ref{thm:gk}. The equivalence \eqref{eq:motdd} yields in this case $n_0=0$ and for all $i\in\lbrace 1,\dots,2t\rbrace$, $n_i=-n_{2t+1-i}$. Once again, the subset of $\ccdd_{(2t+1)}$

\begin{prop}\label{prop:dd2t+1}\cite[Bijection $4$]{GKS}
Let $\omega\in \ccdd_{(2t+1)}$. Set $\phi(\omega)=(n_i)_{0\leq i\leq 2t}$. Then the application mapping $\omega\in\ccdd_{(2t+1)}$ to $(n_i)_{1\leq i\leq t}\in \mathbb{Z}^t$ is a bijection. Moreover we have that:

\begin{align}
\lvert \omega \rvert &=\frac{2t+1}{2}\sum_{i=0}^{2t}n_i^2+\sum_{i=0}^{2t}in_i=(2t+1)\sum_{i=1}^{t}n_i^2+\sum_{i=1}^{t}(2i-2t-1)n_i.\label{eq:gkddimpair}
\end{align}
\end{prop}

The subset of self-conjugate $2t$-cores arises in the Macdonald identity for $D^{(2)}_{t+1}$. A restriction of this previous results to $\ccsc_{(2t)}$ yields some additional conditions on the associated vector of integers. Let $\omega\in \ccsc_{(2t)}$ and $\phi(\omega)=\mathbf{n}\in\bbbz^{2t}$ by the bijection of Theorem \ref{thm:gk}. The equivalence \eqref{eq:motsc} ensures that for all $i \in \lbrace 0,\dots,2t-1\rbrace, n_i=-n_{2t-1-i}$.

\begin{prop}\label{prop:sc2t}
Let $\omega\in \ccsc_{(2t)}$. Set $\phi(\omega)=(n_i)_{0\leq i\leq 2t-1}$. Then the application mapping $\omega\in\ccsc_{(2t)}$ to $(n_i)_{0\leq i\leq t-1}\in\mathbb{Z}^t$ is a bijection. Moreover we have that:

\begin{align}
\lvert \omega \rvert &=\frac{2t}{2}\sum_{i=0}^{2t-1}n_i^2+\sum_{i=0}^{2t-1}in_i=2t\sum_{i=0}^{t-1}n_i^2+\sum_{i=0}^{t-1}(2i-2t+1)n_i.\label{eq:gkscpair}
\end{align}

\end{prop}


%
Regarding types $B^{(1)}_{t}$, $A^{(2)}_{2t-1}$, and $D^{(1)}_{t}$, the leading coefficients of the quadratic forms on vectors of $\bbbz^t$ appearing in the Macdonald identities as rewritten in Proposition \ref{prop:intermed_rewrite} are respectively $2t-1$, $2t$ and $2t-2$. We start by examining the properties of the restriction of the bijection $\phi$ to the set $\ccdd'_{(2t)}$ of partitions whose conjugate is a doubled distinct partition and that are $(2t)$-cores. Set $\omega\in \ccdd'_{(2t)}$ and $\phi(\omega)=\mathbf{n}\in\bbbz^{2t}$. From \eqref{eq:motddprime}, one has
\begin{itemize}
\item $n_{2t-1}=0$,
\item $\forall i \in \lbrace 0,\dots,2t-2\rbrace, n_i=-n_{2t-2-i}$.
\end{itemize}

\noindent So in particular, the last condition implies $n_{t-1}=0$ and $\omega$ is bijectively associated with a vector of $t-1$ integers.

Therefore one has the following characterization:

\begin{prop}\label{prop:ddprime2tcore}
Let $\omega\in \ccdd'_{(2t)}$. Set $\phi(\omega)=(n_i)_{0\leq i\leq 2t-1}$. Then the application mapping $\omega\in\ccdd'_{(2t)}$ to $(n_i)_{0\leq i\leq t-2}\in\mathbb{Z}^{t-1}$ is a bijection. Moreover we have that:
\begin{align}
\lvert \omega \rvert &=\frac{2t}{2}\sum_{i=0}^{2t-1}n_i^2+\sum_{i=0}^{2t-1}in_i=2\left(t\sum_{i=0}^{t-2}n_i^2+\sum_{i=0}^{t-2}(i-t+1)n_i\right).\label{eq:gkddprimepair}
\end{align}

\end{prop}

%

Following the exact same path, one derives the following proposition for $\ccdd'_{(2t-1)}$:
\begin{prop}\label{prop:ddprime2t-1core}
Let $\omega\in \ccdd'_{(2t-1)}$. Set $\phi(\omega)=(n_i)_{0\leq i\leq 2t-2}$. Then the application mapping $\omega\in\ccdd'_{(2t-1)}$ to $(n_i)_{0\leq i\leq t-2}\in\mathbb{Z}^{t-1}$ is a bijection. Moreover we have that:

\begin{align}
\lvert \omega \rvert &=\frac{2t-1}{2}\sum_{i=0}^{2t-3}n_i^2+\sum_{i=0}^{2t-3}in_i=2\left(\left(t-\frac{1}{2}\right)\sum_{i=0}^{t-2}n_i^2+\sum_{i=0}^{t-2}\left(i-t+1+\frac{1}{2}\right)n_i\right).\label{eq:gkddprimeimpair}
\end{align}

\end{prop}

%

 As remarked above, $\phi$ maps $\ccdd'_{(2t)}$ and $\ccdd'_{(2t-1)}$ to $\bbbz^{t-1}$ whereas it maps bijectively $\ccdd'_{(2t-2)}$ to $\bbbz^{t-2}$. Therefore these quadratic forms on vectors of $t$ integers cannot be interpreted as the weight of (conjugate or not) doubled distinct core partitions. Nevertheless they can be interpreted as weights of elements of subsets of partitions with an explicit Littlewood decomposition whose quotient is almost empty but for one or two components.

For instance, in the case of type $B^{(1)}_{t}$, let $m_1$ be a positive integer and let us introduce the set $\ccdd_{2t-1}^{'1}:=\lbrace \lambda\in\ccdd'\mid \Phi_{2t-1}(\lambda)=(\omega,\emptyset,\dots,\emptyset,\nu^{(2t-2)})\rbrace$ where $\nu^{(2t-2)}$ is the rectangular conjugate doubled distinct partition whose parts are all equal to $m_1-1$ and are repeated $m_1$ times.

 Hence $\lvert\nu^{(2t-2)}\rvert=m_1(m_1-1)=(\ell(\left(\nu^{(2t-2)}\right)')+1)\times \ell(\left(\nu^{(2t-2)}\right)')$. Note that its corresponding word is:
\begin{equation}\label{eq:nu0typeb}
\psi(\nu^{(2t-2)})=\dots 0\underbrace{1\dots 1}_{m_1-1}\underbrace{0\mid 0\dots  0}_{m_1}1\dots
\end{equation}

Moreover using \eqref{eq:gkddimpair} and Lemma \ref{lem:DDprimeLittlewood} $(DD'3)$, the subset of $\ccdd^{'1}_{(2t-1)}$ verifies the following characterization:
\begin{prop}\label{prop:ddprime2t-1}
Let $\lambda\in \ccdd^{'1}_{(2t-1)}$. Set $\Phi_{2t-1}(\lambda)=(\omega,\emptyset,\dots,\emptyset,\nu^{(2t-2)})$, $m_1=\ell(\left(\nu^{(2t-2)}\right)')+1$ and $\phi(\omega)=(n_i)_{0\leq i\leq 2t-2}$. Then the application mapping $\lambda\in\ccdd^{'1}_{(2t-1)}$ to $(n_0,\dots, n_{t-2},m_1)\in\mathbb{Z}^{t-1}\times\mathbb{N}^*$ is a bijection. Moreover we have that:

\begin{align}
\lvert \lambda \rvert &=(2t-1)\sum_{i=0}^{t-2}n_i^2+\sum_{i=0}^{t-2}(2(i-t)+3)n_i+(2t-1)(m_1^2-m_1)\notag\\
&=(2t-1)\left(m_1^2+\sum_{i=0}^{t-2}n_i^2\right)-(2t-1)m_1+\sum_{i=0}^{t-2}(2i-2t+3)n_i.\label{eq:gkddnuimpair}
\end{align}
\end{prop}


For type $A^{(2)}_{2t-1}$, let $m_1$ be a positive integer and let us introduce the set $\ccdd_{2t}^{'1}:=\lbrace \lambda\in\ccdd'\mid \Phi_{2t}(\lambda)=(\omega,\emptyset,\dots,\emptyset,\nu^{(2t-1)})\rbrace$ where $\nu^{(2t-1)}$ is the same rectangular partition as in type $B^{(1)}_{t}$ whose corresponding word is \eqref{eq:nu0typeb}. Now if we set $\phi(\omega)=(n_0,\dots,n_{t-2})$, $\lambda\in \ccdd_{2t}^{'1}$ is mapped bijectively to $(n_0,\dots,n_{t-2},m_1)\in\bbbz^{t-1}\times\bbbn^*$. 
Moreover using \eqref{eq:gkddprimepair} and Lemma \ref{lem:DDprimeLittlewood} $(DD'3)$, we get this time:

\begin{prop}\label{prop:ddprime2t}
Let $\lambda\in \ccdd^{'1}_{(2t)}$. Set $\Phi_{2t}(\lambda)=(\omega,\emptyset,\dots,\emptyset,\nu^{(2t-1)})$, $m_1=\ell(\left(\nu^{(2t-1)}\right)')+1$ and $\phi(\omega)=(n_i)_{0\leq i\leq 2t-1}$. Then the application mapping $\lambda\in\ccdd^{'1}_{(2t)}$ to $(n_0,\dots, n_{t-2},m_1)\in\mathbb{Z}^{t-1}\times\mathbb{N}^*$ is a bijection. Moreover we have that:
\begin{align}
\lvert \lambda \rvert &=2t\sum_{i=0}^{t-2}n_i^2+\sum_{i=0}^{t-2}2(i-t+1)n_i+2t(m_1^2-m_1)\notag\\
&=2t\left(m_1^2+\sum_{i=0}^{t-2}n_i^2\right)-2tm_1+2\sum_{i=0}^{t-2}(i-t+3)n_i.\label{eq:gkddnupairbvtquot}
\end{align}
\end{prop}
%

For type $D^{(1)}_{t}$, let $m_1$ be a positive integer and $m_t$ be a nonnegative integer and let us introduce the set $\ccdd_{2t-2}^{'2}:=\lbrace \lambda\in\ccdd'\mid \Phi_{2t-2}(\lambda)=(\omega,\emptyset,\dots,\emptyset,\nu^{(t-2)},\emptyset,\dots,\emptyset,\nu^{(2t-3)})\rbrace$ where $\nu^{(2t-3)}$ is the same rectangular partition as in types $B^{(1)}_{t}$ and $A^{(2)}_{2t-1}$, and $\nu^{(t-2)}$ is the $m_t\times m_t$ square self-conjugate partition. Hence $\lvert\nu^{(2t-3)}\rvert=m_1(m_1-1)$ and $\lvert\nu^{(t-2)}\rvert=m_t^2$. Note that the corresponding word is:
\begin{equation}\label{eq:nu0typed}
\psi(\nu^{(t-2)})=\dots 0\underbrace{1\dots 1}_{m_t}\mid\underbrace{0\dots 0}_{m_t}1\dots
\end{equation}

Moreover using \eqref{eq:gkddprimepair} and Lemma \ref{lem:DDprimeLittlewood} $(DD'3)$, we derive 

\begin{prop}\label{prop:ddprime2t-2}
Let $\lambda\in \ccdd^{'2}_{(2t-2)}$. Set $\Phi_{2t-2}(\lambda)=(\omega,\emptyset,\dots,\emptyset,\nu^{(t-2)},\emptyset,\dots,\emptyset,\nu^{(2t-3)})\rbrace$, $m_1=\ell(\left(\nu^{(2t-3)}\right)')+1$, $m_t=\ell(\nu^{(t-2)})$ and $\phi(\omega)=(n_i)_{0\leq i\leq 2t-3}$. Then the application mapping $\omega\in\ccdd^{'1}_{(2t)}$ to $(n_0,\dots, n_{t-2},m_1,m_t)\in\mathbb{Z}^{t-1}\times\mathbb{N}^*\times \mathbb{N}$ is a bijection. Moreover we have that:

\begin{align}
\lvert \lambda \rvert &=(2t-2)\sum_{i=0}^{t-3}n_i^2+\sum_{i=0}^{t-3}2(i-t+2)n_i+(2t-2)(m_1^2-m_1+m_t^2)\notag\\
&=(2t-2)\left(\sum_{i=0}^{t-3}n_i^2+m_1^2+m_t^2\right)-(2t-2)m_1+\sum_{i=0}^{t-3}2(i-t+2)n_i.\label{eq:gkddnupairtquot}
\end{align}

\end{prop}



\subsection{The enumeration of hook lengths products}\label{chap3:hooks}


The aim of this section is to show Theorems \ref{thm:HD} and \ref{thm:DDpair} and their analogues for the $5$ other families of partitions defined in Section \ref{chap3:quad}. However we will sketch the proof for $\ccp_{(t)}$ corresponding to type $A^{(1)}_{t-1}$ (Theorem \ref{thm:HD}) and give details for the more complicated case of $\ccdd_{(2t+2)}$ corresponding to type $C^{(1)}_{t}$ (Theorem \ref{thm:DDpair}). We will then only highlight the method for the $5$ other families. Proofs are made by induction on the length of the Durfee square of a partition.

A common property of all of the subsets of partitions introduced above is that their elements can be associated with a vector of $t$ integers. Given the properties of symmetry of these partitions, we will bijectively associate them with $t$ \emph{positive} integers through their word interpretations: it consists in the greatest $t$ indices of letters ``$0$'' in the corresponding word such that all of those $t$ indices have a different congruence class modulo an integer $g$ to which we add $g$.
To be more precise, first we introduce the notion of $V_{g,t}$-coding.
\begin{df}\label{def:vcoding}
Let $t$ and $g$ be two positive integers such that $t\leq g$. Set $\lambda \in \ccp$ and $\psi(\lambda)=(c_k)_{k\in\bbbz}$ its corresponding binary word. For $i\in\lbrace 0,\dots,g-1\rbrace$, define $\beta_i:=\max\lbrace (k+1)g+i\mid c_{kg+i}=0\rbrace$. Let $\sigma:\lbrace 1,\dots,g\rbrace\rightarrow \lbrace 0,\dots,g-1\rbrace$ be the unique bijection such that $\beta_{\sigma(1)}>\dots>\beta_{\sigma(g)}$. The vector $\mathbf{v}:=(\beta_{\sigma(1)},\dots,\beta_{\sigma(t)})$ is called the $V_{g,t}$-coding corresponding to $\lambda$.
\end{df}

\begin{rk}\label{vcodingtypeA}
When $g=t$, the map between $\omega\in\ccp_{(t)}$ and its $V_{t,t}$-coding is bijective. Indeed one can deduce $\sigma$ from $v_i\equiv \sigma(i)\pmod{t}$ and $0\in\lbrace \sigma(1),\dots,\sigma(t)\rbrace$. Moreover we have $n_{\overline{\sigma(i)}}=\lfloor v_i/t\rfloor$,  where $\overline{\sigma(i)}$ is the remainder of $\sigma(i)$ by $t$. Therefore
$$\omega=\phi^{-1}\left(\left\lfloor\frac{v_{\sigma^{-1}(1)}}{t}\right\rfloor,\dots, \left\lfloor\frac{v_{\sigma^{-1}(t)}}{t}\right\rfloor\right),$$
where $\phi$ is the bijection from Theorem \ref{thm:gk}.
\end{rk}

For any condition $C$, we will use the boolean notation
$$\mathds{1}_C:=\begin{cases}
1 \text{ if } C\text{ is true,}\\
0 \text{ otherwise.}\end{cases}$$
The following proposition bridges the gap between Definition \ref{def:vcoding} and all of the subsets of partitions introduced in Section \ref{chap3:quad}. Indeed it shows that the $V_{g,t}$-coding of a partition is equivalent to the vector of integers used to characterize these subsets of partitions.
\begin{prop}\label{prop:core_vcoding}
Let $t$ be a positive integer. Any $\lambda$ in one of the sets $\ccp_{(t)}$, $\ccdd_{(2t+2)}$,$\ccdd_{(2t+1)}$, $\ccsc_{(2t)},\ccdd_{2t}^{'1},\ccdd_{2t-1}^{'1}$, and $\ccdd_{2t-2}^{'2}$ is in bijective correspondence with its $V_{g,t}$-coding, where $g$ is the index of the corresponding set.
\end{prop}


\begin{proof}
We will show that the correspondence between the $V_{g,t}$-coding of $\lambda$ and the vector of integers $\phi(\lambda)$ is bijective, and therefore so is the correspondence between $\lambda$ and its $V_{g,t}$-coding.
Let $\lambda$ be a partition in one of the sets of Proposition \ref{prop:core_vcoding}. Set $\psi(\lambda)=\left(c_k\right)_{k\in\bbbz}$ the word corresponding to $\lambda$. Let $\mathbf{v}$ be its $V_{g,t}$-coding. 
The case $\lambda\in\ccp_{(t)}$ is already handled in Remark \ref{vcodingtypeA}.

If $\quot_{g}(\lambda)$ is empty, then $\lambda$ is a $g$-core. By definition $\beta_i-g$ is the last letter ``$0$'' in the subword of $\psi(\lambda)$ whose indices are congruent to $i\pmod{g}$. This implies that for all $i$, $\beta_i=\underset{k\in\bbbz}{\min}\lbrace kg+i\mid c_{kg+i}=1\rbrace$.
 Set $\phi(\lambda)=\mathbf{n}\in\bbbz^g$. Let $\sigma'$ be the unique permutation of $\lbrace 0,\dots,g-1\rbrace$ such that the sequence $(n_{\sigma'(i)},\sigma'(i))_i$ is strictly decreasing with respect to the lexicographic order. By definition of $\sigma'$, note that $\beta_{\sigma'(0)}>\dots>\beta_{\sigma'(g-1)}$. 
By definition of the $V_{g,t}$-coding, we define $\sigma$ from $\sigma'$ as $\sigma(i)=\sigma'(i-1)$ for all $i\in\lbrace 1,\dots,g\rbrace$ and $v_i=\beta_{\sigma(i)}$ for all $i\in\lbrace 1,\dots,t\rbrace$.
Moreover setting
 $$g':=\begin{cases} g-1 \text{ if }\lambda\in\ccsc,\\
g \text{ otherwise,}\end{cases}$$
recall that we have $n_{i-\mathds{1}_{\lambda\in\ccsc}}=-n_{\overline{-i}}$ for all $i\in\lbrace 1,\dots,t\rbrace$ where $\overline{-i}$ is the remainder of $-i$ by $g$. Hence for all $i\in\lbrace 0,\dots, t-1\rbrace$, $\beta_{\sigma'(i)}=-\beta_{\sigma'(g-1-i)}+g'$.
Hence knowing the $V_{g,t}$-coding of $\lambda$ is equivalent to knowing $\beta_{\sigma'(i)}$ for $i\in\lbrace 0,\dots,g-1\rbrace$.
Conversely, let $\mathbf{v}=(v_i)_{i\in\lbrace 1,\dots,t}$ be a $V_{g,t}$-coding.For any $1\leq i\leq t$, if $v_i\equiv k\pmod{g}$, then set $n_k=\lfloor v_i/g\rfloor$. Then define for any $k\in\lbrace v_i\pmod{g},1\leq i\leq t\rbrace$, define $n_{g'-k}=-n_k$. For $j\not\in \lbrace v_i\pmod{g},1\leq i\leq t\rbrace\cup \lbrace g'-v_i\pmod{g},1\leq i\leq t\rbrace$, set $n_j=0$ and $\lambda=\phi_g^{-1}((n_i)_{0\leq i \leq t-1})$.
 Whence the $V_{g,t}$-coding of $\lambda$ is in bijective correspondence with $\lambda$.

Moreover we have that for all $i\in\lbrace 1,\dots,t\rbrace$:
\begin{equation}\label{eq:vcoding}
-v_i+g\leq \frac{g}{2}\leq v_i.
\end{equation}

If $\quot_g(\lambda)$ is not empty, it implies that $\lambda$ belongs to one of the subsets $\ccdd_{2t}^{'1},\ccdd_{2t-1}^{'1}$ or $\ccdd_{2t-2}^{'2}$. We now have for all $i\in\lbrace 0,\dots,g-2\rbrace\setminus\lbrace g/2\rbrace$, $\beta_i=\underset{k\in\bbbz}{\min}\lbrace kg+i\mid c_{kg+i}=1\rbrace$. By Section \ref{chap3:quad}, note that the index congruent to $g-1\pmod{g}$ of the last letter ``$0$'' in $\psi(\lambda)$ is $(m_1-1)g-1$. Hence $\beta_{g-1}\geq g-1$. Similarly if $g$ is even, $\beta_{g/2}\geq g/2$. If $g$ is equal to $2t$ or $2t-1$ and if we set $\omega=\core_g(\lambda)$ and $\phi(\omega)=\mathbf{n'}\in\bbbz^{g}$, let $\sigma'$ be once again the unique permutation of $\lbrace 0,\dots,g-1\rbrace$ such that the sequence $(n'_{\sigma'(i)},\sigma'(i))_i$ is strictly decreasing with respect to the lexicographic order. By the word interpretation of elements in $\ccdd_{g}^{'1}$ given in Section \ref{chap3:quad} and following the exact reasoning as for elements of $\ccdd_{(g)}$, we derive the same inequalities as \eqref{eq:vcoding}. Hence there are exactly $t-1$ elements $k$ in $\lbrace 0,\dots,g-2\rbrace$ such that $\beta_k\geq g/2$ and exactly $t-1$ elements such that $\beta_k\leq g/2$.
Thus there exists $i_0\in\lbrace 0,\dots,t-1\rbrace$ such that $\beta_{\sigma'(i_0)}>\beta_{g-1}>\beta_{\sigma'(i_0-1)}$.
By setting $\sigma(i)=\sigma'(i-1)$ for any $1\leq i\leq i_0$, $\sigma(i_0)=g-1$ and $\sigma(i)=\sigma'(i-2)$ for $i_0+1\leq i\leq t-1$, we obtain $v_i=\beta_{\sigma(i)}$ for all $i\in\lbrace 1,\dots,t\rbrace$. By \eqref{eq:motddprime}, the $V_{g,t}$-coding associated with $\lambda$ is in bijective correspondence with $(n'_0,\dots,n'_{t-2},m_1)$.
Conversely let $\mathbf{v}=(v_i)_{i\in\lbrace 1,\dots,t}$ be a $V_{g,t}$-coding. For any $1\leq i\leq t$, if $v_i\equiv k\pmod{g}$, then set $m_1=\lfloor v_i/g\rfloor$ if $k=g-1$ and $n_k=\lfloor v_i/g\rfloor$ if $k\neq g-1$.
Then define for any $k\in\lbrace v_i\pmod{g},1\leq i\leq t\rbrace$, define $n_{g-k}=-n_k$. For $j\not\in \lbrace v_i\pmod{g},1\leq i\leq t\rbrace\cup \lbrace g-2-v_i\pmod{g},1\leq i\leq t\rbrace$, set $n_j=0$ and $\lambda=\Phi_{g}^{-1}\left(\phi_g^{-1}((n_i)_{0\leq i \leq t-1}),\emptyset,\dots,\emptyset,(m_1-0)^{m_1}\right)$.
 Hence as remarked in Section \ref{chap3:quad}, this is in bijective correspondence with $\lambda$.

The case $g=2t-2$ is exactly the same except that there exists $(k,l)\in\lbrace 1,\dots,t\rbrace^2$ such that $v_k=m_tg+t-2$ and $v_l=m_1g-1$.
\end{proof}
%

For example, the $V_{6,2}$-coding of the partition $\omega=(11,6,4,2,2,1,1,1,1,1)\in\ccdd_{(6)}$ is $(6\times 2+4,6\times 1+1)=(16,7)$ and $\sigma:\lbrace 1,\dots ,6\rbrace\rightarrow \lbrace 0,\dots, 5\rbrace$ such that $\sigma(1)=4, \,\sigma(2)=1,\,\sigma(3)=3,\,\sigma(4)=0,\,\sigma(5)=5$ and $\sigma(6)=2$.

Note that some components of vectors of integers in bijection with the subsets of partitions considered in this paper are always constant. Therefore the $V_{g,t}$-coding associated with those partitions verifies some additional restrictions. The following table summarizes for a given subset considered in this paper the congruences satisfied by the elements of the $V_{g,t}$-coding as well as the $\beta_i$ that are constant for all of the elements in the subset.

\begin{center}
\begin{tabular}{|c|c|c|}
\hline
type & constant $\beta_i$ & congruence set of the $V_{g,t}$-coding $\pmod{g}$  \\
\hline
\rule{0pt}{15pt}$A^{(1)}_{t-1}$ ($\tilde{A}_{t-1}$) &  none &  $\lbrace 0,\dots,t-1\rbrace$\\
\hline
\rule{0pt}{15pt}$B^{(1)}_{t}$ ($\tilde{B}_t$) &  none &  $\lbrace i \text{ or } 2t-3-i, 0\leq i\leq t-2\rbrace\cup\lbrace 2t-2\rbrace$\\
\hline
\rule{0pt}{15pt}$A^{(2)}_{2t-1}$ ($\tilde{B}_t^{\vee}$) &  $\beta_{t-1}=t-1$ &  $\lbrace i \text{ or } 2t-2-i, 0\leq i\leq t-2\rbrace\cup\lbrace 2t-1\rbrace$ \\
\hline
\rule{0pt}{15pt}$C^{(1)}_{t}$ ($\tilde{C}_t$)&  $\beta_{t+1}=t+1$, $\beta_0=0$ &  $\lbrace i \text{ or } 2t+2-i, 1\leq i\leq t\rbrace$\\
\hline
\rule{0pt}{15pt}$D^{(2)}_{t+1}$ ($\tilde{C}_{t}^{\vee}$)&  none &  $\lbrace i \text{ or } 2t-1-i, 0\leq i\leq t-1\rbrace$\\
\hline
\rule{0pt}{15pt}$A^{(2)}_{2t}$ ($\tilde{BC}_{t}$)&  $\beta_0=0$ &   $\lbrace i \text{ or } 2t+1-i, 1\leq i\leq t\rbrace$\\
\hline
\rule{0pt}{15pt}$D^{(1)}_{t}$ ($\tilde{D}_{t}$)&  none &  $\lbrace i \text{ or } 2t-4-i, 0\leq i\leq t-3\rbrace\cup\lbrace t-2,2t-3\rbrace$\\
\hline

\end{tabular}
\captionof{table}{Table of conditions verified by $V_{g,t}$-coding for all types.}\label{tableau_cond}
\end{center}

In the next subsections, we first state the lemmas corresponding to the indices of the boxes in either the largest part or the largest hook. Then we use these lemmas combined with the induction property to derive the theorems of hook lengths products. These hook lengths enumerations and more specifically their proofs are to be applied in the next section, both to rewrite Macdonald identities for all types and also to derive $q$-Nekrasov--Okounkov formulas.  As seen in Section \ref{chap3:quad}, quadratic forms appear in Macdonald identities and are related to our subsets of partitions. Each of the following subsections hence give the product of hook lengths on the subset of partitions defined in Section \ref{chap3:quad}. The case of type $C^{(1)}_{t}$ is fully detailed, whereas the other subsections contain only the main technical results for the other types.

Indeed, to derive the $q$-Nekrasov--Okounkov formula \eqref{Hande} from a specialization of the Macdonald identity \eqref{eq:macdoAstart} for type $A^{(1)}_{t-1}$, Dehaye--Han prove the following theorem, which we rephrase using the definition of $V_{t,t}$-coding (see Definition \ref{def:vcoding}).

\begin{thm}\label{thm:HD}
Set $t$ a positive integer. Let $\omega\in \ccp_{(t)}$ and $\mathbf{v}\in\bbbz^{t}$ its associated $V_{t,t}$-coding, and set $r_i=v_i$ for any $i\in\lbrace1,\dots,t\rbrace$. Then we have
\begin{equation}\label{eq:ppoids}
\lvert \omega\rvert = \frac{1}{2t}\sum_{i=1}^{t} r_i^2-\frac{(t-1)(2t-1)}{12}.
\end{equation}
Moreover, setting $\alpha_i(\omega):=\#\lbrace s\in\omega, h_s=t-i\rbrace$, for any function $\tau:\bbbz\rightarrow F^{\times}$, where $F$ is a field, we also have
\begin{align}\label{eq:thmp}
\prod_{s\in\omega}\frac{\tau(h_s-t)\tau(h_s+t)}{\tau(h_s)^2}=\prod_{i=1}^{t-1}\left(\frac{\tau(-i)}{\tau(i)}\right)^{\alpha_i(\omega)}
\prod_{1\leq i<j\leq t} \frac{\tau(r_i-r_j)}{\tau(j-i)}.
\end{align}

\end{thm}

In order to generalize the proof of Dehaye--Han, we need to prove results that are similar to Theorem \ref{thm:HD} for partitions arising in Section~\ref{chap3:quad}.
The following theorem deals for instance with the partitions in $\ccdd_{(2t+2)}$ appearing in type $C^{(1)}_{t}$.
\begin{thm}\label{thm:DDpair}
Set $t$ a positive integer and $g=2t+2$. Let $\omega\in \ccdd_{(g)}$ and $\mathbf{v}\in\bbbz^{t}$ its associated $V_{g,t}$-coding, and set $r_i=v_i-g/2$ for any $i\in\lbrace1,\dots,t\rbrace$. Then we have
\begin{equation}\label{eq:ddpoids}
\lvert \omega\rvert = \frac{1}{g}\sum_{i=1}^{t} r_i^2-\frac{(g/2-1)(g-1)}{12}.
\end{equation}
Moreover, setting $\alpha_i(\omega):=\#\lbrace s\in\omega, h_s=g-i, \varepsilon_s=1\rbrace$, for any function $\tau:\bbbz\rightarrow F^{\times}$, where $F$ is a field, we also have
\begin{align}\label{eq:thmdd}
\prod_{s\in\omega}\frac{\tau(h_s-\varepsilon_s g)}{\tau(h_s)}=\prod_{i=1}^{g-1}\left(\frac{\tau(-i)}{\tau(i)}\right)^{\alpha_i(\omega)}
\prod_{i=1}^{t}\frac{\tau(r_i)}{\tau(i)}\prod_{1\leq i<j\leq t} \frac{\tau(r_i-r_j)}{\tau(j-i)} \frac{\tau(r_i+r_j)}{\tau(g-i-j)},
\end{align}
and setting $\alpha'_i(\omega):=\#\lbrace s\in\omega, h_s=g-i, \varepsilon_s=-1\rbrace$
\begin{multline}\label{eq:thmddprime}
\prod_{s\in\omega}\frac{\tau(h_s+\varepsilon_s g)}{\tau(h_s)}=\prod_{i=1}^{g-1}\left(\frac{\tau(-i)}{\tau(i)}\right)^{\alpha'_i(\omega)}
\prod_{i=1}^{t}\frac{\tau(2r_i)}{\tau(2i)}\frac{\tau(r_i)\tau(r_i+t+1)\tau(r_i-t-1)}{\tau(i)\tau(i+t+1)\tau(t+1-i)}\\
\times\prod_{1\leq i<j\leq t} \frac{\tau(r_i-r_j)}{\tau(j-i)} \frac{\tau(r_i+r_j)}{\tau(g-i-j)}.
\end{multline}

\end{thm}

\medskip

We need to introduce the \textit{$g$-intervals} as $\I_{m,M}^{g,+}:=\lbrace k\in \bbbz \mid m\leq k<M, k\equiv m \pmod g\rbrace$, $\I_{m,M}^{1,g,+}:=\lbrace (k,M)\in \bbbz \mid m\leq k<M, k\equiv m \pmod g\rbrace$, $\I_{m,M}^{g,-}:=\lbrace l\in \bbbz \mid m<l\leq M, l\equiv M \pmod g\rbrace$ and
$\I_{m,M}^{1,g,-}:=\lbrace (m,l)\in \bbbz \mid m<l\leq M, l\equiv M \pmod g\rbrace$ for $m$, $M$ two integers. 

This notion is of particular interest in our case ever since we are trying to enumerate hook lengths with a fixed residue $\pmod{g}$. We will need the following lemma all along this section.

\begin{lm}\label{lem:telescop}
Let $g$ be a positive integer, let $\tau$ be a function defined over $\bbbz$, let $m$ and $M$ be two integers such that $m<M$. We have
\begin{equation}\label{telescoplus}
\prod_{k\in \I_{m,M}^{g,+}}\frac{\tau(M-k-g)}{\tau(M-k)}=\frac{\tau(M-\max(I_{m,M}^{g,+})-g)}{\tau(M-m)},
\end{equation}
and
\begin{equation}\label{telescopmoins}
\prod_{l\in \I_{m,M}^{g,-}}\frac{\tau(l-m+g)}{\tau(l-m)}=\frac{\tau(M-m+g)}{\tau(\min(I_{m,M}^{g,-})-m)}.
\end{equation}
\end{lm}

The proof is straightforward since all products are telescopic.

The proofs of Theorems \ref{thm:HD} and \ref{thm:DDpair}, as well as those of all theorems present in the next subsections, are done by induction on the maximal element of the $V_{g,t}$-coding of the partitions: it means that, given $\mathbf{v}$ a $V_{g,t}$-coding and $\lambda$ the corresponding partition, $v_1$ is the maximal element. By definition of $\mathbf{v}$, $v_1-g$ corresponds to the maximal index of ``$0$'' in $\psi(\lambda)$. To use the induction property, one needs to consider the partition associated with $\mathbf{v}'$ a $V_{g,t}$-coding such that there exists $i\in\lbrace 1,\dots,t\rbrace$ with $v'_i=v_1-g$.
 For generic $t$-core partitions (i.e. when $g=t$), in order to use the induction property, one needs to consider the $V_{t,t}$-coding corresponding to the partition stripped from its largest part (i.e. stripped from the colored part in Figure \ref{fig:illustrationpartitions}). Nevertheless for all the other subsets of partitions defined in Section \ref{chap3:quad}, given their property of symmetry along $\Delta$,  one needs to consider the partition stripped from its largest hook (i.e. stripped from the colored hook in Figure \ref{fig:wordddcore}) to use the induction property on the maximal element of their associated $V_{g,t}$-coding.

\subsubsection{Type $A^{(1)}_{t-1}$ and $\ccp_{(t)}$}

As mentioned above, the case of $t$-core partitions is a bit particular: it is the only subset of partitions studied in this section with no properties of symmetry along $\Delta$. For instance, when $t=6$, to prove Theorem \ref{thm:HD} in the example of $\omega=(7,3,3,2,2,1,1)\in\ccp_{(6)}$ (see Figure \ref{fig:illustrationpartitions}), the induction property is used on the partition $(3,3,2,2,1)$, which is still a $6$-core partition.

\begin{figure}[H]
\centering
\begin{ytableau}
\none[\omega_1] & *(red!60)13 &*(red!60)10 & *(red!60)7  &*(red!60)5 & *(red!60)4 &*(red!60)3&*(red!60)2 & *(red!60)1&\none[\textcolor{red!60}{\leftarrow}]&\none & \none[\textcolor{red!60}{c_{v_1-6}=0}]\\
\none[\omega_2] & *(blue!30)8 & 5& 2&\none \\
\none[\omega_3] & *(blue!30)7 &4 & 1  &\none \\
\none[\omega_4] & *(blue!30)5 &2 &\none &\none &\none & \none \\
\none[\omega_5] & *(blue!30)4 &1 &\none &\none &\none & \none \\
\none[\omega_6] & *(blue!30)2 &\none &\none &\none & \none \\
\none[\omega_7] & *(blue!30)1 &\none &\none &\none & \none \\
\none & \none[\textcolor{blue!60}{\uparrow}] \\
\none & \none[\textcolor{blue!60}{c_{v_6}=1}]
\end{ytableau}
\caption{$\omega\in \ccp_{(6)}$}
\label{fig:illustrationpartitions}
\end{figure}

\begin{lm}\label{lem:maxpart}
Let $t$ be a positive integer. Set $\omega\in \ccp_{(t)}$ and $(v_1,\dots,v_{t})\in\bbbz^t$ its associated $V_{t,t}$-coding. Then the largest part of $\omega$ corresponds to the collection of boxes of indices in $\cup_{i=2}^{t}\I_{v_i,v_1-t}^{t,+}$. Moreover the largest part of $\omega'$ deprived from the box on the main diagonal corresponds to the collection of boxes of indices in $\I_{v_t,v_1-t}^{t,-}\cup_{i=2}^{t-1}\I_{v_t,v_i-t}^{t,-}$
\end{lm}

The proof is straightforward from the fact that the boxes in $\omega_1$, respectively $\omega_1'$ are those whose index of the letter ``$0$'', respectively ``$1$" is $v_1-g$, respectively $v_t$, and the definition of $g$-intervals. 

Lemma \ref{lem:telescop} is necessary to prove Theorem \ref{thm:HD}, but given the terms of the product over hook lengths slightly differs from the ones in Theorem \ref{thm:DDpair}, one needs also the following results, which are proved as \eqref{telescoplus} and \eqref{telescopmoins} above.
Let $g$ be a positive integer, let $\tau$ be a function defined over $\bbbz$, let $m$ and $M$ be two integers such that $m<M$. We have
\begin{equation*}
\prod_{k\in \I_{m,M}^{g,+}}\frac{\tau(M-k+g)}{\tau(M-k)}=\frac{\tau(M-m+g)}{\tau(M-\max(I_{m,M}^{g,+}))},
\end{equation*}
and
\begin{equation*}
\prod_{l\in \I_{m,M}^{g,-}}\frac{\tau(l-m-g)}{\tau(l-m)}=\frac{\tau(\min(I_{m,M}^{g,-})-m-g)}{\tau(M-m)}.
\end{equation*}

The proof of Theorem \ref{thm:HD} follows from Lemmas \ref{lem:telescop} and \ref{lem:maxpart} and manipulations of products as (but simpler than) in the proof of Theorem \ref{thm:DDpair} detailed below.

\subsubsection{Type $C^{(1)}_{t}$ and $\ccdd_{(2t+2)}$}
 The following lemma is useful in order to decompose the product over hook lengths of doubled distinct $(2t+2)$-core as a product over $g$-intervals.

\begin{lm}\label{lem:maxhook}
Let $t$ be a positive integer and set $g=2t+2$. Set $\lambda\in \ccdd_{(g)}$ and $(v_1,\dots,v_{t})\in\bbbz^t$ its associated $V_{g,t}$-coding. Then the largest hook of $\lambda$, denoted by $H_1$ corresponds to the collection of boxes of indices in

\begin{align*}
H_1:=\cch_{1,+}\cup\cch_{1,-}
\end{align*}
where $\cch_{1,+}$ (respectively $\cch_{1,-}$) denotes the set of indices of boxes $s$ in the first hook such that $\varepsilon_s=1$ (respectively $\varepsilon_s=-1$) and
\begin{align*}
\cch_{1,+}&=\textcolor{red}{\I_{-v_1+g,v_1-g}^{1,g,+}}  \cup \textcolor{teal}{\I_{0,v_1-g}^{1,g,+}} \cup\textcolor{brown}{\I_{g/2,v_1-g}^{1,g,+}}
 \displaystyle\bigcup_{i=2}^{t}\left(\I_{v_i,v_1-g}^{1,g,+}\cup \I_{-v_i+g,v_1-g}^{1,g,+}\right),\\
\cch_{1,-}&= \textcolor{blue!30}{\I_{-v_1+g,v_1-2g}^{1,g,-}} \cup\textcolor{teal}{\I_{-v_1+g,-g}^{1,g,-}}\cup\textcolor{brown}{\I_{-v_1+g, -g/2}^{1,g,-}} \displaystyle\bigcup_{i=2}^{t}\left(\I_{-v_1+g,v_i-g}^{1,g,-}\cup \I_{-v_1+g,-v_i}^{1,g,-}\right).
\end{align*}

\end{lm}

\begin{proof}
Set $\lambda\in \ccdd_{(g)}$, $\psi(\lambda)=\left(c_k\right)_{k\in \bbbz}$, and set $(v_i)_{i\in\lbrace 1,\dots,t\rbrace}$ its associated $V_{g,t}$-coding and $\sigma$ as defined in Definition \ref{def:vcoding}.
By Lemma \ref{lem:indices}, the box with the largest hook length has the corresponding pair of indices $(i_{\min},j_{\max})$ where $i_{\min}:=\min\lbrace i \in \bbbz \mid c_i=1\rbrace$ and $j_{\max}:=\max\lbrace j \in \bbbz \mid c_j=0\rbrace$. Since $\lambda$ is a $g$-core, $v_1=\max\lbrace j \in \bbbz \mid c_{j-g}=0\rbrace$ by Proposition \ref{prop:core_vcoding}. Therefore $j_{\max}=v_1-g$. Moreover by \eqref{eq:motdd}, the maximal index of a letter ``$0$'' corresponds to the minimal index of a letter ``$1$'', whence $i_{\min}=-v_1+g$. Therefore the box of coordinates $(1,1)$ in the Ferrers diagram of $\lambda$ corresponds to word indices $(-v_1+g,v_1-g)$ as shown in the boxes shaded in Figure \ref{fig:wordddcore}.

Moreover, as illustrated in Figure \ref{fig:wordddcore}, each box $s\in\lambda$ is by Lemma \ref{lem:indices} in bijection with a pair of indices $(i_s,j_s)$ such that $i_s \not\equiv j_s \pmod g$, $c_{i_s}=1$, $c_{j_s}=0$, and $i_s<j_s$. By Proposition \ref{prop:core_vcoding} and illustrated in Table \ref{tableau_cond}, it implies that
\begin{align*}
\exists ! k \in\lbrace 1,\dots,t\rbrace\mid i_s\equiv \pm \sigma(k) \pmod g ,\\
\exists ! l \in\lbrace 1,\dots,t\rbrace \setminus\lbrace k\rbrace \mid j_s\equiv \pm \sigma(l) \pmod g .
\end{align*}

By Theorem \ref{thm:gk}, we have
\begin{equation}\label{eq:indices}
\begin{array}{lll}
&i_s \geq &\begin{cases}
v_k \text{ if } i_s\equiv \sigma(k) \pmod g,\\
-v_k+g \text{ if } i_s\equiv -\sigma(k) \pmod g,
\end{cases} 
\\
\text{and}& &
\\
&j_s \leq &\begin{cases}
v_l-g \text{ if } i_s\equiv \sigma(l) \pmod g,\\
-v_l \text{ if } i_s\equiv -\sigma(l) \pmod g.
\end{cases}
\end{array}
\end{equation}

Ultimately, and as illustrated in Figure \ref{fig:wordddcore} in the red and blue shaded areas, let $s$ be a box of $\lambda$ in the first hook. Then $\varepsilon_s=1$ (respectively $\varepsilon_s=-1$) implies $j_s=v_1-g$ (respectively $i_s=-v_1+g$) and $j_s<v_1-g$. 
\end{proof}

From the proof above, note that the largest hook length in the boxes of the largest hook of $\lambda$ is $2(v_1-g)=2(n_{\sigma(1)}-1)g+2\sigma(1)$ and that it corresponds to the hook length of a box on the main diagonal $\Delta$.
From this remark, we derive by induction on the length of $\Delta$
\begin{equation}\label{lem:Delta}
\lbrace h_s\mid s\in\Delta\rbrace=\bigcup_{i=1}^{t}\lbrace 2kg+2\sigma(i)\mid  0\leq k\leq n_{\sigma(i)}-1\rbrace.
\end{equation}

We now prove Theorem \ref{thm:DDpair}.

\begin{proof}[\textbf{Proof of Theorem \ref{thm:DDpair}}]
First we show \eqref{eq:ddpoids}. By Theorem \ref{thm:gk}, the weight of $\omega$ is:
\begin{equation*}
\lvert \omega\rvert=\frac{g}{2}\sum_{i=0}^{2t+1} n_i^2+\sum_{i=0}^{2t+1} in_i=\frac{1}{2g}\left(\sum_{i=0}^{2t+1} (gn_i)^2+\sum_{i=0}^{2t+1} 2gn_i(i-t-1)\right).
\end{equation*}
From \eqref{eq:ndd2tplus2} and the fact $\lbrace r_i\pmod{g}\rbrace=\lbrace 1,\dots,t\rbrace$, we derive \eqref{eq:ddpoids}.

To prove \eqref{eq:thmdd}, we proceed by induction on the size of the Durfee square by removing the hook with largest length (which is the largest hook on the main diagonal). When $\omega=\emptyset$, its corresponding $V_{g,t}$-coding is $\mathbf{v}\in\bbbz^t$ whose components are $v_i=g-i$. Hence the vector $\mathbf{r}$ as defined in Theorem \ref{thm:DDpair} has components $r_i=g-i-t-1=t+1-i$. The left-hand side of \eqref{eq:thmdd} is empty, thus equals $1$ by convention. So is the right-hand side. Hence the initialization.

We can rewrite the left-hand side of \eqref{eq:thmdd} as
\begin{multline}\label{eq:intermed}
\prod_{s\in\omega}\frac{\tau(h_s-\varepsilon_s g)}{\tau(h_s)}=\textcolor{blue!30}{\prod_{\substack{s\in\cch_{1,+}(\omega)\\h_s\in\cch(\omega)}}\frac{\tau(h_s-g)}{\tau(h_s)}}\textcolor{red!30}{\prod_{\substack{s\in\cch_{1,-}(\omega)\\h_s\in\cch(\omega)}}\frac{\tau(h_s+ g)}{\tau(h_s)}}
\prod_{\substack{s\in\omega\setminus H_1\\h_s\in\cch(\omega)}}\frac{\tau(h_s-\varepsilon_s g)}{\tau(h_s)}.
\end{multline}

By Lemma \ref{lem:maxhook}, the first product on the right-hand side of \eqref{eq:intermed} writes as
\begin{multline}\label{eq:telescoprod}
\textcolor{red}{\prod_{x \in\I_{-v_1+g,v_1-g}^{g,+}}\frac{\tau(v_1-g-x-g)}{\tau(v_1-g-x)}}\textcolor{teal}{\prod_{x \in\I_{0,v_1-g}^{g,+}}\frac{\tau(v_1-g-x-g)}{\tau(v_1-g-x)}}\\
\times\textcolor{brown}{\prod_{x \in\I_{g/2,v_1-g}^{g,+}}\frac{\tau(v_1-g-x-g)}{\tau(v_1-g-x)}}\\\times\prod_{i=2}^{t}\prod_{x \in\I_{v_i,v_1-g}^{g,+}}\frac{\tau(v_1-g-x-g)}{\tau(v_1-g-x)}\prod_{x \in\I_{-v_i+g,v_1-g}^{g,+}}\frac{\tau(v_1-g-x-g)}{\tau(v_1-g-x)}.
\end{multline}

By definition of $g$-intervals, if $\I_{m,M}^{g,+}\neq \emptyset$, then 
$$x\in\I_{m,M}^{g,+}\setminus\lbrace m\rbrace\implies x-g\in \I_{m,M}^{g,+}.$$
 Moreover by Lemma \ref{lem:maxhook} there exists a unique box $s$ in the Ferrers diagram of $\omega$ such that $h_s=v_1-g-\max(\I_{m,v_1-g}^{g,+})<g$. Products in \eqref{eq:telescoprod} are therefore telescopic and by Lemma \ref{lem:telescop} it can be rewritten as
%
\begin{multline}\label{eq:last_plus}
\textcolor{red}{\left(\frac{\tau(2\sigma(1)-g(1+\mathds{1}_{\sigma(1)>g/2}))}{\tau(2(v_1-g))}\right)^{\mathds{1}_{g<v_1}}}\textcolor{teal}{\left(\frac{\tau(\sigma(1)-g)}{\tau(v_1-g)}\right)^{\mathds{1}_{g<v_1}}}\\\times\textcolor{brown}{\left(\frac{\tau(\sigma(1)-g/2-\mathds{1}_{\sigma(1)>g/2}g)}{\tau(v_1-\frac{3g}{2})}\right)^{\mathds{1}_{3g/2<v_1}}}
\prod_{i=2}^{t}\left(\frac{\tau(\sigma(1)-\sigma(i)-\mathds{1}_{\sigma(1)>\sigma(i)}g)}{\tau(v_1-v_i-g)}\right)^{\mathds{1}_{g<v_1-v_i}}\\
\times\prod_{i=2}^{t}\left(\frac{\tau(\sigma(1)+\sigma(i)-g(1+\mathds{1}_{\sigma(1)+\sigma(i)>g}))}{\tau(v_1+v_i-2g)}\right)^{\mathds{1}_{2g<v_1+v_i}}.
\end{multline}

To complete the rewriting of the first hook product $H_1$, it remains to proceed the same way for the second product on the right-hand side of \eqref{eq:intermed}. Note that the latter can be written as follows
\begin{multline}\label{eq:telescoprodmoins}
\textcolor{red}{\prod_{y \in\I_{-v_1+g,v_1-2g}^{g,-}}\frac{\tau(v_1+y)}{\tau(v_1-g+y)}}\textcolor{teal}{\prod_{y \in\I_{-v_1+g,-g}^{g,-}}\frac{\tau(v_1+y)}{\tau(v_1-g+y)}}\textcolor{brown}{\prod_{y \in\I_{-v_1+g,-g/2}^{g,-}}\frac{\tau(v_1+y)}{\tau(v_1-g+y)}}\\
\times\prod_{i=2}^{t}\prod_{y \in\I_{-v_1+g,v_i-g}^{g,-}}\frac{\tau(v_1+y)}{\tau(v_1-g+y)}\prod_{y \in\I_{-v_1+g,-v_i}^{g,-}}\frac{\tau(v_1+y)}{\tau(v_1-g+y)}.
\end{multline}

Again the products in \eqref{eq:telescoprodmoins} are telescopic and the same computations as before yield
\begin{multline}\label{eq:last_avant_moins}
\textcolor{red}{\left(\frac{\tau(2(v_1-g))}{\tau(2\sigma(1)-g\mathds{1}_{\sigma(1)>g/2})}\right)^{\mathds{1}_{3g/2<v_1}}}\textcolor{teal}{\left(\frac{\tau(v_1-g)}{\tau(\sigma(1))}\right)^{\mathds{1}_{2g<v_1}}}\\
\times\textcolor{brown}{\left(\frac{\tau(v_1-g/2)}{\tau(\sigma(1)-g/2+\mathds{1}_{\sigma(1)<g/2}g)}\right)^{\mathds{1}_{3g/2<v_1}}}
\prod_{i=2}^{t}\left(\frac{\tau(v_1+v_i-g)}{\tau(\sigma(1)+\sigma(i)-g\mathds{1}_{\sigma(1)+\sigma(i)>g})}\right)^{\mathds{1}_{2g<v_1+v_i}}\\
\times\prod_{i=2}^{t}\left(\frac{\tau(v_1-v_i)}{\tau(\sigma(1)-\sigma(i)+\mathds{1}_{\sigma(1)<\sigma(i)}g)}\right)^{\mathds{1}_{g<v_1-v_i}}.
\end{multline}

First note that $g<v_1<3g/2$ is equivalent to $g<gn_{\sigma(1)}+\sigma(1)<3g/2$. This implies $n_{\sigma(1)}=1$ and $\sigma(1)<g/2$. Hence $2(v_1-g)=2\sigma(1)-g\mathds{1}_{\sigma(1)>g/2}$ and $v_1-g/2=\sigma(1)-g/2+g\mathds{1}_{\sigma(1)<g/2}$. Similarly if $g<v_1<2g$, then $v_1-g=\sigma(1)$. We can then rewrite \eqref{eq:last_avant_moins} as
\begin{multline}\label{eq:last_moins}
\textcolor{red}{\left(\frac{\tau(2(v_1-g))}{\tau(2\sigma(1)-g\mathds{1}_{\sigma(1)>g/2})}\right)^{\mathds{1}_{g<v_1}}}\textcolor{teal}{\left(\frac{\tau(v_1-g)}{\tau(\sigma(1))}\right)^{\mathds{1}_{g<v_1}}}\\
\times\textcolor{brown}{\left(\frac{\tau(v_1-g/2)}{\tau(\sigma(1)-g/2+\mathds{1}_{\sigma(1)<g/2}g)}\right)^{\mathds{1}_{g<v_1}}}\prod_{i=2}^{t}\left(\frac{\tau(v_1+v_i-g)}{\tau(\sigma(1)+\sigma(i)-g\mathds{1}_{\sigma(1)+\sigma(i)>g})}\right)^{\mathds{1}_{2g<v_1+v_i}}\\\times\prod_{i=2}^{t}\left(\frac{\tau(v_1-v_i)}{\tau(\sigma(1)-\sigma(i)+\mathds{1}_{\sigma(1)<\sigma(i)}g)}\right)^{\mathds{1}_{g<v_1-v_i}}.
\end{multline}

Hence using \eqref{eq:last_plus} and \eqref{eq:last_moins}, the red terms $\tau(2(v_1-g))$ cancel and \eqref{eq:intermed} can be rewritten as
\begin{multline}\label{eq:last_eq_ind}
\prod_{s\in\omega}\frac{\tau(h_s-\varepsilon_s g)}{\tau(h_s)}=\textcolor{red}{\left(\frac{\tau(2\sigma(1)-g(1+\mathds{1}_{\sigma(1)>g/2}))}{\tau(2\sigma(1)-g\mathds{1}_{\sigma(1)>g/2})}\right)^{\mathds{1}_{g<v_1}}}
\\\times \textcolor{teal}{\left(\frac{\tau(\sigma(1)-g)}{\tau(\sigma(1))}\right)^{\mathds{1}_{g<v_1}}}
\textcolor{brown}{\left(\frac{\tau(\sigma(1)-g/2-\mathds{1}_{\sigma(1)>g/2}g)}{\tau(\sigma(1)-g/2+\mathds{1}_{\sigma(1)<g/2}g)}\right)^{\mathds{1}_{g<v_1}}}
\\ 
\times\prod_{i=2}^{t}\left(\frac{\tau(\sigma(1)+\sigma(i)-g(1+\mathds{1}_{\sigma(1)+\sigma(i)>g}))}{\tau(\sigma(1)+\sigma(i)-g\mathds{1}_{\sigma(1)+\sigma(i)>g})}\right)^{\mathds{1}_{2g<v_1+v_i}}\\
\times\prod_{i=2}^{t}\left(\frac{\tau(\sigma(1)-\sigma(i)-\mathds{1}_{\sigma(1)>\sigma(i)}g)}{\tau(\sigma(1)-\sigma(i)+\mathds{1}_{\sigma(1)<\sigma(i)}g)}\right)^{\mathds{1}_{g<v_1-v_i}}\\
\times\textcolor{brown}{\left(\frac{\tau(v_1-\frac{g}{2})}{\tau(v_1-\frac{g}{2}-g)}\right)^{\mathds{1}_{g<v_1}}}\prod_{i=2}^{t}\left(\frac{\tau(v_1-v_i)}{\tau(v_1-v_i-g)}\right)^{\mathds{1}_{g<v_1-v_i}}\left(\frac{\tau(v_1+v_i-g)}{\tau(v_1+v_i-2g)}\right)^{\mathds{1}_{2g<v_1+v_i}}
\\ 
\times\prod_{\substack{s\in\omega\setminus H_1\\h_s\in\cch(\omega)}}\frac{\tau(h_s-\varepsilon_s g)}{\tau(h_s)}.
\end{multline}

To complete the induction step, we need to investigate the second largest hook of $\omega$, denoted $H_2$. There are two cases according to the coordinates of the box on the main diagonal of $H_2$ (that of largest hook length), which has coordinates
\begin{enumerate}
\item (first case) either $(-v_1+2g,v_1-2g)$, then for any $i\in\lbrace 2,\dots, t\rbrace$, $v_1-v_i>g$,

\item (second case) or $(-v_2+g,v_2-g)$, then in particular one has $v_1-v_2<g$. To study this case, let $l$ be the integer such that $l:=\max\lbrace i\in\lbrace 2,\dots, t\rbrace \mid v_1-v_i<g\rbrace$. 
\end{enumerate}

\noindent We handle the first case. Let $\omega\setminus H_1$ be the partition $\omega$ deprived from its largest hook. As $\cch(\omega\setminus H_1)\subset \cch(\omega)$, $\omega\setminus H_1$ still belongs to $\ccdd_{(g)}$. In this case, $r'_1=v_1-g-t-1$ and $r'_i=v_i-t-1$ for $2\leq i\leq t$. It also implies by \eqref{eq:vcoding} that $v_1>3g/2$. Moreover for all $2\leq i\leq t$, $v_1-v_i>g$ and $2g<v_1+v_i$. Hence all of the boolean conditions in \eqref{eq:last_eq_ind} are satisfied. 

Now we examine the residual terms depending on $\sigma$ in the product \eqref{eq:last_eq_ind}. Set $(i,j)\in \lbrace 1,\dots, t\rbrace$, and note that
\[\sigma(1)-\sigma(i)+\mathds{1}_{\sigma(1)<\sigma(i)}g= \sigma(1)-\sigma(j)+\mathds{1}_{\sigma(1)<\sigma(j)}g \implies \sigma(i) \equiv\sigma(j)\pmod g.\] 
\noindent By definition of $\sigma$ this cannot be true, so we have $\sigma(1)-\sigma(i)+\mathds{1}_{\sigma(1)<\sigma(i)}g\neq \sigma(1)-\sigma(j)+\mathds{1}_{\sigma(1)<\sigma(j)}g$, 
hence:
\begin{multline}\label{eq:residu}
\lbrace 1,\dots,g-1\rbrace=\lbrace \sigma(1), 2\sigma(1)-g\mathds{1}_{\sigma(1)>g/2},\sigma(1)-g/2+\mathds{1}_{\sigma(1)<g/2}g)\rbrace\\\bigcup_{i=2}^{t}\lbrace \sigma(1)-\sigma(i)+\mathds{1}_{\sigma(1)<\sigma(i)}g, \sigma(1)+\sigma(i)-g\mathds{1}_{\sigma(1)+\sigma(i)>g}\rbrace,
\end{multline}
and
\begin{multline}\label{eq:residuneg}
\lbrace -g+1,\dots,-1\rbrace=\lbrace \sigma(1)-g, 2\sigma(1)-g(1+\mathds{1}_{\sigma(1)>g/2})),\sigma(1)-g/2-\mathds{1}_{\sigma(1)>g/2}g)\rbrace\\\bigcup_{i=2}^{t}\lbrace \sigma(1)-\sigma(i)-\mathds{1}_{\sigma(1)<\sigma(i)}g, \sigma(1)+\sigma(i)-g(1+\mathds{1}_{\sigma(1)+\sigma(i)>g})\rbrace.
\end{multline}

Hence by \eqref{eq:residu} and \eqref{eq:residuneg}, the products depending on $\sigma$ in \eqref{eq:last_eq_ind} can be rewritten as
\begin{equation}\label{eq:prodtau}
\prod_{i=1}^{g-1}\frac{\tau(-i)}{\tau(i)}.
\end{equation}
By the induction property, we have:
\begin{multline}\label{eq:inductionhypfirst}
\prod_{\substack{s\in\omega\setminus H_1\\h_s\in\cch(\omega)}}\frac{\tau(h_s-\varepsilon_s g)}{\tau(h_s)}= \prod_{i=1}^{g-1}\left(\frac{\tau(-i)}{\tau(i)}\right)^{\alpha_i(\omega\setminus H_1)}\\
\times\frac{\tau(v_1-3g/2)}{\tau(1)} \prod_{j=1}^{t} \frac{\tau(v_1-g-v_j)}{\tau(j-1)}
\prod_{2\leq j\leq t}\frac{\tau(v_1+v_j-2g)}{\tau(1+j)}
\\
\times\prod_{i=2}^{t}
\frac{\tau(v_i-g/2)}{\tau(i)}\prod_{2\leq i<j\leq t}  \frac{\tau(v_i+v_j-g)}{\tau(i+j)} \prod_{2\leq i<j\leq t} \frac{\tau(v_i-v_j)}{\tau(j-i)}.
\end{multline}

Plugging \eqref{eq:prodtau} and \eqref{eq:inductionhypfirst} into \eqref{eq:last_eq_ind} concludes the proof of this first case.

\noindent 
The second case is yet more subtle. To handle the second point, let us first remark that $\omega\setminus H_1\in \ccdd_{(g)}$ and let $\mathbf{v'}$ be its associated $V_{g,t}$-coding.
 Let $l$ be the integer such that, if $g<v_1<3g/2$, then $v'_l=g-\sigma(1)=-v_1+2g$ and $\mathbf{v'}=(v_2,\dots,v_l,-v_1+2g,v_{l+1},\dots,v_{t})$ and $\mathbf{v'}=(v_2,\dots,v_l,v_1-g,v_{l+1},\dots,v_{t})$ otherwise.
We start with the case $v_1>3g/2$. 

 By the induction property, we have:
\begin{multline}\label{eq:inductionhyp}
\prod_{\substack{s\in\omega\setminus H_1\\h_s\in\cch(\omega)}}\frac{\tau(h_s-\varepsilon_s g)}{\tau(h_s)}= \prod_{i=1}^{g-1}\left(\frac{\tau(-i)}{\tau(i)}\right)^{\alpha_i(\omega\setminus H_1)}\\
\times\frac{\tau(v_1-3g/2)}{\tau(1)}\prod_{j=2}^l\frac{\tau(v_j+g-v_1)}{\tau(j-1)} \prod_{j=l+1}^{t} \frac{\tau(v_1-g-v_j)}{\tau(j-1)}
\prod_{2\leq j\leq t}\frac{\tau(v_1+v_j-2g)}{\tau(1+j)}
\\
\times\prod_{i=2}^{t}
\frac{\tau(v_i-g/2)}{\tau(i)}\prod_{2\leq i<j\leq t}  \frac{\tau(v_i+v_j-g)}{\tau(g-i-j)} \prod_{2\leq i<j\leq t} \frac{\tau(v_i-v_j)}{\tau(j-i)}.
\end{multline}

Hence \eqref{eq:last_eq_ind} and \eqref{eq:inductionhyp} lead to:
\begin{multline}\label{eq:last_hope}
\prod_{s\in\omega}\frac{\tau(h_s-\varepsilon_s g)}{\tau(h_s)}=\textcolor{red}{\left(\frac{\tau(2\sigma(1)-g(1+\mathds{1}_{\sigma(1)>g/2}))}{\tau(2\sigma(1)-g\mathds{1}_{\sigma(1)>g/2})}\right)^{\mathds{1}_{g<v_1}}}
\\\times \textcolor{teal}{\left(\frac{\tau(\sigma(1)-g)}{\tau(\sigma(1))}\right)^{\mathds{1}_{g<v_1}}}
\textcolor{brown}{\left(\frac{\tau(\sigma(1)-g/2-\mathds{1}_{\sigma(1)>g/2}g)}{\tau(\sigma(1)-g/2+\mathds{1}_{\sigma(1)<g/2}g)}\right)^{\mathds{1}_{g<v_1}}}
\\ \times
\prod_{i=2}^{t}\left(\frac{\tau(\sigma(1)+\sigma(i)-g(1+\mathds{1}_{\sigma(1)+\sigma(i)>g}))}{\tau(\sigma(1)+\sigma(i)-g\mathds{1}_{\sigma(1)+\sigma(i)>g})}\right)^{\mathds{1}_{2g<v_1+v_i}}\\
\times\prod_{i=2}^{t}\left(\frac{\tau(\sigma(1)-\sigma(i)-\mathds{1}_{\sigma(1)>\sigma(i)}g)}{\tau(\sigma(1)-\sigma(i)+\mathds{1}_{\sigma(1)<\sigma(i)}g)}\right)^{\mathds{1}_{g<v_1-v_i}}\\
\times\prod_{j=2}^l\frac{\tau(v_j+g-v_1)}{\tau(v_1-v_j)}\\ \times
\prod_{i=1}^{g-1}\left(\frac{\tau(-i)}{\tau(i)}\right)^{\alpha_i(\omega\setminus H_1)}\prod_{i=1}^{t}\frac{\tau(v_i-g/2)}{\tau(i)}\prod_{1\leq i<j\leq t} \frac{\tau(v_i-v_j)}{\tau(j-i)} \frac{\tau(v_i+v_j-g)}{\tau(i+j)}.
\end{multline}

Moreover we have:
\begin{align*}
\prod_{j=2}^l\frac{\tau(v_j+g-v_1)}{\tau(v_1-v_j)}&=\prod_{j=2}^l\frac{\tau(\sigma(j)+g(1-\mathds{1}_{\sigma(1)<\sigma(j)})-\sigma(1))}{\tau(\sigma(1)-\sigma(j)+\mathds{1}_{\sigma(1)<\sigma(j)}g)}\\
&=\prod_{j=2}^l\frac{\tau(\sigma(j)+\mathds{1}_{\sigma(1)>\sigma(j)}g-\sigma(1))}{\tau(\sigma(1)-\sigma(j)+\mathds{1}_{\sigma(1)<\sigma(j)}g)}.
\end{align*}


Let us introduce
\begin{multline*}
A:=\textcolor{red}{\left(\frac{\tau(2\sigma(1)-g(1+\mathds{1}_{\sigma(1)>g/2}))}{\tau(2\sigma(1)-g\mathds{1}_{\sigma(1)>g/2})}\right)^{\mathds{1}_{g<v_1}}} \textcolor{brown}{\left(\frac{\tau(\sigma(1)-g/2-\mathds{1}_{\sigma(1)>g/2}g)}{\tau(\sigma(1)-g/2+\mathds{1}_{\sigma(1)<g/2}g)}\right)^{\mathds{1}_{g<v_1}}}
\\
\times \textcolor{teal}{\left(\frac{\tau(\sigma(1)-g)}{\tau(\sigma(1))}\right)^{\mathds{1}_{g<v_1}}}
 \prod_{i=2}^{t}\left(\frac{\tau(\sigma(1)+\sigma(i)-g(1+\mathds{1}_{\sigma(1)+\sigma(i)>g}))}{\tau(\sigma(1)+\sigma(i)-g\mathds{1}_{\sigma(1)+\sigma(i)>g})}\right)^{\mathds{1}_{2g<v_1+v_i}}\\
 \times \prod_{i=2}^{t} \left(\frac{\tau(\sigma(1)-\sigma(i)-\mathds{1}_{\sigma(1)>\sigma(i)}g)}{\tau(\sigma(1)-\sigma(i)+\mathds{1}_{\sigma(1)<\sigma(i)}g)}\right)^{\mathds{1}_{g<v_1-v_i}} \prod_{j=2}^l\tau(v_j+g-v_1).
\end{multline*}

By \eqref{eq:residu}, we can rewrite $A$ as
\begin{multline}\label{eq:last_step_induc}
\frac{\prod_{i=1}^{g-1}\tau(-i)}{\prod_{i=1}^{g-1}\tau(i)} \prod_{j=2}^l\frac{\tau(v_1-v_j)}{\tau(v_1-v_j-g)} \prod_{j=2}^l\tau(v_j+g-v_1) \\=\prod_{i=1}^{g-1}\frac{\tau(-i)}{\tau(i)} \prod_{j=2}^l\frac{\tau(-(v_1-v_j-g))}{\tau(v_1-v_j-g)} \prod_{j=2}^l\tau(v_1-v_j)\\
=\prod_{i=1}^{g-1}\left(\frac{\tau(-i)}{\tau(i)}\right)^{\alpha_i(H_1)}\prod_{j=2}^l\tau(v_1-v_j),
\end{multline}
where $\alpha_i(H_1)=\begin{cases} 1 \text{ if } g-i \in H_{1,+},\\
0 \text{ otherwise.}
\end{cases}$

Plugging \eqref{eq:last_hope} in \eqref{eq:last_step_induc} completes the induction.

The case $g<v_1<3g/2$ is settled by the induction property and the remarks used to obtain \eqref{eq:last_eq_ind}.

The proof of \eqref{eq:thmddprime} follows the exact same steps except for the first terms of each product of the right-hand sides of \eqref{eq:last_avant_moins} and \eqref{eq:last_plus} where the products in red and in green are now:
\begin{multline}\label{eq:interprime}
\textcolor{red}{\left(\frac{\tau(2\sigma(1)-g(1+\mathds{1}_{\sigma(1)>g/2}))}{\tau(2(v_1-g)-g)}\right)^{\mathds{1}_{3/2g<v_1}}}
 \textcolor{red}{\left(\frac{\tau(2(v_1-g)+g)}{\tau(2\sigma(1)-g\mathds{1}_{\sigma(1)>g/2})}\right)^{\mathds{1}_{g<v_1}}}\\
 \times  \textcolor{teal}{\left(\frac{\tau(v_1)}{\tau(\sigma(1))}\right)^{\mathds{1}_{g<v_1}}}\textcolor{teal}{\left(\frac{\tau(\sigma(1)-g)}{\tau(v_1-2g)}\right)^{\mathds{1}_{2g<v_1}}}.
\end{multline}

To handle the boolean conditions, we use once again that if $g<v_1<3g/2$ then $2\sigma(1)-g(1+\mathds{1}_{\sigma(1)>g/2})=2v_1-3g$ and if $g<v_1<2g$ then $\sigma(1)-g=v_1-2g$ which concludes the proof.
\end{proof}

\subsubsection{Type $B^{(1)}_{t}$ and $\ccdd_{2t-1}^{'1}$}

As mentioned in Section \ref{chap3:quad}, interpreting the geometric conditions on lattices in Macdonald identities for types $B^{(1)}$, $A^{(2)}_{odd}$ and $D^{(1)}$ as $g$-cores of $\ccdd$ or $\ccsc$ is not possible but these can be interpreted as a subset of doubled distinct partitions with a specific $g$-quotient. In this subsection, we set $g=2t-1$.

If we take $\lambda\in\ccdd_{g}^{'1}$ such that $\nu^{(2t-2)}$ has its corresponding word equal to \eqref{eq:nu0typeb}, and $\psi(\lambda)=(c_k)_{k\in\bbbz}$, note that the set of indices congruent to $2t-2\pmod{g}$ corresponding to a letter ``$0$'' is $\lbrace kg-1 \mid k\in\bbbz, k\leq -m_1 \text{ or } 0\leq k\leq m_1-1\rbrace$. Using \eqref{eq:motddprime} and the same considerations as for the proof of Lemma \ref{lem:maxhook}, we have the following similar result.
\begin{lm}\label{lem:maxhookddprimeimpair}
Let $t$ be a positive integer and set $g=2t-1$. Set $\lambda\in \ccdd_{g}^{'1}$ and $\mathbf{v}\in\bbbz^{t}$ its associated $V_{g,t}$-coding. Let $i_0\in\lbrace 1,\dots,t\rbrace$ be such that $v_{i_0}=m_1g-1$.
Then the largest hook of $\lambda$, denoted by $H_1$, corresponds to the collection of boxes of indices 
\begin{align*}
H_1=\cch_{1,+}\cup\cch_{1,-}
\end{align*}
where $\cch_{1,+}$ (respectively $\cch_{1,-}$) denotes the set of indices of boxes $s$ in the first hook such that $\varepsilon_s=1$ (respectively $\varepsilon_s=-1$) and
\begin{multline*}
\cch_{1,+}= \displaystyle\bigcup_{\substack{i=1\\i\neq i_0}}^{t}\left(\I_{v_i,v_1-g}^{1,g,+}\cup \I_{-v_i+g-2,v_1-g}^{1,g,+}\right)\cup\textcolor{teal}{\I_{v_{i_0},v_1-g}^{1,g,+}\cup \left(\I_{-v_{i_0}+g-2,v_1-g}^{g,+}\setminus \I_{-1,v_1-g}^{1,g,+}\right)} ,
 \end{multline*}
 \begin{multline*}
\cch_{1,-}= \displaystyle\bigcup_{\substack{i=1\\i\neq i_0}}^{t}\left(\I_{-v_1+g-2,v_i-(1+\mathds{1}_{i=1})g}^{1,g,-}\cup \I_{-v_1+g-2,-v_i-2}^{1,g,-}\right)\cup\\ \textcolor{teal}{\I_{-v_1+g-2,-v_{i_0}-2}^{1,g,-}\cup\left(\I_{-v_1+g-2,v_{i_0}-(1+\mathds{1}_{i_0=1})g}^{1,g,-}\setminus \I_{-v_1+g-2,-g-1}^{1,g,-}\right)} .
\end{multline*}
\end{lm}

Note that we avoid the discussion on whether $i_0=1$ by remarking that $\I_{v_1,v_1-g}^{g,+}=\emptyset$.

As for \eqref{lem:Delta}, we derive similarly that for any partition in $\ccdd_{2t}^{'1}, \ccdd_{2t-1}^{'1}$ or $\ccdd_{2t-2}^{'2}$ where $g$ is the index of the corresponding set, we have:
\begin{equation}\label{lem:Deltaddprime}
\lbrace h_s\mid s\in\Delta\rbrace=\bigcup_{i=1}^{t}\lbrace 2kg+2(\sigma(i)+1)\mid  0\leq k\leq n_{\sigma(i)}-1\rbrace.
\end{equation}

The issues arising when using \eqref{eq:thmdd} is that if the $g$-quotient of a partition is not empty, we have at least one box $s$ such that $h_s=g$. As when we derive $q$-Nekrasov--Okounkov type formulas, we specialize $\tau(x)=1-q^x$ or $\tau=\Id$, the product can be equal to $0$. Therefore one needs the following result for type $B^{(1)}_{t}$.

\begin{thm}\label{thm:DDtquotimpair}
Set $t$ a positive integer and $g=2t-1$. Let $\lambda\in \ccdd_{g}^{'1}$ and $\mathbf{v}\in\bbbz^{t}$ the $V_{g,t}$-coding associated to $\lambda$, and set $r_i=v_i-g/2+1$ for any $i\in\lbrace1,\dots,t\rbrace$. Then we have
$$\lvert \lambda\rvert = \frac{1}{g}\sum_{i=1}^{t} r_i^2-\frac{(g+2)(\lfloor g/2\rfloor +1)}{12}.$$

Moreover setting $\alpha_i(\lambda)$ as in Theorem \ref{thm:DDpair}, for any function $\tau:\bbbz\rightarrow F^{\times}$, where $F$ is a field, we also have
\begin{equation}\label{eq:thmddprimepairtimpairquot}
\prod_{s\in\lambda}\frac{\tau(h_s-\varepsilon_s g)}{\tau(h_s)}=\prod_{i=1}^{g-1}\left(\frac{\tau(-i)}{\tau(i)}\right)^{\alpha_i(\lambda)}
\prod_{1\leq i<j\leq t} \frac{\tau(r_i-r_j)}{\tau(j-i)} \frac{\tau(r_i+r_j)}{\tau(g+2-i-j)}.
\end{equation}

\end{thm}

\begin{proof}
The proof is exactly the same as the one of Theorem \ref{thm:DDpair} but without the products in brown in \eqref{eq:telescoprod} and \eqref{eq:telescoprodmoins}. These products would have been over the sets $\I_{g/2-1,v_1-g}^{g,+}$  and $\I_{-v_1+g-2,-v_{i_0}-2}^{g,-}$ in this case. Note that the latter products are the ones which yield the terms $\tau(v_i-g/2)/\tau(i)$ in \eqref{eq:thmdd}. Hence there are no terms $\tau(r_i)/\tau(i)$ in \eqref{eq:thmddprimepairtimpairquot}. The other difference with \eqref{eq:thmdd} is that $\Phi_{2t-1}(\lambda)=(\omega,\emptyset,\dots,\emptyset,\nu^{(2t-2)})\rbrace$ where $\nu^{(2t-2)}$ has a corresponding word verifying \eqref{eq:nu0typeb}. Then there exists $i_0\in\lbrace 1,\dots,t\rbrace$ such that $\sigma(i_0)=g-1$ and $v_{i_0}=m_1g-1$. The product of hook lengths belonging to $\cch_{2t-1}(\lambda)$ corresponding to pairs of indices $(i,j)$ such that $i<j$ and $i\equiv j \pmod{g}\equiv g-1\pmod{g}$ are telescopic as well as the other products over hook lengths corresponding to pairs of indices $(i,j)$ with $i<j$ such that $i$ or $j$ is congruent to $g-1$ modulo $g$. Using Lemma \ref{lem:maxhookddprimeimpair}, the products give 
\begin{equation*}
\frac{\tau(2v_{i_0}-2g+2)}{\tau(g)}\frac{\tau(g)}{\tau(2v_{i_0}-2g+2)}=1.
\end{equation*}
If $i_0>1$, the products in green in \eqref{eq:interprime} are now
\begin{multline*}
\textcolor{teal}{\left(\frac{\tau(v_1+v_{i_0}-g+2)}{\tau(v_1+1-g)}\right)^{\mathds{1}_{g<v_1}}\left(\frac{\tau(v_1-v_{i_0})}{\tau(\sigma(1)+1)}\right)^{\mathds{1}_{g<v_1-v_{i_0}}}}\\
\times\textcolor{teal}{\left(\frac{\tau(\sigma(1)+1-g)}{\tau(v_1-v_{i_0}-g)}\right)^{\mathds{1}_{g<v_1-v_{i_0}}}\left(\frac{\tau(v_i+1-g)}{\tau(v_1+v_{i_0}-2g+2)}\right)^{\mathds{1}_{0<v_1}}}.
\end{multline*}
The case $i_0=1$ is handled in the same way and the proof is completed with the same computations as in the previous theorems.
\end{proof}

\subsubsection{Type $A^{(2)}_{2t-1}$ and $\ccdd_{2t}^{'1}$}
If we take $\lambda\in\ccdd_{2t}^{'1}$ such that $\nu^{(2t-1)}$ has its corresponding word equal to \eqref{eq:nu0typeb}, and $\psi(\lambda)=(c_k)_{k\in\bbbz}$, note that the set of indices congruent to $2t-1\pmod{g}$ corresponding to a letter ``$0$'' is $\lbrace kg-1 \mid k\in\bbbz, k\leq -m_1 \text{ or } 0\leq k\leq m_1-1\rbrace$. Using \eqref{eq:motddprime} and the same considerations as for the proof of Lemma \ref{lem:maxhook}, we have the following similar result.
\begin{lm}\label{lem:maxhookddprime}
Let $t$ be a positive integer and set $g=2t$. Set $\lambda\in \ccdd_{g}^{'1}$ and $\mathbf{v}\in\bbbz^{t}$ its associated $V_{g,t}$-coding. Let $i_0\in\lbrace 1,\dots,t\rbrace$ be such that $v_{i_0}=m_1g-1$.
Then the largest hook of $\lambda$, denoted by $H_1$, corresponds to the collection of boxes of indices 
\begin{align*}
H_1=\cch_{1,+}\cup\cch_{1,-}
\end{align*}
where $\cch_{1,+}$ (respectively $\cch_{1,-}$) denotes the set of indices of boxes $s$ in the first hook such that $\varepsilon_s=1$ (respectively $\varepsilon_s=-1$) and
\begin{multline*}
\cch_{1,+}= \textcolor{teal}{\I_{v_{i_0},v_1-g}^{1,g,+}\cup \left(\I_{-v_{i_0}+g-2,v_1-g}^{1,g,+}\setminus \I_{-1,v_1-g}^{1,g,+}\right)}\\
\bigcup_{\substack{i=1\\i\neq i_0}}^{t}\left(\I_{v_i,v_1-g}^{1,g,+}\cup \I_{-v_i+g-2,v_1-g}^{1,g,+}\right)
\cup\textcolor{brown}{\I_{g/2-1,v_1-g}^{1,g,+}}
 ,
 \end{multline*}
 \begin{multline*}
\cch_{1,-}=   \displaystyle\bigcup_{\substack{i=1\\i\neq i_0}}^{t}\left(\I_{-v_1+g-2,v_i-(1+\mathds{1}_{i=1})g}^{1,g,-}\cup \I_{-v_1+g-2,-v_i-2}^{1,g,-}\right)\\\cup\textcolor{teal}{\I_{-v_1+g-2,-v_{i_0}-2}^{1,g,-}\cup\left(\I_{-v_1+g-2,v_{i_0}-(1+\mathds{1}_{i_0=1})g}^{1,g,-}\setminus \I_{-v_1+g-2,-g-1}^{1,g,-}\right)}\cup\textcolor{brown}{\I_{-v_1+g-2, -g/2-1}^{1,g,-}}.
\end{multline*}
\end{lm}

Following the steps of the proof of Theorem \ref{thm:DDpair}, using Lemma \ref{lem:maxhookddprime}, we derive the following result which will be necessary to derive a Nekrasov--Okounkov formula for type $A^{(2)}_{2t-1}$.
\begin{thm}\label{thm:DDtquotpair}
Set $t$ a positive integer and $g=2t$. Set $\lambda\in \ccdd_{g}^{'1}$ and $\mathbf{v}\in\bbbz^{t}$ the $V_{g,t}$-coding associated to $\lambda$, and set $r_i=v_i-g/2+1$ for any $i\in\lbrace1,\dots,t\rbrace$. Then we have
$$\lvert \lambda\rvert = \frac{1}{g}\sum_{i=1}^{t} r_i^2-\frac{(g+1)(g/2 +1)}{12}.$$

Moreover setting $\alpha_i(\lambda)$ as in Theorem \ref{thm:DDpair}, for any function $\tau:\bbbz\rightarrow F^{\times}$, where $F$ is a field, we also have
\begin{multline}\label{eq:thmddprimepairtquot}
\prod_{s\in\lambda}\frac{\tau(h_s-\varepsilon_s g)}{\tau(h_s)}=\prod_{i=1}^{g-1}\left(\frac{\tau(-i)}{\tau(i)}\right)^{\alpha_i(\lambda)}
\prod_{i=1}^{t}\frac{\tau(r_i)}{\tau(i)}\prod_{1\leq i<j\leq t} \frac{\tau(r_i-r_j)}{\tau(j-i)} \frac{\tau(r_i+r_j)}{\tau(g+2-i-j)}.
\end{multline}

\end{thm}

\subsubsection{Type $D^{(2)}_{t+1}$ and $\ccsc_{(2t)}$}

In this subsection, set $g=2t$. Recall that the quadratic forms appearing to the exponent of $T$ in the Macdonald identity for type $D^{(2)}_{t+1}$ can be interpreted as half the weight of an element in $\ccsc_{(g)}$.
Using \eqref{eq:motsc} and the same considerations as for the proof of Lemma \ref{lem:maxhook}, we have the following similar result.
\begin{lm}\label{lem:maxhooksc}
Let $t$ be a positive integer and set $g=2t$. Set $\lambda\in \ccsc_{(g)}$ and $(v_1,\dots,v_{t})\in\bbbz^t$ its associated $V_{g,t}$-coding. Then the largest hook of $\lambda$, denoted by $H_1$ corresponds to the collection of boxes of indices in

\begin{align*}
H_1:=\cch_{1,+}\cup\cch_{1,-}
\end{align*}
where $\cch_{1,+}$ (respectively $\cch_{1,-}$) denotes the set of indices of boxes $s$ in the first hook such that $\varepsilon_s=1$ (respectively $\varepsilon_s=-1$) and
\begin{align*}
\cch_{1,+}&=\textcolor{red}{\I_{-v_1+g-1,v_1-g}^{1,g,+}} \displaystyle\bigcup_{i=2}^{t}\left(\I_{v_i,v_1-g}^{1,g,+}\cup \I_{-v_i+g-1,v_1-g}^{g,+}\right),\\
\cch_{1,-}&= \textcolor{blue!30}{\I_{-v_1+g,v_1-2g}^{1,g,-}} \displaystyle\bigcup_{i=2}^{t}\left(\I_{-v_1+g-1,v_i-g}^{1,g,-}\cup \I_{-v_1+g-1,-v_i}^{1,g,-}\right).
\end{align*}

\end{lm}

As for \eqref{lem:Delta}, we derive similarly that for any self-conjugate $g$-core partition
\begin{equation}\label{lem:Deltasc}
\lbrace h_s\mid s\in\Delta\rbrace=\bigcup_{i=1}^{t}\lbrace 2kg+2\sigma(i)+1	\mid  0\leq k\leq n_{\sigma(i)}-1\rbrace\rbrace.
\end{equation}

Using Lemmas \ref{lem:maxhooksc} and \ref{lem:telescop}, we derive the following analogue of Theorem \ref{thm:DDpair}.
\begin{thm}\label{thm:SC}
Set $t$ a positive integer and $g=2t$. Let $\omega\in \ccsc_{(g)}$ and $\mathbf{v}\in\bbbz^{t}$ its associated $V_{g,t}$-coding, and set $r_i=v_i-g/2+1/2$ for any $i\in\lbrace 1,\dots,t\rbrace$. Then we have
$$\lvert \omega\rvert = \frac{1}{g}\sum_{i=1}^{t} r_i^2-\frac{g^2-1}{24}.$$

Moreover setting $\alpha_i(\omega)$ as in Theorem \ref{thm:DDpair}, for any function $\tau:\bbbz\rightarrow F^{\times}$, where $F$ is a field, we also have

\begin{equation}\label{eq:thmsc}
\prod_{s\in\omega}\frac{\tau(h_s-\varepsilon_s g)}{\tau(h_s)}=\prod_{i=1}^{g-1}\left(\frac{\tau(-i)}{\tau(i)}\right)^{\alpha_i(\omega)}
\prod_{1\leq i<j\leq t} \frac{\tau(r_i-r_j)}{\tau(j-i)} \frac{\tau(r_i+r_j)}{\tau(g+1-i-j)}.
\end{equation}
\end{thm}

\subsubsection{Type $A^{(2)}_{2t}$ and $\ccdd_{(2t+1)}$}

In this subsection, we set $g=2t+1$. Recall that the quadratic forms appearing to the exponent of $T$ in the Macdonald identity for type $A^{(2)}_{2t}$ can be interpreted as half the weight of an element in $\ccdd_{(g)}$.
Similarly to Lemma \ref{lem:maxhook}, one derives the following lemma.
\begin{lm}\label{lem:maxhookimpair}
Let $t$ be a positive integer and set $g=2t+1$. Set $\lambda\in \ccdd_{(g)}$ and $(v_1,\dots,v_{t})\in\bbbz^t$ its associated $V_{g,t}$-coding. Then the largest hook of $\lambda$, denoted by $H_1$ corresponds to the collection of boxes of indices in

\begin{align*}
H_1:=\cch_{1,+}\cup\cch_{1,-}
\end{align*}
where $\cch_{1,+}$ (respectively $\cch_{1,-}$) denotes the set of indices of boxes $s$ in the first hook such that $\varepsilon_s=1$ (respectively $\varepsilon_s=-1$) and
\begin{align*}
\cch_{1,+}&=\textcolor{red}{\I_{-v_1+g,v_1-g}^{1,g,+}}  \cup \textcolor{teal}{\I_{0,v_1-g}^{1,g,+}}
 \displaystyle\bigcup_{i=2}^{t}\left(\I_{v_i,v_1-g}^{1,g,+}\cup \I_{-v_i+g,v_1-g}^{1,g,+}\right),\\
\cch_{1,-}&= \textcolor{blue!30}{\I_{-v_1+g,v_1-2g}^{1,g,-}} \cup\textcolor{teal}{\I_{-v_1+g,-g}^{1,g,-}} \displaystyle\bigcup_{i=2}^{t}\left(\I_{-v_1+g,v_i-g}^{1,g,-}\cup \I_{-v_1+g,-v_i}^{1,g,-}\right).
\end{align*}

\end{lm}

The proof of Theorem \ref{thm:DDimpair} stated below is exactly the same as the one of \eqref{eq:thmdd} in Theorem \ref{thm:DDpair} but without the products in brown  in \eqref{eq:telescoprod} and \eqref{eq:telescoprodmoins}, as underlined in the proof of Theorem \ref{thm:DDtquotimpair}. Using Lemmas \ref{lem:maxhookimpair} and \ref{lem:telescop}, the proof follows the same step as the one of Theorem \ref{thm:DDpair}.

\begin{thm}\label{thm:DDimpair}
Set $t$ a positive integer and $g=2t+1$. Set $\omega\in \ccdd_{(g)}$ and $\mathbf{v}\in\bbbz^{t}$ its associated $V_{g,t}$-coding, and set $r_i=v_i-g/2$ for any $i\in\lbrace1,\dots,t\rbrace$. Then we have
$$\lvert \omega\rvert = \frac{1}{g}\sum_{i=1}^{t} r_i^2-\frac{(g-2)\lfloor g/2\rfloor}{12},$$
Moreover, setting $\alpha_i(\omega)$ as in Theorem \ref{thm:DDpair}, for any function $\tau:\bbbz\rightarrow F^{\times}$, where $F$ is a field, we also have
\begin{equation}\label{eq:thmddimpair}
\prod_{s\in\omega}\frac{\tau(h_s-\varepsilon_s g)}{\tau(h_s)}=\prod_{i=1}^{g-1}\left(\frac{\tau(-i)}{\tau(i)}\right)^{\alpha_i(\omega)}
\prod_{1\leq i<j\leq t} \frac{\tau(r_i-r_j)}{\tau(j-i)} \frac{\tau(r_i+r_j)}{\tau(g-i-j)}.
\end{equation}

\end{thm}

\subsubsection{Type $D^{(1)}_{t}$ and and $\ccdd_{2t-2}^{'2}$}
Set $g=2t+2$.
If we take $\lambda\in\ccdd_{g}^{'2}$ such that $\nu^{(2t-3)}$ and $\nu^{(t-2)}$ has their corresponding word equal to \eqref{eq:nu0typeb} and \eqref{eq:nu0typed}, and $\psi(\lambda)=(c_k)_{k\in\bbbz}$, remark that the set of indices congruent to $t-2\pmod{g}$ corresponding to a letter ``$0$'' is $\lbrace kg-t \mid k\in\bbbz, k\leq -m_t \text{ or } 1\leq k\leq m_t\rbrace$. Using \eqref{eq:motddprime} and the same considerations as for the proof of Lemma \ref{lem:maxhook}, we have the following similar result. One can remark that \eqref{eq:gkddnupairtquot} could also have been interpreted as a doubled distinct partition. We chose this convention in order to have a definition of $V_{g,t}$-coding from Definition \ref{def:vcoding} as much independent of type as possible. The price to pay is that in the following results, if $\lambda\in\ccdd_{g}^{'2}$ and $s$ a box of $\lambda$, the sign statistic $\varepsilon_s$ is replaced with $\varepsilon'_s$ being equal to $1$ if $s$ is strictly above the main diagonal $\Delta$ and $-1$ otherwise. Note that in this case, one has:
$$\lbrace h_s-\varepsilon'_s g,\lambda\in\ccdd_{g}^{'2},s\in\lambda\rbrace=\lbrace h_s+\varepsilon_s g, \lambda'\in\ccdd_{g}^{'2},s\in\lambda'\rbrace.$$

\begin{lm}\label{lem:maxhookddprimed}
Let $t$ be a positive integer and set $g=2t-2$. Set $\lambda\in \ccdd_{g}^{'2}$ and $\mathbf{v}\in\bbbz^{t}$ its associated $V_{g,t}$-coding. Let $(i_0,i_1)\in\lbrace 1,\dots,t\rbrace^2$ be such that $v_{i_0}=m_1g-1$ and $v_{i_1}=m_tg+t-2$
Then the largest hook of $\lambda$, denoted by $H_1$, corresponds to the collection of boxes of indices 
\begin{align*}
H_1=\cch_{1,+}\cup\cch_{1,-}
\end{align*}
where $\cch_{1,+}$ (respectively $\cch_{1,-}$) denotes the set of indices of boxes $s$ in the first hook such that $\varepsilon'_s=1$ (respectively $\varepsilon'_s=-1$) and

\begin{multline*}
\cch_{1,+}=  \displaystyle\bigcup_{\substack{i=1\\ i\neq i_0\\ i\neq i_1}}^{t}\left(\I_{v_i,v_1	-g}^{1,g,+}\cup \I_{-v_i+(1+\mathds{1}_{i=1})g-2,v_1-g}^{1,g,+}\right)\\\cup \textcolor{teal}{\I_{v_{i_0},v_1-g}^{1,g,+}\cup \left(\I_{-v_{i_0}+ (1+\mathds{1}_{i_0=1})g-2,v_1-g}^{1,g,+}\setminus \I_{-1,v_1-g}^{1,g,+}\right)}\\
\cup\textcolor{brown}{\I_{v_{i_1},v_1-g}^{1,g,+}\cup \left(\I_{-v_{i_1}+(1+\mathds{1}_{i_1=1})g-2,v_1-g}^{1,g,+}\setminus \I_{g/2-1,v_1-g}^{1,g,+}\right)},
 \end{multline*}
 \begin{multline*}
\cch_{1,-}=\displaystyle\bigcup_{\substack{i=1\\ i\neq i_0\\ i\neq i_1}}^{t}\left(\I_{-v_1+g-2,v_i-g}^{1,g,-}\cup \I_{-v_1+g-2,-v_i-2}^{1,g,-}\right)\\\cup\textcolor{teal}{\I_{-v_1+g-2,-v_{i_0}-2}^{1,g,-}\cup\left(\I_{-v_1+g-2,v_{i_0}-g}^{1,g,-}\setminus \I_{-v_1+g-2,-g-1}^{1,g,-}\right)}\\\cup\textcolor{brown}{\I_{-v_1+g-2,-v_{i_1}-2}^{1,g,-}\cup\left(\I_{-v_1+g-2,v_{i_1}-g}^{1,g,-}\setminus \I_{-v_1+g-2,-g/2-1}^{1,g,-}\right)} .
\end{multline*}
\end{lm}

Following the steps of the proof of Theorem \ref{thm:DDtquotimpair}, using Lemmas \ref{lem:maxhookddprimed} and \ref{lem:telescop}, we derive the following result on products of hook lengths for elements of $\ccdd_{2t-2}^{'2}$ related to type $D^{(1)}_{t}$ as underlined in Section \ref{chap3:quad}.

\begin{thm}\label{thm:DDtquotpaird}
Set $t$ a positive integer and $g=2t-2$. Set $\lambda\in \ccdd_{g}^{'2}$ and $\mathbf{v}\in\bbbz^{t}$ the $V_{g,t}$-coding associated to $\lambda$, and set $r_i=v_i-g/2+1$ for any $i\in\lbrace1,\dots,t\rbrace$. Then we have
$$\lvert \lambda\rvert = \frac{1}{g}\sum_{i=1}^{t} r_i^2-\frac{(g/2+1)(g+1)}{12}.$$
Moreover setting $\alpha_i(\omega):=\#\lbrace s\in\omega, h_s=g-i, \varepsilon'_s=1\rbrace$ , for any function $\tau:\bbbz\rightarrow F^{\times}$, where $F$ is a field, we also have
\begin{multline*}\label{eq:thmddprimepairtquotd}
\prod_{s\in\lambda}\frac{\tau(h_s-\varepsilon'_s g)}{\tau(h_s)}=\prod_{s\in\lambda'}\frac{\tau(h_s+\varepsilon_s g)}{\tau(h_s)}\\
=\prod_{i=1}^{g-1}\left(\frac{\tau(-i)}{\tau(i)}\right)^{\alpha_i(\lambda)}\prod_{i=1}^t\frac{\tau(2r_i)}{\tau(r_i)}\prod_{i=1}^{t-1}\frac{\tau(i)}{\tau(2i)}
\prod_{1\leq i<j\leq t} \frac{\tau(r_i-r_j)}{\tau(j-i)} \frac{\tau(r_i+r_j)}{\tau(g+2-i-j)}.
\end{multline*}

\end{thm}

\begin{proof}
First of all, we start by remarking that the terms of the product \linebreak $\tau(h-\varepsilon' g)/\tau(h)$ telescop in the same way as the terms $\tau(h-\varepsilon g)/\tau(h)$ of Theorem \ref{thm:DDtquotpair} except for the hook lengths of the boxes on the main diagonal $\Delta$. 
This induces the term of the form $\tau(2r_i)/\tau(2\sigma(i)+2-\mathds{1}_{\sigma(i)\geq g/2-1}g)$ in the above theorem, where $\sigma(i)\equiv r_i\pmod{g}$.

The only other thing that differs from the proof of Theorem \ref{thm:DDtquotpair} resides in the fact that the self-conjugate part of the $(2t-2)$-quotient is no longer necessarily empty. This leads to change the form of the telescopic product over the boxes whose pair of indices in the corresponding word has at least one component congruent to~$g/2-1\pmod{g}$ (which are shaded in brown in Lemma \ref{lem:maxhookddprimed}).
If $i_0>1$ and $i_1>1$, then the products in brown when computing the first step of induction are now
\begin{multline*}
\textcolor{brown}{\left(\frac{\tau(\sigma(1)-(g/2-1)-\mathds{1}_{\sigma(1)>g/2-1}g)}{\tau(v_1-v_{i_1}-g)}\right)^{\mathds{1}_{g<v_1-v_{i_1}}}\left(\frac{\tau(v_1-g/2+1-g)}{\tau(v_1+v_{i_1}-2g+2)}\right)^{\mathds{1}_{g/2-1<v_1}}}\\
\times\textcolor{brown}{\left(\frac{\tau(v_1-v_{i_1})}{\tau(\sigma(1)-(g/2-1)+\mathds{1}_{\sigma(1)<g/2-1}g)}\right)^{\mathds{1}_{g<v_1-v_{i_1}}}\left(\frac{\tau(v_1+v_{i_1}-g+2)}{\tau(v_1-g/2+1)}\right)^{\mathds{1}_{g/2-1<v_1}}}.
\end{multline*}
 
The cases $i_0=1$ or $i_1=1$ are handled in the same way. Note that
the term $\tau(2*0)$ in the last step of the induction of the product over the boxes of $\Delta$ corresponds to the case of $r_i\equiv 0\pmod{g}$. This term cancels with the term $\tau(0)$ coming from the last step of the induction of the product above. Indeed recall when $\lambda$ is the empty partition, $(r_i)_{1\leq i\leq t}$ is the vector $(0,1,\dots,g/2)$.
\end{proof}



%

\section{Macdonald identities for affine root systems and integer partitions}\label{chap4:macdo}
%

Motivated by the computations from Section \ref{sec:Macdointer}, the last section detailed some enumerative results about subsets of partitions linked to the quadratic forms that appear in the Macdonald identities of each infinite affine root system. In this section, the goal is to show how these enumerations enable one to transcribe the Macdonald identities within the language of integer partitions and Schur functions. The $V_{g,t}$-codings introduced in Definition \ref{def:vcoding} (that are some specific decreasingly ordered vectors of integers) are going to be the key tool for this purpose. These are heavily linked to the bijection $\sigma : \lbrace 1,\dots,g\rbrace\rightarrow \lbrace 0,\dots,g-1\rbrace $ involved in their definition and which can be thought of as a shifted permutation $\sigma'\in S_{g}$, i.e. $\sigma'(i):=g-\sigma(i)$. In Section \ref{chap4:sign}, we detail how the sign of these $\sigma$, denoted by $\sgn(\sigma):=-\sgn(\sigma')$ and which can be thought of as a signature of a permutation sorting in decreasing order, is characterized by both the cardinal of some subsets of hook lengths of the partitions introduced in Section \ref{chap3:quad} and some statistics attached to the partitions (such as $\Delta$ or the size of the Durfee square). 

The purpose of Section \ref{sec:Macdo} is then to provide the desired rewriting of Macdonald identities for all infinite affine types in the language of partition. These formulas are of type ``sum = product'', and recall that in \eqref{eq:macdoAfin} and \eqref{eq:misign}, the sum is computed over $\mathbf{m}$ in some subsets of $\bbbz^t$ without any order on its components. However, partitions come with an order, hence when sending the vector $\mathbf{m}$ to the vector $\mathbf{r}$ from Theorems \ref{thm:HD}-\ref{thm:DDpair} and the theorems from Section \ref{chap3:hooks} we have chosen a specific order: that is the translation by a constant of a $V_{g,t}$-coding. We apply this formalism to rewrite the sum part of Macdonald identities for all $7$ infinite root systems as a sum over Schur functions, and over symplectic, special orthogonal and even orthogonal Schur functions. 


\subsection{Signs of permutations and integer partitions}\label{chap4:sign}

In this section, we establish the lemmas that bridge the sign of $\sigma : \lbrace 1,\dots,g\rbrace\rightarrow \lbrace 0,\dots,g-1\rbrace $ as defined in Definition \ref{def:vcoding} together with some subsets of the corresponding partitions for each type. The order of the types in this section slightly differs from the rest of the paper because these enumerations depend on the subset of partitions attached to the quadratic form as introduced in Section \ref{chap3:quad}.

\subsubsection{Type $A^{(1)}_{t-1}$ and $\ccp_{(t)}$}

The case of type $A^{(1)}_{t-1}$ is once again a bit particular. Indeed one of the differences between Theorems \ref{thm:HD} (stated for type $A^{(1)}_{t-1}$) and \ref{thm:DDpair} (stated for type $C^{(1)}_{t}$) is that in the left-hand side of Theorem \ref{thm:HD}, $\tau(h-t)\tau(h+t)/\tau(h)^2$ does not depend on the signed statistic $\varepsilon$, whereas in Theorem \ref{thm:DDpair} this term is a single quotient of functions $\tau$ depending on $\varepsilon$. 

\begin{lm}\label{lem:signschurA}
Let $t$ be a positive integer and $g=t$.
Let $\omega\in\ccp_{(t)}$ and $\psi(\omega)=(c_k)_{k\in\bbbz}$. 
Let $\mathbf{v}$ be its associated $V_{t,t}$-coding and $\sigma$ as defined in Definition \ref{def:vcoding}. Then

\begin{equation}
\lvert H_{<t}\rvert=\#\lbrace s\in\omega, h_s<t\rbrace\equiv \sgn(\sigma)\pmod{2}.
\end{equation}
\end{lm}

\begin{proof}
Set $\psi(\omega)=(c_k)_{k\in\bbbz}$ the word corresponding to $\omega$ by Lemma \ref{lem:indices}, then $H_{<t}$ as defined in Theorem \ref{prop:MacdotypeA} is equivalently described by $H_{<t}=\lbrace (k,l)\in\bbbz \mid k<l,\, c_k=1,\, c_l=0 \text{ and } l-k<t\rbrace$. We decompose $H_{<t}$ with respect to the remainders $\pmod{t}$ of the indices of the letters ``$0$'' and ``$1$''. Let $\mathbf{v}$ be its associated $V_{t,t}$-coding and $\sigma$ as defined in Definition \ref{def:vcoding}.
More precisely, let $k,l$ be two integers such that $1\leq k <l \leq t$. Note that
\begin{align*}
\lbrace a\in \Z, c_{a}=1, a\equiv \sigma(l)\pmod{t}\rbrace&=\lbrace \sigma(l)+qn,q\geq \lfloor\frac{v_l}{t}\rfloor\rbrace,\\
\lbrace b\in \Z, c_{b}=0, b\equiv \sigma(k)\pmod{t}\rbrace&=\lbrace \sigma(k)+qn,q< \lfloor\frac{v_k}{t}\rfloor\rbrace.
\end{align*}

The elements $(\sigma(l)+qt,\sigma(k)+(q-\mathds{1}_{\sigma(l)>\sigma(k)})t)$ correspond to the indices of boxes whose hook length is strictly less than $t$ such that the indices in the corresponding word are congruent to $\sigma(l)$ and $\sigma(k)\pmod{t}$. This set is fully characterized by the value of $q$ which is greater than $n_{\sigma(l)}$ and $n_{\sigma(k)}-1-\mathds{1}_{\sigma(l)>\sigma(k)}$ where we set $n_{\sigma(i)}=\lfloor v_i/t\rfloor$.

Therefore this enumeration of $H_{<t}$ yields
\begin{equation}\label{interhA}
\lvert H_{<t}\rvert=\sum_{\substack{(i,j)\in \lbrace 1,\dots ,t\rbrace^2\\i<j}}\left(n_{\sigma(i)}-n_{\sigma(j)}-\mathds{1}_{\sigma(i)<\sigma(j)}\right).
\end{equation}

Since we are interested here in having a relation on the parity of $\sgn(\sigma)$, rewriting \eqref{interhA} $\pmod{2}$ yields
\begin{equation}\label{interhAinter}
\lvert H_{<t}\rvert\equiv\sum_{\substack{(i,j)\in \lbrace 1,\dots ,t\rbrace^2\\i<j}}\left(n_{\sigma(i)}+n_{\sigma(j)}+\mathds{1}_{\sigma(i)<\sigma(j)}\right)\pmod {2}.
\end{equation}
Gathering the common terms, one derives from \eqref{interhAinter}
\begin{equation*}
\lvert H_{<t}\rvert\equiv(t-1)\sum_{i=0}^{t-1} n_{\sigma(i)}+\sum_{\substack{(i,j)\in \lbrace 1,\dots ,t\rbrace^2\\i<j}}\mathds{1}_{\sigma(i)<\sigma(j)}\pmod{2}.
\end{equation*}

We then conclude using Theorem \ref{thm:gk} which states that for elements $\omega\in\ccp_{(t)}$ and $\phi(\omega)=(n_0,\dots,n_{t-1})$ one has the vanishing condition $\sum_{i=0}^{t-1} n_i=0$. 
\end{proof}

\subsubsection{Type $C^{(1)}_{t}$ and $\ccdd_{(2t+2)}$}

We now prove the following lemma that will be quite useful when applying Theorem \ref{thm:DDpair} to specializations in Macdonald identities. Its proof follows the same path as the proof of Lemma \ref{lem:signschurA}, but requires an extra discussion due to the presence of $g$-intervals of different kinds in this case (as explained in Section \ref{chap3:hooks}). This explains the presence of an extra term (which is the size of the Durfee square).
\begin{lm}\label{lem:signschur}
Set $t$ be a positive integer and $g=2t+2$.
Set $\omega\in\ccdd_{(g)}$ and $\psi(\omega)=(c_k)_{k\in\bbbz}$. 
Let $\mathbf{v}$ be the $V_{g,t}$-coding associated to $\omega$ together with $\sigma : \lbrace 1,\dots,g\rbrace\rightarrow \lbrace 0,\dots,g-1\rbrace$ as defined in Definition \ref{def:vcoding}. Then

\begin{equation}
\lvert H_{<g,+}\rvert=\#\lbrace s\in\omega, h_s<g, \varepsilon_s=1\rbrace\equiv d_\omega+\sgn(\sigma)\pmod{2}.
\end{equation}
\end{lm}
\begin{proof}
By \eqref{eq:indices} and Lemma \ref{lem:indices} $(4)$, we can split the boxes in $H_{<g,+}$ according to the corresponding pair of indices. By Lemma \ref{lem:indices}, we associate a box $s\in\omega$ to $(i_s,j_s)$. Then $H_{<g,+}=\lbrace j_s-i_s \mid i_s<j_s, c_{j_s}=0, c_{i_s}=1\rbrace$ and we can rewrite $H_{<g,+}$ as the disjoint union of the respective congruence classes $\pmod{g}$ of $(i_s,j_s)$. Remark that the indices of letter ``$0$" of the boxes who are located above the main diagonal must be positive which means that they are congruent to $\sigma(i)\pmod{g}$ for $1\leq i \leq t$.
Let $k,l$ be two integers such that $1\leq k <l \leq t$.

The elements $(\sigma(l)+qg,\sigma(k)+(q-\mathds{1}_{\sigma(l)>\sigma(k)})g)$ correspond to the indices of boxes above the main diagonal whose hook length is strictly less than $g$ such that the indices in the corresponding word are congruent to $\sigma(l)$ and $\sigma(k)\pmod{t}$. This set is fully characterized by the value of $q$ which is greater than $n_{\sigma(l)}$ and $n_{\sigma(k)}-1-\mathds{1}_{\sigma(l)>\sigma(k)}$ where we set $n_{\sigma(i)}=\lfloor v_i/g\rfloor$.

Next we enumerate the number of boxes above the main diagonal whose indices of ``$1$" in the corresponding word are not congruent to $\sigma(i)\pmod{g}$ for $1\leq i \leq t$.
For any $(k,l)\in\lbrace 1,\dots,t\rbrace^2$, the elements $(g-\sigma(l)+qg,\sigma(k)+(q-\mathds{1}_{g-\sigma(l)>\sigma(k)})g)$ correspond to the indices of boxes whose hook length is strictly less than $g$ such that the indices of the box in the corresponding word are congruent to $g-\sigma(l)$ and $\sigma(k)\pmod{t}$ and the box is above the main diagonal.
If $\sigma(k)+\sigma(l)\geq g$, there are exactly $n_{\sigma(k)}$ elements. If $\sigma(k)+\sigma(l)<g$, then it remains to determine whether the box $s$ associated to the pair of indices $(-\sigma(l),\sigma(k))$ is above the main diagonal or not. Hence by Corollary \ref{cor:position}, the number of boxes in $H_{<g,+}$ corresponding to this enumeration corresponds to $n_{\sigma(k)}-\mathds{1}_{\sigma(k)+\sigma(l)<g}\times\mathds{1}_{\sigma(k)>\sigma(l)}$. 
It remains to deal with the boxes above the main diagonal  whose hook length is strictly less than  $g$ such that the corresponding indices of letters ``$1$" are congruent to either $0$ or $g/2$ $\pmod{g}$. The first one corresponds to the set of indices of the form $(qg,\sigma(k)+qg)$ for $0\leq q \leq n_{\sigma(k)}-1$, whereas the second corresponds to the set of indices $(g/2+qg,\sigma(k)+(q-\mathds{1}_{g/2>\sigma(k)})g)$ for $0\leq q \leq n_{\sigma(k)}-1$.
This enumeration yields:
\begin{multline}\label{eq:interhplus}
\lvert H_{<g,+}\rvert=\sum_{\substack{(i,j)\in \lbrace 1,\dots ,t\rbrace^2\\i<j}}\left(n_{\sigma(i)}-n_{\sigma(j)}-\mathds{1}_{\sigma(i)<\sigma(j)}+n_{\sigma(i)}+n_{\sigma(j)}-\mathds{1}_{\sigma(i)+\sigma(j)<g}\right)\\
+\sum_{i=1}^t\left(n_{\sigma(i)}+n_{\sigma(i)}+n_{\sigma(i)}-\mathds{1}_{\sigma(i)<g/2}\right).
\end{multline}

%

By gathering the common terms, we derive the following from the above equality:
\begin{equation}
\lvert H_{<g,+}\rvert\equiv\sum_{i=1}^t n_{\sigma(i)}+\sum_{\substack{(i,j)\in \lbrace 1,\dots ,t\rbrace^2\\i<j}}\left(\mathds{1}_{\sigma(i)<\sigma(j)}+\mathds{1}_{\sigma(i)<g-\sigma(j)}+\mathds{1}_{\sigma(i)<g/2}\right)\pmod{2}.\label{eq:hpluspair}
\end{equation}
The last step of the proof consists in noticing that the sum over $n_{\sigma(i)}$ for $i\in\lbrace 1,\dots,t\rbrace$ is the sum over the positive $(n_k)_{k\in\lbrace 0,\dots,2t+1\rbrace}$. It is therefore the size of the Durfee square of $\omega$, whence the desired equality.
\end{proof}

\subsubsection{Type $A^{(2)}_{2t}$ and $\ccdd_{(2t+1)}$}
Similarly we get the following lemma needed in the rewriting of Macdonald identities for affine type $A^{(2)}_{2t}$ as a sum of special orthogonal Schur functions.
\begin{lm}\label{lem:signschurbc}
Set $t$ be a positive integer and $g=2t+1$.
Set $\omega\in\ccdd_{(g)}$ and $\psi(\omega)=(c_k)_{k\in\bbbz}$. Let $\mathbf{v}$ be the $V_{g,t}$-coding associated to $\omega$ together with $\sigma : \lbrace 1,\dots,g\rbrace\rightarrow \lbrace 0,\dots,g-1\rbrace$ as defined in Definition \ref{def:vcoding}. Then
\begin{equation}
\lvert H_{<g,+}\rvert=\#\lbrace s\in\omega, h_s<2t+1, \varepsilon_s=1\rbrace\equiv \sgn(\sigma)+
\lvert H_{<g,+}\cap \Delta\rvert\pmod{2}.
\end{equation}
\end{lm}
\begin{proof}
The proof follows the same exact steps as for Lemma \ref{lem:signschur} except that the terms coming from the indices of $(g/2+qg,\sigma(k)+(q-\mathds{1}_{g/2>\sigma(k)})g)$ for $0\leq q \leq n_{\sigma(k)}-1$ no longer appears since $g/2$ is not an integer. Hence \eqref{eq:interhplus} becomes here
\begin{multline}\label{eq:hplusimpair}
\lvert H_{<g,+}\rvert=\sum_{\substack{(i,j)\in \lbrace 1,\dots ,t\rbrace^2\\i<j}}\left(n_{\sigma(i)}-n_{\sigma(j)}-\mathds{1}_{\sigma(i)<\sigma(j)}+n_{\sigma(i)}+n_{\sigma(j)}-\mathds{1}_{\sigma(i)+\sigma(j)<g}\right)\\
+\sum_{i=1}^t\left(n_{\sigma(i)}+n_{\sigma(i)}\right).
\end{multline}
Hence \eqref{eq:hplusimpair} becomes
\begin{equation*}
\lvert H_{<g,+}\rvert\equiv\sum_{\substack{(i,j)\in \lbrace 1,\dots ,t\rbrace^2\\i<j}}\left(\mathds{1}_{\sigma(i)<\sigma(j)}+\mathds{1}_{\sigma(i)<g-\sigma(j)}\right)\pmod{2}.
\end{equation*}
To complete the proof, it remains to add $\mathds{1}_{\sigma(i)<g/2}-\mathds{1}_{\sigma(i)<g/2}$ and note that by \eqref{lem:Delta}
$$\lvert H_{<g,+} \cap \Delta \rvert=\# \lbrace i\in \lbrace 1,\dots,t\rbrace \mid
\sigma(i)<g/2 \rbrace.$$
\end{proof}

\subsubsection{Type $D^{(2)}_{t+1}$ and $\ccsc_{(2t)}$}

Following the same kind of enumeration as for Lemma \ref{lem:signschurbc}, we have that \eqref{eq:hplusimpair} becomes:
\begin{multline}\label{eq:hplussc}
\lvert H_{<g,+}\rvert=\sum_{\substack{(i,j)\in \lbrace 1,\dots ,t\rbrace^2\\i<j}}\left(n_{\sigma(i)}-n_{\sigma(j)}-\mathds{1}_{\sigma(i)<\sigma(j)}+n_{\sigma(i)}+n_{\sigma(j)}-\mathds{1}_{\sigma(i)+\sigma(j)<g-1}\right)\\
+\sum_{i=1}^t n_{\sigma(i)}.\notag
\end{multline}

Hence	 one derives the following lemma which will be applied when rewriting the Macdonald identity for affine type $D^{(2)}_{t+1}$ as a sum over orthogonal Schur functions.
\begin{lm}\label{lem:signschursc}
Set $t$ be a positive integer and $g=2t$.
Set $\omega\in\ccsc_{(g)}$ and $\psi(\omega)=(c_k)_{k\in\bbbz}$. 
Let $\mathbf{v}$ be its associated $V_{g,t}$-coding and $\sigma$ as defined in Definition \ref{def:vcoding}. Then
\begin{equation}
\lvert H_{<g,+}\rvert=\#\lbrace s\in\omega, h_s<2t, \varepsilon_s=1\rbrace\equiv \sgn(\sigma)+\lvert H_{<g,+}\cap \Delta\rvert+d_\omega\pmod{2}.
\end{equation}
\end{lm}

\subsubsection{Type $A^{(2)}_{2t-1}$ and $\ccdd_{2t}^{'1}$}

Adapting the proof of Lemma \ref{lem:signschur}, we derive the following lemma.
\begin{lm}\label{lem:signschurtypebv}
Set $t$ be a positive integer and $g=2t$.
Set $\lambda\in\ccdd_{g}^{'1}$ and $\psi(\lambda)=(c_k)_{k\in\bbbz}$. 
Let $\mathbf{v}$ be the $V_{g,t}$-coding associated to $\lambda$ together with $\sigma : \lbrace 1,\dots,g\rbrace\rightarrow \lbrace 0,\dots,g-1\rbrace$ as defined in Definition \ref{def:vcoding}. Then
\begin{equation}
\lvert H_{<g,+}\rvert=\#\lbrace s\in\omega, h_s<2t, \varepsilon_s=1\rbrace\equiv d_\lambda+\sgn(\sigma)\pmod{2}.
\end{equation}
\end{lm}

\begin{proof}
The only difference with the proof of Lemma \ref{lem:signschur} resides in the fact that $\lambda$ has a $(2t)$-quotient not necessarily empty. Let $m_1$ be a positive integer as in Proposition \ref{prop:ddprime2t}. Set $i_0\in\lbrace 1,\dots,t\rbrace$ such that $\sigma(i_0)=g-1$ and $v_{i_0}=m_1g-1$. Set $i\in\lbrace 1,\dots,t\rbrace$. If $i<i_0$, then $\sigma(i)<\sigma(i_0)$ and there are exactly $n_{\sigma(i)}-m_1$ pairs of indices $(k,l)$ such that $l-k\in H_{<g,+}$ and $l\equiv\sigma(i)\pmod{g}$ with $c_l=0$ and $k\equiv g-1\pmod{g}$ and $c_k=1$.
Similarly there are $m_1-1$ pair of indices $(k,l)$ such that $l-k\in H_{<g,+}$ and $l\equiv g-1\pmod{g}$ such that $c_l=0$ and $k\equiv \sigma(i)\pmod{g}$ such that $c_k=1$.
Similarly if $i>i_0$, this number becomes $2(m_1-1)-n_{\sigma(i)}$.
Noting that $d_\lambda=m_1-1+\sum_{\substack{i=1\\i\neq i_0}}^t n_{\sigma(i)}$ and following the same steps as for Lemma \ref{lem:signschur} where $\mathds{1}_{\sigma(i)<g-\sigma(j)}$ is replaced by $\mathds{1}_{\sigma(i)<g-2-\sigma(j)} $ concludes the proof.
\end{proof}

\subsubsection{Type $B^{(1)}_{t}$ and $\ccdd_{2t-1}^{'1}$}

Following the same path as for Lemma \ref{lem:signschurtypebv} yields the following lemma.
\begin{lm}\label{lem:signschurtypeb}
Set $t$ be a positive integer and $g=2t-1$.
Set $\lambda\in\ccdd_{g}^{'1}$ and $\psi(\lambda)=(c_k)_{k\in\bbbz}$.
Let $\mathbf{v}$ be the $V_{g,t}$-coding associated to $\lambda$ together with $\sigma : \lbrace 1,\dots,g\rbrace\rightarrow \lbrace 0,\dots,g-1\rbrace$ as defined in Definition \ref{def:vcoding}. Then
\begin{equation}
\lvert H_{<g,+}\rvert=\#\lbrace s\in\omega, h_s<2t-1, \varepsilon_s=1\rbrace\equiv \sgn(\sigma)+\lvert H_{<g,+}\cap \Delta\rvert\pmod{2}.
\end{equation}
\end{lm}

\subsubsection{Type $D^{(1)}_{t}$ and and $\ccdd_{2t-2}^{'2}$}

Following the same path as for Lemma \ref{lem:signschurtypebv} yields the following lemma.
\begin{lm}\label{lem:signschurtyped}
Set $t$ be a positive integer and $g=2t-2$.
Set $\lambda\in\ccdd_{g}^{'2}$ and $\psi(\lambda)=(c_k)_{k\in\bbbz}$. Let $\mathbf{v}$ be the $V_{g,t}$-coding associated to $\lambda$ together with $\sigma : \lbrace 1,\dots,g\rbrace\rightarrow \lbrace 0,\dots,g-1\rbrace$ as defined in Definition \ref{def:vcoding}. Then
\begin{equation}
\lvert H_{<g,+}\rvert=\#\lbrace s\in\omega, h_s<2t-2, \varepsilon'_s=1\rbrace\equiv d_\lambda+\sgn(\sigma)\pmod{2}.
\end{equation}
\end{lm}

\begin{proof}
The proof is essentially the same as the one of Lemma \ref{lem:signschurtypebv} with the difference that there are $n_{\sigma(i)}-\mathds{1}_{\sigma(i)<g/2-1}$ for $i\in\lbrace
1,\dots,t\rbrace \setminus \lbrace i_0,i_1\rbrace$  elements coming from the letters whose indices of ``$0"$, respectively ``$1$", in the corresponding word are congruent to $\sigma(i)\pmod{g}$, respectively $g-2-\sigma(i)\pmod{g}$. The term $\mathds{1}_{\sigma(i)<g/2-1}$ is compensated by the elements whose indices of ``$0$" is congruent to $g/2-1\pmod{g}$.
\end{proof}

\subsection{Rewriting of the Macdonald identities for affine root systems}\label{sec:Macdo}
%
Recall that Macdonald identities of Proposition \ref{prop:intermed_rewrite} involve power series over sublattices of $\bbbz^t$ of a formal variable $T$ whose coefficients are fractions of two determinants. The proofs in this section all follow the same steps:
\begin{enumerate}
\item interpret the exponent of the formal variable $T$ as one of the quadratic forms from Section \ref{chap3:quad},
\item then use the correspondence between vectors of integers and partitions studied in Section \ref{chap3:quad} together with the $V_{g,t}$-coding. Note that the convention used by Rosengren and Schlosser for the product of classical roots, e.g. $\Delta_A, \Delta_B, \Delta_C$ and $\Delta_D$ does not correspond to the product of classical roots in the denominator of characters of classical types, one needs to compose the $\sigma$ from Definition \ref{def:vcoding} with the permutation $j\mapsto t-j$ in order to reorder the fraction of determinants so that it corresponds to the classical definition of the irreducible character,
\item finally the signature coming from rearrangement of the columns of the determinants to the numerator is reinterpreted as some statistics associated with the corresponding partition thanks to the Lemmas from Section \ref{chap4:sign}.
\end{enumerate}

\subsubsection{Type $A^{(1)}_{t-1}$}

Here we prove Theorem \ref{prop:MacdotypeA}. First note that the right-hand side of \eqref{eq:macdoAfin} is the desired right-hand side of \eqref{eq:macdoa}. Recall the left-hand side of \eqref{eq:macdoAfin}
\begin{equation}\label{eq:macdoAlhs}
\sum_{\substack{\mathbf{m}\in\bbbz^t\\m_1+\dots+m_t=0}}T^{t\lVert m\rVert^2/2+\sum_{i=1}^t(i-1)m_i}\frac{\underset{1\le i,j\le t}{\det}\left(x_i^{tm_j+j-1}\right)}{\underset{1\le i,j\le t}{\det}\left(x_i^{j-1}\right)}.
\end{equation}

Set $\omega\in\ccp_{(t)}$ such that $\phi(\omega)=(m_1,\dots,m_t)$, where $\phi$ is the bijection from Theorem \ref{thm:gk} and $\mathbf{v}$ its associated $V_{t,t}$-coding. By definition of the $V_{t,t}$-coding, we have that \eqref{eq:ppoids} corresponds to \eqref{eq:gk}. Hence the exponent of $T$ is the weight of $\omega$. It remains to deal with the quotient of determinants.

Note that the exponents $tm_j+j-1$ of the $x_i$'s in \eqref{eq:macdoAlhs} can be negative, whereas in the definition of a Schur polynomial in \eqref{gldef} the exponents of the $x_i$'s in the determinant are positive.
To prove Theorem \ref{prop:MacdotypeA}, we will first reorder the columns of the matrix $\left(x_i^{tm_j+j-1}\right)_{1\le i,j\le t}$ so that the sequence $(tm_{\sigma(j)}+\sigma(j)-1)$ is in decreasing order. Set $\sigma:\lbrace 1,\dots,t\rbrace\mapsto\lbrace 0,\dots,t-1\rbrace$ as defined in Definition \ref{def:vcoding}. Take $\sigma'\in S_t$ such that $\sigma'(i):=1+\sigma(t-i)$ and $\sigma'^{-1}$ the inverse of $\sigma'$. Therefore for all $j\in\lbrace 1,\dots,t\rbrace$, $tm_j+j-1=v_{(\sigma')^{-1}(t-j)}$ and then reorder the columns to both the denominator and the numerator with the permutation that sends $j$ to $t-j$.

Hence, using Lemma \ref{lem:signschurA}, \eqref{eq:macdoAlhs} becomes
\begin{equation}\label{eq:macdoAlhsfin}
\sum_{\omega\in\ccp_{(t)}}(-1)^{\lvert H_{< t} \rvert}T^{\lvert \omega\rvert}\frac{\underset{1\le i,j\le t}{\det}\left(x_i^{v_j}\right)}{\underset{1\le i,j\le t}{\det}\left(x_i^{t-j}\right)},
\end{equation}
where $\mathbf{v}$ is the $V_{t,t}$-coding associated bijectively with $\omega$, as underlined in Remark \ref{vcodingtypeA}.
The final step of the proof to rewrite the quotient of determinants as a Schur polynomial is to transform \eqref{eq:macdoAlhsfin} so that all of the exponents of the $x_i$'s become nonnegative. By definition of $\mathbf{v}$, $v_t=\min(v_1,\dots,v_t)$. Therefore $v_i-v_t\geq 0$ for any $i\in\lbrace 1,\dots,t \rbrace$ and factorizing by $\prod_{i=1}^t x_i^{v_t}$, one has

$$\underset{1\le i,j\le t}{\det}\left(x_i^{v_j}\right)=\underset{1\le i,j\le t}{\det}\left(x_i^{v_j-v_t}\right)\prod_{i=1}^t x_i^{v_t}.$$

To complete the proof of Theorem \ref{prop:MacdotypeA}, it remains to prove that $v_t=-\ell(\omega)$ where $\ell(\omega)$  is the number of parts of $\omega$. Let $\psi(\omega)=(c_k)_{k\in\bbbz}$ be the word corresponding to $\omega$. Then by Definition \ref{def:vcoding}, $v_t=\min\lbrace k\in\bbbz \mid c_k=1\rbrace$. The number of parts of $\omega$ then corresponds to the number of letters ``$0$'' whose indices are greater than $v_t$. Therefore
\begin{align*}
\ell(\omega)&=\sum_{i=1}^{t-1} (m_{\sigma(i)}-m_{\sigma(t)}-\mathds{1}_{\sigma(i)<\sigma(t)})\\
&=\sum_{i=1}^{t-1} m_{\sigma(i)}-(t-1)m_{\sigma(t)}-\sigma(t).
\end{align*}

By using $\sum_{i=1}^t m_{\sigma(i)}=0$, one derives
$\ell(\omega)=-v_t$. By setting $\mu\in\ccp$ such that $\mu_i=v_i-v_t+i-t$, we conclude the proof.

\subsubsection{Type $C^{(1)}_{t}$}

\begin{proof}[\textbf{Proof of Theorem \ref{prop:MacdotypeC}}]
First note that the right-hand side of \eqref{eq:misign} is the desired right-hand of \eqref{eq:macdoc}. We recall the left-hand side of \eqref{eq:misign} here:
\begin{multline}
\sum_{\mathbf{m}\in\bbbz^t}T^{(t+1)\lVert m\rVert^2+\sum_{i=1}^tm_i(i-t-1)}\frac{\underset{1\le i,j\le t}{\det}\left(x_i^{(2t+2)m_j+j-t-1}-x_i^{-((2t+2)m_j+j-t-1)}\right)}{\underset{1\le i,j\le t}{\det}(x_i^{j-t-1}-x_i^{-(j-t-1)})}.\label{eq:misignlhs}
\end{multline}

It remains to order the $m_j$'s to derive a sum of symplectic Schur functions $\sp$ (defined in \eqref{spdef}). Set $\omega\in\ccdd_{(g)}$ such that $\phi(\omega)=(0,m_1,\dots,m_t,0,-m_t,\dots,-m_1)$, where $\phi$ is the bijection from Theorem \ref{thm:gk}, $\mathbf{v}$ its $V_{g,t}$-coding with $g=2t+2$ as in Definition \ref{def:vcoding} and $\mathbf{r}=(v_1-g/2,\dots,v_t-g/2)$ the vector defined in Theorem \ref{thm:DDpair}. Let $\sigma'\in S_{g}$ be such that the sequence $(\phi(\omega)_{\sigma'(i)-1},\sigma'(i))_{i\in\lbrace 1,\dots,g\rbrace}$ is increasing with respect to the lexicographic order. Note that by definition of $\sigma$ in Definition \ref{def:vcoding}, for any $i\in\lbrace 1,\dots,t\rbrace$, $\sigma'(g-i)=\sigma(g-i)+1$. Hence for any $i\in\lbrace 1,\dots,t\rbrace$

$$gm_i+i-g/2=\begin{cases} r_{(\sigma')^{-1}(i)} \text{ if } gm_i+i-g/2<-(gm_i+i-g/2),\\
 -r_{(\sigma')^{-1}(i)} \text{ otherwise}.
 \end{cases}
$$
 Moreover the exponent of $T$ in \eqref{eq:misign} corresponds to the quadratic form \eqref{eq:gkddpair}. Recall that in Theorem \ref{thm:DDpair} the vector $\mathbf{r}$ is in bijective correspondence with $\mathbf{v}$ and so with $\omega\in\ccdd_{(g)}$. Hence the left-hand side of \eqref{eq:misign} becomes

\begin{equation*}
\sum_{\mathbf{r}\in\bbbz^t}\sgn(\sigma)T^{\lvert\omega\rvert/2}\frac{\det\left(x_i^{r_j}-x_i^{-r_j}\right)}{\det(x_i^{t+1-j}-x_i^{j-t-1})}.
\end{equation*}

Set $\mu_j:=r_j-g/2+j$ for all $j\in\lbrace 1,\dots,t\rbrace$. By definition of the $V_{g,t}$-coding and \eqref{eq:vcoding}, we have $\mu_1\geq \dots\geq \mu_t\geq 0$. By Lemma \ref{lem:signschur} $\sgn(\sigma)$ is equal to $(-1)^{d_{\omega}+\lvert H_{<g,+}\rvert}$ which concludes the proof.
\end{proof}

\subsubsection{Type $B^{(1)}_{t}$}

Following the same path as before, one rewrites the Macdonald identity for type $B^{(1)}_{t}$ as follows.
\begin{thm}\label{prop:Macdotypeb}
Set $t\in\mathbb{N}^*$ such that $t\geq 3$ and $g=2t-1$. The Macdonald identity for type $B^{(1)}_{t}$ can be rewritten as follows:
\begin{equation}\label{eq:macdob}
\sum_{\lambda\in\ccdd_{g}^{'1}} (-1)^{\lvert H_{<g,+}\rvert+\lvert H_{<g,+}\cap \Delta\rvert}T^{\lvert\lambda\rvert/2}\oo_{\mu}({\bf x})=\left(T;T\right)_\infty^{t} K_T(t,\mathbf{x})\prod_{i=1}^t \left(Tx_i^{\pm};T\right)_\infty,
\end{equation}
where $\mathbf{v}$ is the $V_{g,t}$-coding corresponding to $\lambda$ (see Definition \ref{def:vcoding}), $\Delta$ is the main diagonal of $\lambda$, $H_{<g,+}:=\lbrace s\in\omega, h_s<g, \varepsilon_s=1\rbrace$ and $\mu\in\ccp$ is such that
$\mu_i:=\left(v_i+i-g\right)$ for all $1\leq i \leq t$.
\end{thm}

\begin{proof}

Recall from Proposition \ref{prop:intermed_rewrite} the rewriting of the Macdonald identity in type $B_t^{(1)}$ \eqref{eq:misignbterm}
\begin{multline*}
\sum_{\mathbf{m}\in\mathbb{N}^*\times\bbbz^{t-1}}T^{(2t-1)\lVert m\rVert^2/2-m_1(2t-1)/2+\sum_{i=2}^tm_i(i-t-1/2)}\\\times\frac{\underset{1\le i,j\le t}{\det}\left(x_i^{(2t-1)m_j+j-t}-x_i^{-((2t-1)m_j+j-t-1)}\right)}{\underset{1\le i,j\le t}{\det}\left(x_i^{j-t}-x_i^{-(j-t)+1}\right)}
=\left(T;T\right)_\infty^{t} K_T(t,\mathbf{x})\prod_{i=1}^t \left(Tx_i^{\pm};T\right)_\infty ,
\end{multline*}
where $K_T(t,\mathbf{x})$ as defined in \eqref{eq:kt}.

By Proposition \ref{prop:ddprime2t-1core}, the exponent of $T$ is the quadratic form that corresponds to half the weight of $\lambda\in\ccdd_{2t-1}^{'1}$ such that 
\begin{itemize}
\item setting $\core_{2t-1}(\lambda)=\omega$ then $\phi(\omega)=(m_2,\dots,m_{t-1},-m_{t-1},\dots,-m_2,0)$,
\item  and $\quot_{2t-1}(\lambda)=\nu^{(2t-2)}$ with $\lvert\nu^{(2t-2)}\rvert=m_1(m_1-1)$.
\end{itemize}

%
%
As in the case of type $C_t^{(1)}$, let $\sigma$ as defined Definition \ref{def:vcoding}, and define $\sigma'\in S_g$ such that for any $i\in\lbrace 1,\dots,t\rbrace$, $\sigma'(g-i)=\sigma(g-i)+1$. By applications of $\sigma'$ to the columns of $\underset{1\le i,j\le t}{\det}\left(x_i^{(2t-1)m_j+j-t}-x_i^{-((2t-1)m_j+j-t-1)}\right)$ and Lemma \ref{lem:signschurtypeb}, one derives Theorem \ref{prop:Macdotypeb}.
\end{proof}

\subsubsection{Type $A^{(2)}_{2t-1}$}

From \eqref{eq:misignbvterm} of Proposition \ref{prop:intermed_rewrite}, by Proposition \ref{prop:ddprime2t}, using Lemma \ref{lem:signschurtypebv} and the same computations, one derives the following result.
\begin{thm}\label{prop:Macdotypebv}
Set $t\in\mathbb{N}^*$ such that $t\geq 3$ and $g=2t$. The Macdonald identity for type $A^{(2)}_{2t-1}$ can be rewritten as follows:
\begin{equation}\label{eq:macdobv}
\sum_{\lambda\in\ccdd_{g}^{'1}} (-1)^{d_{\lambda}+\lvert H_{<g,+}\rvert}T^{\lvert\lambda\rvert/2}\sp_{\mu}({\bf x})=\left(T;T\right)_\infty^{t} K_T(t,\mathbf{x})\prod_{i=1}^t \left(T^2x_i^{\pm 2};T^2\right)_\infty,
\end{equation}
where $\mathbf{v}$ is the $V_{g,t}$-coding corresponding to $\lambda$ (see Definition \ref{def:vcoding}), $d_\lambda$ is the length of the main diagonal of $\lambda$, $H_{<g,+}:=\lbrace s\in\omega, h_s<g, \varepsilon_s=1\rbrace$ and $\mu\in\ccp$ is such that
$\mu_i:=v_i+i-g$ for all $1\leq i \leq t$.
\end{thm}

\subsubsection{Type $D^{(2)}_{t+1}$}

Following the same steps and using Lemma \ref{lem:signschursc} one can also rewrite the Macdonald identities for type $D^{(2)}_{t+1}$ as follows.
\begin{thm}\label{prop:MacdotypeCv}
Set $t\in\mathbb{N}^*$ such that $t\geq 2$ and $g=2t$. The Macdonald identity for type $D^{(2)}_{t+1}$ can be rewritten as follows:
\begin{multline}\label{eq:macdocv}
\sum_{\omega\in\ccsc_{(g)}} (-1)^{\lvert H_{<g,+}\rvert+\lvert H_{<g,+}\cap \Delta\rvert+d_\omega}T^{\lvert\omega\rvert/2}\oo_{\mu}({\bf x})\\
=\left(T^{1/2};T^{1/2}\right)_\infty\left(T;T\right)_\infty^{t-1}
K_T(t,{\bf x})\prod_{i=1}^t \left(T^{1/2}x_i^{\pm};T^{1/2}\right)_\infty,
\end{multline}
where $\mathbf{v}$ is the $V_{g,t}$-coding corresponding to $\omega$ (see Definition \ref{def:vcoding}), $\Delta$ is the main diagonal of $\omega$, $H_{<g,+}:=\lbrace s\in\omega, h_s<g, \varepsilon_s=1\rbrace$ and $\mu\in\ccp$ is such that
$\mu_i:=v_i+i-g$ for all $1\leq i \leq t$.
\end{thm}
\begin{proof}

The exponent of the formal variable $T$ on the left-hand side of \eqref{eq:misigncv} from Proposition \ref{prop:intermed_rewrite} is the same quadratic form as  \eqref{eq:gkscpair} from Proposition \ref{prop:sc2t} when taking $\omega\in\ccsc_{(2t)}$ such that $\phi(\omega)=(m_1,\dots,m_t,-m_t,\dots,-m_1)$. Using the correspondence between $\omega$ and its $V_{g,t}$-coding and $\bf{r}$ as defined in Theorem \ref{thm:SC}, note that for any $i\in\lbrace 1,\dots,t\rbrace$ there exists $k\in\lbrace 1,\dots,t\rbrace$ such that $r_k=\max(gm_i+i-t,-gm_i-i+1+t)$.
By reordering the columns in a decreasing order, one derives \eqref{eq:macdocv}.
\end{proof}

\subsubsection{Type $A^{(2)}_{2t}$}
Following the same steps as before starting from \eqref{eq:misignbc} from Proposition \ref{prop:intermed_rewrite}, then using Proposition \ref{prop:dd2t+1} and Lemma \ref{lem:signschurbc}, one derives the following results
\begin{thm}\label{prop:Macdotypebc}
Set $t\in\mathbb{N}^*$ and $g=2t+1$. The Macdonald identity for type $A^{(2)}_{2t}$ can be rewritten as follows:
\begin{multline}\label{eq:macdobc}
\sum_{\omega\in\ccdd_{(g)}} (-1)^{\lvert H_{<g,+}\rvert+\lvert H_{<g,+}\cap \Delta\rvert}T^{\lvert\omega\rvert/2}\oo_{\mu}({\bf x})\\=\left(T;T\right)_\infty^tK_T(t,{\bf x})
\prod_{i=1}^t \left(Tx_i^{\pm};T\right)_\infty \left(Tx_i^{\pm 2};T^2\right)_\infty,
\end{multline}
where $\mathbf{v}$ is the $V_{g,t}$-coding corresponding to $\omega$ (see Definition \ref{def:vcoding}), $\Delta$ is the main diagonal of $\omega$, $H_{<g,+}:=\lbrace s\in\omega, h_s<g, \varepsilon_s=1\rbrace$ and $\mu\in\ccp$ is such that
$\mu_i:=v_i+i-g$ for all $1\leq i \leq t$.
\end{thm}

\subsubsection{Type $D^{(1)}_{t}$}
Similarly, using Lemma \ref{lem:signschurtyped}, Theorem \ref{thm:DDtquotpaird} and the same computations, one derives the rewriting of Macdonald identity for type $D^{(1)}_{t}$.
\begin{thm}\label{prop:Macdotyped}
Set $t\in\mathbb{N}^*$ such that $t\geq 4$ and $g=2t-2$. The Macdonald identity for type $D^{(1)}_{t}$ can be rewritten as follows:
\begin{equation}\label{eq:macdod}
\sum_{\lambda\in\ccdd_{g}^{'2}} (-1)^{d_{\lambda}+\lvert H_{<g,+}\rvert}T^{\lvert\lambda\rvert/2}\oe_{\mu}({\bf x})=
\left(T;T\right)_\infty^{t} K_T(t,\mathbf{x}),
\end{equation}
where $\mathbf{v}$ is the $V_{g,t}$-coding corresponding to $\lambda$ (see Definition \ref{def:vcoding}), $d_\lambda$ the length of the main diagonal of $\omega$, $H_{<g,+}:=\lbrace s\in\omega, h_s<g, \varepsilon'_s=1\rbrace$ and $\mu\in\ccp$ is such that
$\mu_i:=\left(v_i+i-g\right)$ for all $1\leq i \leq t$.
\end{thm}
\begin{proof}

Starting from \eqref{eq:misignd} of Proposition \ref{prop:intermed_rewrite}, the proof is done using Proposition \ref{prop:ddprime2t-2} to map the lattice to integer partitions, then reorder columns so that the corresponding $V_{g,t}$-coding appears in decreasing order and Lemma \ref{lem:signschurtyped} to convert the signature of the permutation into a partition statistics. It remains only to verify that the sum of fraction of determinants appearing as coefficients of formal variable $T$ actually match with the  Definition \ref{oedef} of even orthogonal character. This is ensured by the fact that $m_t=0$ is equivalent to $\mu_t=0$, which concludes the proof of \eqref{eq:macdod}.
\end{proof}

\medskip
The following table summarizes the results of the previous subsection for each type with the notations of Macdonald \cite{Mac} in parenthesis. Remark that the the classical type of Schur function appearing for each affine type can actually be seen on the affine Dynkin diagram: it is the one associated with type of the Dynkin diagram obtained when removing the node $0$ (see \cite{LW} for more details):
\begin{center}
\begin{tabular}{|l|c|c|}
\hline
type & family & Schur function \\
\hline
\rule{0pt}{15pt}$A^{(1)}_{t-1}$ ($\tilde{A}_{t-1}$) &  $\ccp_{(t)}$ &  Classical\\
\hline
\rule{0pt}{15pt}$B^{(1)}_{t}$ ($\tilde{B}_t$) &  $\ccdd^{'1}_{2t-1}$ &  Special orthogonal\\
\hline
\rule{0pt}{15pt}$A^{(2)}_{2t-1}$ ($\tilde{B}_t^{\vee}$) &  $\ccdd^{'1}_{2t}$ &  Symplectic \\
\hline
\rule{0pt}{15pt}$C^{(1)}_{t}$ ($\tilde{C}_t$)&  $\ccdd_{(2t+2)}$ &  Symplectic\\
\hline
\rule{0pt}{15pt}$D^{(2)}_{t+1}$ ($\tilde{C}_{t}^{\vee}$)&  $\ccsc_{(2t)}$ &  Special orthogonal\\
\hline
\rule{0pt}{15pt}$A^{(2)}_{2t}$ ($\tilde{BC}_{t}$)&  $\ccdd_{(2t+1)}$ &  Special orthogonal\\
\hline
\rule{0pt}{15pt}$D^{(1)}_{t}$ ($\tilde{D}_{t}$)&  $\ccdd^{'2}_{2t-2}$ &  Even orthogonal\\
\hline

\end{tabular}
\captionof{table}{Table of Macdonald identities with their corresponding partition family and their corresponding associated character.}\label{tableau}

\end{center}

\section{$q$-Nekrasov--Okounkov formulas and confluences}\label{sec:NO}

\subsection{$q$-Nekrasov--Okounkov formulas}\label{subsecc:qnp}
In this section, we combine the results obtained in Sections \ref{chap3:hooks} and \ref{sec:Macdo} to derive $q$-Nekrasov--Okounkov type formulas for all infinite types except for the type $A^{(1)}_{t-1}$ already given in \eqref{Hande}. The proofs of all theorems below follow the same steps:
\begin{enumerate}
\item Specialize all the Macdonald identities stated in Section \ref{sec:Macdo}  either taking $x_i=q^i$ or $x_i=q^{2i-1}$. Then
\begin{enumerate}
\item reinterpret the specialized symplectic, special orthogonal Schur functions as a product over hook lengths using one of the theorems in Section \ref{chap3:hooks} with an appropriate choice of the function $\tau$ on the sum part,
\item and rewrite the products with some technical lemmas.
\end{enumerate}  
\item Then setting $u=q^g$ or $u=q^{g/2}$, verify that the coefficients in the formal variable $T$ on both sides of the equalities derived at the end of the previous step can be seen as Laurent polynomials in $u$.
\end{enumerate}



For instance, to derive Theorem \ref{NOC}, we will use a polynomiality argument by remarking that for any positive integer $n$, the coefficients of $T^n$ on both sides of \eqref{eq:noc} can be seen as Laurent polynomials in $u$ with coefficient in $\mathbb{C}(q)$. To see that, note that for any $\lambda\in\ccdd$ we have:
\begin{multline*}
\prod_{s\in\lambda}\frac{1-u^{-2\varepsilon_s}q^{h_s}}{1-q^{h_s}}\prod_{s\in\Delta}\frac{1+uq^{h_s/2}}{1+u^{-1}q^{h_s/2}}=\prod_{\substack{s\in\lambda\\s\notin \Delta}}\frac{1-u^{-2\varepsilon_s}q^{h_s}}{1-q^{h_s}}\prod_{s\in\Delta}\frac{(1+uq^{h_s/2})(1+u^{-1}q^{h_s/2})}{1-q^{h_s}},
\end{multline*}
therefore the left-hand side belongs to $\mathbb{C}(q)[u;u^{-1}]$. By expanding the quotient on the right-hand side of \eqref{eq:noc} and extracting the coefficients in $T^n$, we obtain the Laurent polynomiality in $u$ of the coefficients. Hence it is enough to prove that both sides coincide at $u=q^t$ for infinitely many positive integers $t$.

The following Lemma, whose proof can be found in the Appendix \ref{ap:sec1}, is the key ingredient in step $(1).(b)$.
\begin{lm}\label{lem:lemtech}
Let $x$ be a formal variable and $t$ a positive integer. Then we have:
\begin{equation}\label{techsimple}
\prod_{i=1}^t\left(1-q^ix\right)=\prod_{r\geq 1}\frac{1-q^{r}x}{1-q^{r+t}x},
\end{equation}
\begin{equation}\label{techsimpleneg}
\prod_{i=1}^t\left(1-q^{-i}x\right)=\prod_{r\geq 1}\frac{1-q^{r-1-t}x}{1-q^{r-1}x},
\end{equation}

\begin{align}\label{techstrictplus}
\prod_{1\leq i<j\leq t}\left(1-q^{i+j}x\right)=\prod_{r\geq 1}\frac{\left(1-q^{r+2}x\right)^{\lfloor (r+1)/2\rfloor}\left(1-q^{r+1+2t}x\right)^{\lceil r/2 \rceil}}{\left(1-q^{r+1+t}x\right)^r},
\end{align}

\begin{align}\label{techstrictmoins}
\prod_{1\leq i<j\leq t}\left(1-q^{-i-j}x\right)=\prod_{r\geq 1}\frac{\left(1-q^{r-1}x\right)^{\lfloor (r+1)/2\rfloor}\left(1-q^{r-2t}x\right)^{\lceil r/2 \rceil}}{\left(1-q^{r-1-t}x\right)^r},
\end{align}

and
\begin{equation}\label{techtypea}
\prod_{1\leq i<j\leq t}\left(1-q^{j-i}x\right)\left(1-q^{i-j}x\right)=\frac{1}{\left(1-x\right)^{t-1}}\prod_{r\geq 1}\frac{\left(1-q^{r-t}x\right)^r\left(1-q^{r+t}x\right)^r}{\left(1-q^{r+1}x\right)^r\left(1-q^{r-1}x\right)^r}.
\end{equation}


\end{lm}


\subsubsection{Type $A^{(1)}_{t-1}$}

This corresponds to the result \eqref{Hande}, obtained in parallel by Dehaye--Han \cite{HD} and Iqbal--Nazir--Raza--Saleem \cite{INRS} and is derived using Theorem \ref{thm:HD} and Theorem \ref{prop:MacdotypeA} setting $x_i=q^i$ and $\tau(x)=1-q^x$ and Lemma \ref{lem:lemtech} \eqref{techtypea} to rewrite the product in the right-hand side of \eqref{eq:macdoa} as the right-hand side of \eqref{Hande}.

\subsubsection{Enumeration of hook lengths on the main diagonal}

For all other types than type $A^{(1)}_{t-1}$, when completing the step $1.(a)$ detailed at the beginning of Section \ref{subsecc:qnp}, the specialization of the symplectic and special orthogonal Schur functions appearing in the rewritings of the Macdonald identity correspond to some specialization of the corresponding Weyl denominator formula. Some of the terms coming from these specializations of the Weyl denominator formula are lacking in the theorems from Section \ref{chap3:hooks}. For instance in Theorem \ref{thm:DDpair} in type $C_{t}^{(1)}$, the product over hook lengths is lacking a term of the form $\prod_{i=1}^t\tau(2r_i)/\tau(2i)$ to have the analogue of the term coming from the roots $2\varepsilon_i$ for $1\leq i \leq t$. In order to convert the specialization of symplectic and special orthogonal Schur functions into a hook length formula, one can note that the products corresponding the roots of the $2\varepsilon_i$ correspond to the box over hook lengths on the main diagonal $\Delta$ of the corresponding partition, as depicted in the following lemma:

\begin{lm}\label{lm:lemdiag}
Let $\lambda$ be a partition in one of the sets $\ccp_{(t)}$, $\ccdd_{(2t+2)}$,$\ccdd_{(2t+1)}$, $\ccsc_{(2t)},\ccdd_{2t}^{'1}$ and $\ccdd_{2t-1}^{'1}$. Let $\mathbf{v}$ its corresponding $V_{g,t}$-coding. Set $g'=g/2-1/2\mathds{1}_{\lambda\in \ccsc}-\mathds{1}_{\lambda\in \ccdd'}$ and as in the Theorems from Section \ref{chap3:hooks} set $\mathbf{r}=\mathbf{v}-g'\mathbf{1}$. We have:
$$\prod_{i=1}^t\frac{1-q^{2r_i}}{1-q^{2i}}=\prod_{\substack{s\in \Delta\\h_s<g}}(-q^{h_s-g})\prod_{s\in\Delta}
\frac{1-q^{h_s+g}}{1-q^{h_s-g}},$$
and
$$\prod_{i=1}^t\frac{1+q^{2r_i}}{1+q^{2i}}=\prod_{\substack{s\in \Delta\\h_s<g}}(q^{h_s-g})\prod_{s\in\Delta}
\frac{1+q^{h_s+g}}{1+q^{h_s-g}}.$$
\end{lm}

The proof follows after remarking from Lemmas from Section \ref{chap3:hooks} that the set of hook lengths of the boxes $s$ on the main diagonal $\Delta$ is actually the set $\lbrace
2r_i-2kg,1\leq i \leq t, 0\leq k\leq \lfloor (r_i-i)/g\rfloor\rbrace$. For a box $s$ such that $h_s<g$, we use the following relation for any $X$ (where $X$ is either $q^{h_s-g}$ or $-q^{h_s-g}$):
$$\frac{1+X^{-1}}{1+X}=X^{-1}.$$


\subsubsection{Type $C^{(1)}_{t}$}

 The left-hand side of \eqref{eq:noc} can be obtained by setting $\tau(x)=1-q^x$ in Theorem \ref{thm:DDpair} and multiplying the resulting expression by the product of hook lengths on the main diagonal $\Delta$. We prove the following result.
\begin{lm}\label{lem:schurinter}
Let $t$ be a positive integer, set $\lambda\in\ccdd_{(2t+2)}$, $g=2t+2$ and $\mathbf{v}\in\bbbz^t$ its associated $V_{2t+2,t}$-coding. Set $H_{<g,+}=\#\lbrace s\in\lambda h_s<g, \varepsilon_s=1\rbrace$.
 Set $\mu=(v_1-2t-1,v_2-2t,\dots,v_t-t-2)\in\ccp$ and . Then
 \begin{equation}\label{eq:schursp}
\sp_{\mu}(q,q^2,\dots,q^t)=(-1)^{\lvert H+\rvert} q^{(t+1)d_\lambda}\prod_{s\in\lambda}\frac{1-q^{h_s-(2t+2)\varepsilon_s}}{1-q^{h_s}}\prod_{s\in\Delta}\frac{1+q^{t+1+h_s/2}}{1+q^{-t-1+h_s/2}}.
 \end{equation}
 \end{lm}
 
\begin{proof}
The proof of \eqref{eq:schursp} is done by induction on the length of the main diagonal $\Delta$ of $\lambda$. We show that it follows the same induction as for Theorem \ref{thm:DDpair} with $\tau(x)=1-q^x$. 
When $\lambda$ is the empty partition, $\mathbf{v}=(2t+1,\dots,t+2)$ then $\mu$ is the empty partition. This concludes the initialization step.
Set $\lambda\in\ccdd_{(2t+2)}$, $\mathbf{v}$ its associated $V_{2t+2,t}$-coding, then $\mu=(v_1-2t-1,\dots,v_t-t-2)\in\ccp$ and let $H_1$ be the largest hook of $\lambda$. Then define $\tilde{\lambda}=\lambda\setminus H_1\in\ccdd_{(2t+2)}$ the partition whose Ferrers diagram is the one of $\lambda$ where $H_1$ as been removed. Let $\mathbf{v'}$ be its associated $V_{2t+2,t}$-coding and $\tilde{\mu}=(v'_1-2t-1,\dots,v'_t-t-2)\in\ccp$. Setting $r_i=v_i-t-1$ and $r'_i=v'_i-t-1$, we have
\begin{equation}\label{eq:intersp}
\frac{\sp(\mu)}{\sp(\tilde{\mu})}=\frac{\underset{1\le i,j\le t}{\det}\left(q^{i(v_j-t-1)}-q^{-i(v_j-t-1)}\right)}{\underset{1\le i,j\le t}{\det}\left(q^{i(v'_j-t-1)}-q^{-i(v'_j-t-1)}\right)}=\frac{\underset{1\le i,j\le t}{\det}\left(q^{(i-t-1)r_j}-q^{-(i-t-1)r_j}\right)}{\underset{1\le i,j\le t}{\det}\left(q^{(i-t-1)r'_j}-q^{-(i-t-1)r'_j}\right)}.
\end{equation}
Set $g=2t+2$. We first handle the case $g<v_1<3g/2$. In this case, $v_1=g+\sigma(1)$ and $\sigma(1)<g/2$, with $\sigma$ as defined in Definition \ref{def:vcoding}. Then $\sigma(1)<g-\sigma(1)$, $\lambda$ has for largest hook length $h_{(1,1)}=2\sigma(1)\in\Delta$ and $h_{(1,1)}/2-t-1=\sigma(1)-g/2<0$. There exists $2\leq l\leq t$ such that $r'_l=g/2-\sigma(1)=-r_1+g$, and for any $i\in\lbrace 1,\dots,l-1\rbrace$ we have $r'_i=r_{i+1}$, and $r'_i=r_i$ otherwise. Moreover if we take $\psi(\lambda)=(c_k)_{k\in\bbbz}$ the word corresponding to $\lambda$, then $c_{-\sigma(1)}=1$. For $1\leq i\leq l$, we have $g-\sigma(1)<v_i<g+\sigma(1)$. This implies that $-\sigma(1)<v_i$ and recall that $c_{v_i-g}=0$. Similarly for $l+1\leq i\leq t$, $v_i-g<-\sigma(1)$. Hence $\cch_{1,+}=\lbrace 2\sigma(1),\sigma(1), r_1+r_i-g \text{ for } 1\leq i\leq l\rbrace$ with $\cch_{1,+}$ as defined in Lemma \ref{lem:maxhook}.

By \eqref{def:deltac}, \eqref{eq:intersp} becomes
\begin{multline}\label{eq:premiercas}
\prod_{i=1}^t\frac{q^{-r_i}(1-q^{2r_i})}{q^{-r'_i}(1-q^{2r'_i})}\prod_{i=1}^{t-1}\frac{q^{-(t-i)r_i}}{q^{-(t-i)r'_i}}\prod_{j=i+1}^{t}\frac{(1-q^{r_i-r_j})(1-q^{r_i+r_j})}{(1-q^{r'_i-r'_j})(1-q^{r'_i+r'_j})}\\
=q^{-2r_1+g+\sum_{i=2}^l r_i-r_1(2t-l-1)+g(t-l)}\frac{(1-q^{2r_1})}{1-q^{2(g-r_1)}}\\
\times\prod_{j=2}^{l-1}\frac{(1-q^{r_1-r_j})(1-q^{r_1+r_j})}{(1-q^{r_j+r_1-g})(1-q^{r_j-r_1+g})}\prod_{j=l+1}^{t}\frac{(1-q^{r_1-r_j})(1-q^{r_1+r_j})}{(1-q^{g-r_j-r_1})(1-q^{r_j-r_1+g})}.
\end{multline}

Observe that, using $tg=-g/2+\sum_{i=1}^{2t+1}i$, the first term in the right-hand side of \eqref{eq:premiercas} can be rewritten as
\begin{multline*}
q^{-tg+2(g-r_1)+\sum_{j=2}^t(r_j-r_1+g)-\sum_{j=l+1}^t(r_1+r_i-g)}\\
=q^{-g/2++\sum_{i=1}^{2t+1}i+2(g-r_1)+\sum_{j=2}^t(r_j-r_1+g)-\sum_{j=l+1}^t(r_1+r_i-g)}.
\end{multline*}
 Then using that $q^{x}=-(1-q^{x})/(1-q^{-x})$, for any $x\in\mathbb{C}^*$, this term becomes
\begin{multline*}
q^{t+1}\left(-\frac{(1-q^{g-r_1})(1+q^{g-r_1})}{(1-q^{r_1-g})(1+q^{r_1-g})}\right)\\ \times
\prod_{i=1}^{2t+1}\left(-\frac{1-q^{-i}}{1-q^i}\right)\prod_{j=2}^t\left(-\frac{1-q^{r_j-r_1+g}}{1-q^{r_1-g-r_j}}\right)
\prod_{j=l+1}^t\left(-\frac{1-q^{g-r_1-r_j}}{1-q^{r_1+r_j-g}}\right).
\end{multline*}
This can be rewritten as
$$\frac{1+q^{t+1+h_{(1,1)}/2}}{1+q^{-t-1+h_{(1,1)}/2}}\prod_{i=1}^{2t+1}\left(-\frac{1-q^{-i}}{1-q^{i}}\right)^{\alpha_i(H_{1})}.$$

Using the induction property, it concludes the proof of the case $g<v_1<3g/2$.
As in the proof of Theorem \ref{thm:DDpair}, there are two cases left to deal with to perform the induction:
\begin{enumerate}
\item either $v'_1=v_1-2t-2$ and $v'_i=v_i$ for $2\leq i\leq t$,
\item or there exists $1\leq l\leq t$ such that for any $1\leq i<l$ $v'_i=v_{i+1}$, $v'_l=v_1-2t-2$ and for any $l+1\leq i\leq t$, $v'_i=v_i$.
\end{enumerate}
Moreover by \eqref{lem:Delta}, note that in both cases $h\in H_1\cap\Delta$ implies $h=2r_1-2t-2$.

We handle these cases by using Lemma \ref{lm:lemdiag} in which $q$ is replaced by $q^{1/2}$, $1-x^2=(1-x)(1+x)$, $q^{-\sum_{i=1}^{2t+1}i}=\prod_{i=1}^{2t+1}\left(-\frac{1-q^{-i}}{1-q^{i}}\right)$ and the same reasoning as for the proof of Theorem \ref{thm:DDpair}.
\end{proof} 

%


When $u=q^{t+1}$, the product over hook lengths on the left-hand side of \eqref{eq:noc} is zero unless $\lambda$ has no hook length equal to $2t+2$, i.e. $\lambda\in\ccdd_{(2t+2)}$. By Lemmas \ref{lem:schurinter} and \ref{lem:lemtech} \eqref{techsimple}--\eqref{techtypea}, \eqref{eq:noc} coincides with the Macdonald identity \eqref{eq:macdoc} for type $C^{(1)}_{t}$ in which $x_i=q^i$. Using the remark that both sides of the equality of Theorem \ref{NOC} seen as a power series in $T$ whose coefficients belong to $\mathbb{C}[[q]][u,u^{-1}]$, this completes the proof of Theorem \ref{NOC}.

\subsubsection{Type $B^{(1)}_{t}$}

Following the same steps as for Lemma \ref{lem:schurocv}, using $\sum_{i=1}^{2t-2}i=(t-1)(2t-1)$ and Lemma \ref{lm:lemdiag}, we derive the following from Theorem \ref{thm:DDtquotimpair} with $\tau(x)=1-q^{2x}$.
\begin{lm}\label{lem:schurob}
Let $t$ be a positive integer, set $\lambda\in\ccdd_{2t-1}^{'1}$, $g=2t-1$ and $\mathbf{v}\in\bbbz^t$ its associated $V_{2t-1,t}$-coding. Set $H_{<g,+}=\#\lbrace s\in\lambda h_s<g, \varepsilon_s=1\rbrace$.
 Set $\mu=(v_1-2t+2,v_2-2t+3,\dots,v_t-t+1)\in\ccp$. Then
 \begin{multline*}\label{eq:schursob}
\oo_{\mu}(q,q^3,\dots,q^{2t-1})=(-1)^{\lvert H_{<g,+}\rvert+\lvert H_{<g,+}\cap \Delta\rvert} q^{-(2t-1)d_\lambda}\\\times\prod_{s\in\lambda}\frac{1-q^{2h_s-2(2t-1)\varepsilon_s}}{1-q^{2h_s}}\prod_{s\in\Delta}\frac{1-q^{2t-1+h_s}}{1-q^{-2t+1+h_s}}.
 \end{multline*}
 \end{lm}

Following the same steps, using Theorem \ref{prop:Macdotypeb} and Lemma \ref{lem:schurob}, rewriting the products with Lemma \ref{lem:lemtech}, setting $x_i=q^{2i-1}$, $u=q^{2t-1}$, one derives the next result.
\begin{thm}[A $q$-Nekrasov--Okounkov formula for type $B^{(1)}_{t}$]\label{NOB}
For formal variables $T$, $q$ and $u$, we have:
\begin{multline*}
\sum_{\lambda\in\ccdd'} u^{-d_\lambda}T^{\lvert \lambda\rvert/2}\prod_{s\in\lambda}\frac{1-u^{-2\varepsilon_s}q^{2h_s}}{1-q^{2h_s}}\prod_{s\in\Delta}\frac{1-uq^{h_s}}{1-u^{-1}q^{h_s}}\\=
\prod_{m,r\geq 1}\frac{1-u^{-1}q^{2(r-1)}T^{m}}{1-uq^{2r}T^{m}}\frac{\left(1-u^{-2}q^{2r}T^m\right)^{\lceil r/2 \rceil}\left(1-u^2q^{2(r+1)}T^m\right)^{\lceil r/2 \rceil}}{\left(1-q^{2r}T^m\right)^{\lceil (r+1)/2 \rceil}\left(1-q^{2(r+2)}T^m\right)^{\lceil r/2 \rceil}}.
\end{multline*}
\end{thm}

Like in the case of Theorem \ref{NOC}, if we set $u=q^{2t-1}$ and $x_i=q^{2i-1}$ in Theorem \ref{prop:Macdotypeb}, \eqref{eq:macdob} coincides with the formula of Theorem \ref{NOB}. We conclude by remarking that on both sides the coefficients of the powers of $T$ belong to $\mathbb{C}[[q]][u,u^{-1}]$.

\subsubsection{Type $A^{(2)}_{2t-1}$}
%

Following the same steps of the proof of Lemma \ref{lem:schurinter}, showing that the induction steps are the same as the ones of Theorem \ref{thm:DDtquotpair} with $\tau(x)=1-q^{x}$, using \eqref{lem:Deltaddprime}, Lemma \ref{lm:lemdiag} in which $q$ is replaced by $q^{1/2}$ and $t(2t-1)=\sum_{i=1}^{2t-1}i$, we derive the following result.
\begin{lm}\label{lem:schurbv}
Let $t$ be a positive integer, set $g=2t$, $\lambda\in\ccdd_{g}^{'1}$ and $\mathbf{v}\in\bbbz^t$ its associated $V_{2t,t}$-coding.
Set $H_{<g,+}=\#\lbrace s\in\lambda h_s<g, \varepsilon_s=1\rbrace$. Set $\mu=(v_1-2t+1,v_2-2t+2,\dots,v_t-t)\in\ccp$. Then
 \begin{equation*}
\sp_{\mu}(q,q^2,\dots,q^t)=(-1)^{\lvert H+\rvert} q^{-td_\lambda}\prod_{s\in\lambda}\frac{1-q^{h_s-2t\varepsilon_s}}{1-q^{h_s}}\prod_{s\in\Delta}\frac{1+q^{t+h_s/2}}{1+q^{-t+h_s/2}}.
 \end{equation*}
 \end{lm}
 
 Following the same steps, using Theorem \ref{prop:Macdotypebv} and Lemma \ref{lem:schurbv}, rewriting the products with Lemma \ref{lem:lemtech}, setting $x_i=q^{i}$, $u=q^{t}$, one derives the following $q$-Nekrasov--Okounkov formula.
\begin{thm}[A $q$-Nekrasov--Okounkov formula for type $A^{(2)}_{2t-1}$]\label{NOBV}
For formal variables $T$, $q$ and $u$, we have:
\begin{multline*}
\sum_{\lambda\in\ccdd'} (-u)^{-d_\lambda}T^{\lvert \lambda\rvert/2}\prod_{s\in\lambda}\frac{1-u^{-2\varepsilon_s}q^{h_s}}{1-q^{h_s}}\prod_{s\in\Delta}\frac{1+uq^{h_s/2}}{1+u^{-1}q^{h_s/2}}\\=
\prod_{m\geq 1}\left(1+u^{-1}T^{m}\right)\prod_{r\geq 1}\frac{\left(1+u^{-1}q^rT^{m}\right)}{\left(1+u q^{r}T^{m}\right)}\frac{ \left(1-q^{r+1}u^{2}T^m\right)^{\lceil r/2 \rceil}\left(1-q^{r}u^{-2}T^m\right)^{\lceil r/2 \rceil}}{\left(1-q^{r}T^m\right)^{\lceil r/2 \rceil}\left(1-q^{r+1}T^m\right)^{\lceil r/2 \rceil}}.
%
\end{multline*}
\end{thm}

Like in the case of Theorem \ref{NOC}, if we set $u=q^{t}$ and $x_i=q^{i}$ in Theorem \ref{prop:Macdotypebv}, \eqref{eq:macdobv} coincides with the formula of Theorem \ref{NOBV}. We conclude by remarking that both sides are Laurent polynomials in $u$.

\subsubsection{Type $D^{(2)}_{t+1}$}

The following lemma corresponds to the specialization of odd orthogonal Schur functions at $x_i=q^{2i-1}$.

\begin{lm}\label{lem:schurocv}
Let $t$ be a positive integer, set $\lambda\in\ccsc_{(2t)}$, $g=2t$ and $\mathbf{v}\in\bbbz^t$ its associated $V_{2t,t}$-coding. Set $H_{<g,+}=\#\lbrace s\in\lambda h_s<g, \varepsilon_s=1\rbrace$. Set $\mu=(v_1-2t+1,v_2-2t+2,\dots,v_t-t)\in\ccp$. Then
 \begin{equation}\label{eq:schursocv}
\oo_{\mu}(q,q^3,\dots,q^{2t-1})=(-1)^{\lvert H_{<g,+}\rvert+\lvert H_{<g,+}\cap \Delta\rvert} \prod_{s\in\lambda}\frac{1-q^{2h_s-4t\varepsilon_s}}{1-q^{2h_s}}\prod_{s\in\Delta}\frac{1-q^{2t+h_s}}{1-q^{-2t+h_s}}.
 \end{equation}
 \end{lm}

The proof of Lemma \ref{lem:schurocv} follows the exact same path as the one of Lemma \ref{lem:schurinter} with the difference that in Theorem \ref{thm:SC} with $\tau(x)=1-q^{2x}$, the product $\prod_{i=1}^t (1-q^{2r_i})/(1-q^{2i})$ is missing and instead of using \eqref{def:deltac}, we use \eqref{def:deltab}. As in the proof of Lemma \ref{lem:schurinter}, this product comes from the product over the terms belonging to the main diagonal of $\lambda\in\ccsc_{(2t)}$. As explained in the proof of Lemma \ref{lm:lemdiag}, using $(1-q^{-i})/(1-q^i)=-q^{-i}$, this product induces a term $(-1)^{\lvert H_{<g,+}\cap\Delta\rvert}$ in \eqref{eq:schursocv}. Whereas in all other lemmas, there is a factor $q$ to some power, this factor is missing in \eqref{eq:schursocv} because, at each step of the induction, this power of $q$ yielding the dependency on $d_\lambda$ for the other types, is equal to $2t(2t-1)-2\sum_{i=1}^{2t-1}i=0$.

Following the same steps, using Theorem \ref{prop:MacdotypeCv} and Lemma \ref{lem:schurocv}, rewriting the products with Lemma \ref{lem:lemtech}, setting $x_i=q^{2i-1}$, $u=q^{2t}$ and a Laurent polynomiality argument, one derives the following $q$-Nekrasov--Okounkov type formula.
\begin{thm}[A $q$-Nekrasov--Okounkov formula for type $D^{(2)}_{t+1}$]\label{NOCV}
For formal variables $T$, $q$ and $u$, we have:
\begin{multline*}
\sum_{\lambda\in\ccsc} (-1)^{d_\lambda}T^{\lvert \lambda\rvert/2}\prod_{s\in\lambda}\frac{1-u^{-2\varepsilon_s}q^{2h_s}}{1-q^{2h_s}}\prod_{s\in\Delta}\frac{1-uq^{h_s}}{1-u^{-1}q^{h_s}}\\=
\prod_{m\geq 1}\frac{1}{1+T^{m/2}}\prod_{r\geq 1}\frac{1-u^{-1}q^{2r-1}T^{m/2}}{1-uq^{2r-1}T^{m/2}}\frac{\left(1-u^{-2}q^{2(r+1)}T^m\right)^{\lceil r/2 \rceil}\left(1-u^2q^{2r}T^m\right)^{\lceil r/2 \rceil}}{\left(1-q^{2r}T^m\right)^{\lceil (r+1)/2 \rceil}\left(1-q^{2(r+1)}T^m\right)^{\lceil r/2 \rceil}}.
\end{multline*}
\end{thm}


\subsubsection{Type $A^{(2)}_{2t}$}
Following the same steps as in the proof of Lemma \ref{lem:schurocv}, using Lemma \ref{lm:lemdiag}, one derives the following from Theorem \ref{thm:DDimpair} with $\tau(x)=1-q^{2x}$.
\begin{lm}\label{lem:schurobc}
Let $t$ be a positive integer, set $\lambda\in\ccdd_{(2t+1)}$, $g=2t+1$ and $\mathbf{v}\in\bbbz^t$ its associated $V_{2t+1,t}$-coding. Set $H_{<g,+}=\#\lbrace s\in\lambda h_s<g, \varepsilon_s=1\rbrace$. Set $\mu=(v_1-2t,v_2-2t+1\dots,v_t-t-1)\in\ccp$. Then
 \begin{multline*}\label{eq:schursobc}
\oo_{\mu}(q,q^3,\dots,q^{2t-1})=(-1)^{\lvert H_{<g,+}\rvert+\lvert H_{<g,+}\cap \Delta\rvert} q^{(2t+1) d_\lambda}\\
\times\prod_{s\in\lambda}\frac{1-q^{2h_s-2(2t+1)\varepsilon_s}}{1-q^{2h_s}}\prod_{s\in\Delta}\frac{1-q^{2t+1+h_s}}{1-q^{-2t-1+h_s}}.
 \end{multline*}
 \end{lm}
 
Using Theorem \ref{prop:Macdotypebc} and Lemma \ref{lem:schurobc}, rewriting the products with Lemma \ref{lem:lemtech}, setting $x_i=q^{2i-1}$, $u=q^{2t+1}$, one derives the next result.
\begin{thm}[A $q$-Nekrasov--Okounkov formula for type $A^{(2)}_{2t}$]\label{NOBC}
For formal variables $T$, $q$ and $u$, we have:
\begin{multline*}
\sum_{\lambda\in\ccdd} u^{d_\lambda}T^{\lvert \lambda\rvert/2}\prod_{s\in\lambda}\frac{1-u^{-2\varepsilon_s}q^{2h_s}}{1-q^{2h_s}}\prod_{s\in\Delta}\frac{1-uq^{h_s}}{1-u^{-1}q^{h_s}}\\=
\prod_{m,r\geq 1}\frac{1-u^{-1}q^{2r}T^{m}}{1-uq^{2r}T^{m}}\frac{1-u^{-2}q^{4r}T^{2m+1}}{1-u^2q^{4r}T^{2m+1}}\frac{\left(1-u^{-2}q^{2(r+1)}T^m\right)^{\lceil r/2 \rceil}\left(1-u^2q^{2r}T^m\right)^{\lceil r/2 \rceil}}{\left(1-q^{2r}T^m\right)^{\lceil (r+1)/2 \rceil}\left(1-q^{2(r+1)}T^m\right)^{\lceil r/2 \rceil}}.
\end{multline*}
\end{thm}

Like in the case of Theorem \ref{NOC}, if we set $u=q^{2t+1}$ and $x_i=q^{2i-1}$ in Theorem \ref{prop:Macdotypebc}, \eqref{eq:macdobc} coincides with the formula of Theorem \ref{NOBC}. We conclude by remarking that both sides are Laurent polynomials in $u$.

\subsubsection{Type $D^{(1)}_{t}$}
Recall from \cite[eqn $(2.6)$]{Kradet1} that:
\begin{equation}\label{eq:detkrat}
\underset{1\le i,j\le t}{\det}\left(x_i^{j-1/2}+x_i^{-(j-1/2)}\right)=\prod_{i=1}^t x_i^{-t+1/2}\prod_{1\leq i<j\leq t} \left((x_i-x_j)(1-x_ix_j)\right)\prod_{i=1}^t(1+x_i). 
\end{equation}

Similarly, by application of Theorem \ref{thm:DDtquotpaird} with $\tau(x)=1-q^{2x}$, and using the fact that $\tau(2x)/\tau(x)=1+q^x$ and \eqref{eq:detkrat}, one derives the next lemma: 

\begin{lm}\label{lem:schurd}
Let $t$ be a positive integer, set $\lambda\in\ccdd_{2t-2}^{'2}$, $g=2t-2$ and $\mathbf{v}\in\bbbz^t$ its associated $V_{2t-2,t}$-coding.Set $H_{<g,+}=\#\lbrace s\in\lambda h_s<g, \varepsilon'_s=1\rbrace$. Set $\mu=(v_1-2t+3,v_2-2t+4\dots,v_t-t+2)\in\ccp$. Then
 \begin{equation*}\label{eq:schuroe}
\oe_{\mu}(q,q^{3},\dots,q^{2t-1})=(-1)^{\lvert H_{<g,+}\rvert} q^{-4(t-1)d_\lambda}\prod_{s\in\lambda}\frac{1-q^{2h_s-2(2t-2)\varepsilon'_s}}{1-q^{2h_s}}.
 \end{equation*}
 \end{lm}
 
 
 Following the same steps, using Theorem \ref{prop:Macdotyped} and Lemma \ref{lem:schurd}, rewriting the products with Lemma \ref{lem:lemtech}, setting $x_i=q^{2i-1}$, $u=q^{2t-2}$, one derives our last result of this subsection.
\begin{thm}[A $q$-Nekrasov--Okounkov formula for type $D^{(1)}_{t}$]\label{NOd}
For formal variables $T$, $q$ and $u$, we have:
\begin{multline*}
\sum_{\lambda\in\ccdd'} (-u^{-2})^{d_\lambda}T^{\lvert \lambda\rvert/2}\prod_{s\in\lambda}\frac{1-u^{-2\varepsilon'_s}q^{2h_s}}{1-q^{2h_s}}
\\=
\prod_{m,r\geq 1}\frac{ \left(1-q^{2(r+2)}u^{2}T^m\right)^{\lceil r/2 \rceil}\left(1-q^{2(r-1)}u^{-2}T^m\right)^{\lceil r/2 \rceil}}{\left(1-q^{2(r+1)}T^m\right)^{\lceil r/2 \rceil}\left(1-q^{2(r+2)}T^m\right)^{\lceil r/2 \rceil}}.
%
\end{multline*}
\end{thm}

As for Theorem \ref{NOC}, if we set $u=q^{2t-2}$ and $x_i=q^{2i-1}$ in Theorem \ref{prop:Macdotyped}, then \eqref{eq:macdod} coincides with the formula of Theorem \ref{NOd}. We conclude by noting that both sides are Laurent polynomials in $u$.

\subsection{Confluences and connections with the Appendix $1$ of Macdonald's paper}\label{sec:macident}

In this subsection, we consider all of the specializations for infinite root systems in \cite[Appendix $1$]{Mac}. These specializations yield some Nekrasov--Okounkov type formulas that are already known and some others which are new. To clarify how Macdonald's specializations are to be transcribed in terms of the set of variables $\mathbf{x}$, recall that if we let $(e_i)_{1\leq i\leq t}$ be a basis for $\mathbb{R}^t$ and write $k+\varepsilon_i$ for the affine function $e_j\to k+\delta_{i,j}$. Then Rosengren--Schlosser set $T=e^{-1}$ and $x_i=T^{-\varepsilon_i}$ in their rewritings of Macdonald identities for all infinite affine types.

As for technical aspects and in order to avoid repetitions, we remind the reader of the following limit, which will be used in all of the specializations introduced in \cite[Appendix $1$]{Mac} to get new Nekrasov--Okounkov formulas. For any $(x,y)\in\bbbr^{*2}$, we have:
\begin{center}
$$\underset{q\rightarrow 1}{\lim}\frac{1-q^x}{1-q^y}=\frac{x}{y}.$$
\end{center}

Recall from Section \ref{subsecc:qnp} that one can specialize $u$ to a integer power of $q^m$, one can find a specialization of the Macdonald identity for the corresponding type. Now if we let $q$ go to $\pm 1$, both sides of the identity belong to $\mathbb{Q}[m][[T]]$, as already performed in type $A$ by Han in \cite{Han08}. Therefore one can lift this identity to any complex number $m$.

\subsubsection{Type $A^{(1)}_{t-1}$}

Both specializations in \cite[Appendix $1$]{Mac} give the same Nekrasov--Okounkov formula \eqref{NOdebut}.

\subsubsection{Type $B^{(1)}_{t}$}
Let $t$ be an integer greater than $3$.
In this case, in \cite{Mac}, $a_0=1-\varepsilon_1-\varepsilon_2$, and $a_i=\varepsilon_i-\varepsilon_{i+1}$ for $1\leq i\leq t-1$ and $a_t=\varepsilon_t$.  In all this section, we set $u=q^{2t-1}$.

\textbf{(a)} The first specialization in \cite{Mac} corresponds to $e^{a_i}\to 1$ for $1\leq i\leq t$. This is equivalent to taking $x_i=1$ in Theorem \ref{prop:Macdotypeb}. By letting $q\rightarrow 1$ in Theorem \ref{NOB}, using this time the fact that both sides are polynomials in $t$, one derives 
\begin{equation}\label{eq:nospecb}
\sum_{\lambda\in\ccdd'}T^{\lvert \lambda\rvert/2}\prod_{s\in\lambda}\left(1-\frac{(2z-1)\varepsilon_s}{h_s}\right)\prod_{s\in\Delta}\frac{2z-1+h_s}{h_s-2z+1}=\left(T;T\right)_{\infty}^{2z^2+z},
\end{equation}
where $z$ is any complex number. With this formulation, the issue is that it is not immediate that the product on the left-hand side of \eqref{eq:nospecb} is well-defined because of the term $h_s-2z+1$ at the denominator. Moreover this formula was stated in a different way in \cite[Theorem $3.20$]{MP}: the sum part was over the set of doubled distinct partitions and the product over hook lengths slightly differs. Recall that the set $\varepsilon\cch(\lambda):=\lbrace \varepsilon_s h_s, s\in \lambda\rbrace$ and $\varepsilon\cch(\lambda')$ only differ for the boxes shaded in yellow in Figure \ref{fig:vardd} and Figure \ref{fig:varddprime}: a horizontal (respectively vertical) strip of length $d_\lambda$ with sign $\varepsilon_s=1$ (respectively $\varepsilon_s=-1$) for $\lambda\in\ccdd'$ (respectively $\ccdd$). Therefore for $\lambda\in\ccdd'$, the left-hand side of \eqref{eq:nospecb} becomes
$$\prod_{s\in\lambda\setminus \Delta}\left(1-\frac{(2z-1)\varepsilon_s}{h_s}\right)\prod_{s\in\Delta}(h_s+2z-1)=(-1)^{d_{\lambda'}}T^{\lvert \lambda'\rvert/2}\prod_{s\in\lambda'}\left(1+\frac{(2z-1)\varepsilon_s}{h_s}\right).$$
Hence one can write \eqref{eq:nospecb} as \cite[Theorem $3.20$]{MP}
\begin{equation}\label{eq:nospecbmatthias}
\sum_{\lambda\in\ccdd}(-1)^{d_\lambda}T^{\lvert \lambda\rvert/2}\prod_{s\in\lambda}\left(1+\frac{(2z-1)\varepsilon_s}{h_s}\right)=\left(T;T\right)_{\infty}^{2z^2+z},
\end{equation}
where $z$ is any complex number.

\textbf{(b)}The second specialization is $e^{a_i}\to 1$ for $1\leq i\leq t-1$ and $e^{a_t}\to -1$.
This specialization is equivalent to taking $x_i=-1$ in Theorem \ref{prop:Macdotypeb}. It is equivalent to letting $q\to -1$ in Theorem \ref{NOB}. This implies that $u=q^{2t-1}\to -1$. Moreover recall that if $\lambda\in\ccdd'$, all the hook lengths of the boxes of $\Delta$ are even. Hence for any complex number $z$, we get
\begin{equation}\label{eq:nospecbmoinsun}
\sum_{\lambda\in\ccdd'}T^{\lvert \lambda\rvert/2}\prod_{s\in\lambda}\left(1-\frac{(2z-1)\varepsilon_s}{h_s}\right)=\left(\left(T;T\right)_{\infty}^{2z-3}\left(T^2;T^2\right)_{\infty}^{2}\right)^z,
\end{equation}
which is new.

\textbf{(c)} The third specialization is $e^{a_i}\to 1$ for $0\leq i\leq t-1$. This specialization is equivalent to taking $x_i\to T^{-1/2}$ in Theorem \ref{prop:Macdotypeb}. In this case, one cannot derive a formula with the $q$-analogues written in this paper. Nevertheless there is a $q$-Nekrasov--Okounkov formula that could be obtained using the same techniques as the ones we developed. To do so, one should first replace $x_i\to x_iT^{-1/2}$ in the Macdonald identity for type $B^{(1)}_{t}$ from Proposition \ref{prop:RS}. In this case, the quadratic form to the exponent of $T$ changes. The subset of partitions linked to the latter is another subset of partitions. Namely it is the subset $\ccsc_{2t-1}^{1}$ of self-conjugate partitions whose $(2t-1)$-quotient is empty except for the self-conjugate element of the quotient which is a square. Analogues of Theorem \ref{thm:SC} and Lemma \ref{lem:signschursc} can be obtained using the same methods but no details are given here. 
The Nekrasov--Okounkov formula corresponding to this specialization is the following:
 \begin{equation}\label{eq:nospecbweird}
\sum_{\lambda\in\ccsc}(-1)^{d_\lambda}T^{\lvert \lambda\rvert/2}\prod_{s\in\lambda}\left(1+\frac{(2z-1)\varepsilon_s}{h_s}\right)=\left(\left(T^{1/2};T^{1/2}\right)_{\infty}^{2}\left(T;T\right)_{\infty}^{2z+3}\right)^{z},
\end{equation}
where $z$ is a complex number.
Actually \eqref{eq:nospecbweird} can be obtained from the Nekrasov--Okounkov formula \eqref{eq:nospeccvmoins1} for type $D^{(2)}_{t+1}$ below, as already noted in \cite[Theorem $3.33$]{MP}. The proof of \eqref{eq:nospeccvmoins1} is detailed in the subsection dedicated to the type $D^{(2)}_{t+1}$: it suffices to substitute $T\to T^{1/2}$ and $z\to -z+1/2$ in \eqref{eq:nospeccvmoins1} to derive \eqref{eq:nospecbweird}.

%
%


%

\subsubsection{Type $A^{(2)}_{2t-1}$}
Let $t$ be an integer greater than $3$.
In this case, in \cite{Mac}, $a_0=1-\varepsilon_1-\varepsilon_2$, and $a_i=\varepsilon_i-\varepsilon_{i+1}$ for $1\leq i\leq t-1$ and $a_t=2\varepsilon_t$.  In all this section, we set $u=q^{t}$.

\textbf{(a)} The first specialization is $e^{a_i}\to 1$ for $1\leq i\leq t$. This is equivalent to taking $x_i=1$ in Theorem \ref{prop:Macdotypebv}. By letting $q\rightarrow 1$ in Theorem \ref{NOBV}, using this time the fact that both sides are polynomials in $t$, one derives the following new Nekrasov--Okounkov type formula
\begin{equation}\label{eq:nospecbv}
\sum_{\lambda\in\ccdd'}(-1)^{d_\lambda}T^{\lvert \lambda\rvert/2}\prod_{s\in\lambda}\left(1-\frac{2z\varepsilon_s}{h_s}\right)=\left(\left(T;T\right)_{\infty}^{z-1}\left(T^2;T^2\right)_{\infty}\right)^{2z+1},
\end{equation}
where $z$ is any complex number.

\textbf{(b)} The second specialization is $e^{a_i}\to 1$ for $0\leq i\leq t-1$. Once again, this specialization is equivalent to taking $x_i\to T^{-1/2}$ in Theorem \ref{prop:Macdotypebv} and one cannot derive a formula with the $q$-analogues written in this paper. Once again a $q$ analogue could be obtained using the same techniques but one still can derive a new Nekrasov--Okounkov type formula here by shifting $z\to -z+1/2$ in \eqref{eq:nospecbmoinsun}:
\begin{equation}\label{eq:nospecbvweirg}
\sum_{\lambda\in\ccdd'}T^{\lvert \lambda\rvert/2}\prod_{s\in\lambda}\left(1+\frac{2z\varepsilon_s}{h_s}\right)=\left(\left(T;T\right)_{\infty}^{z+1}\left(T^2;T^2\right)_{\infty}^{-1}\right)^{2z-1},
\end{equation}
where $z$ is a complex number.

\subsubsection{Type $C^{(1)}_{t}$}
Let $t$ be an integer greater than $2$.
In this case, in \cite{Mac}, $a_0=1-2\varepsilon_1$, and $a_i=\varepsilon_i-\varepsilon_{i+1}$ for $1\leq i\leq t-1$ and $a_t=2\varepsilon_t$. We set $u=q^{2t+2}$.

The specialization  is $e^{a_i}\to 1$ for $1\leq i\leq t$. This is equivalent to taking $x_i=1$ in Theorem \ref{prop:MacdotypeC}. By letting $q\rightarrow 1$ in Theorem \ref{NOB}, using the fact both sides are polynomials in $t$, one derives as in \cite[Theorem $3.1$]{MP} the formula \eqref{eq:nospecc}.
%

\subsubsection{Type $D^{(2)}_{t+1}$}
Let $t$ be an integer greater than $2$.
In this case, in \cite{Mac}, $a_0=1/2-\varepsilon_1$, and $a_i=\varepsilon_i-\varepsilon_{i+1}$ for $1\leq i\leq t-1$ and $a_t=\varepsilon_t$. In this section, we set $u=q^{2t}$.


\textbf{(a)} The first specialization is $e^{a_i}\to 1$ for $1\leq i\leq t$. This is equivalent to taking $x_i=1$ in Theorem \ref{prop:MacdotypeCv}. Moreover in \cite[Appendix $1$]{Mac}, $T$ is replaced with $T^2$. By letting $q\rightarrow 1$ in Theorem \ref{NOCV}, using this time the fact that both sides are polynomials in $t$, one derives the new Nekrasov--Okounkov type formula
\begin{equation}\label{eq:nospeccv}
\sum_{\lambda\in\ccsc}(-1)^{d_\lambda}T^{\lvert \lambda\rvert}\prod_{s\in\lambda\setminus\Delta}\left(1-\frac{2z\varepsilon_s}{h_s}\right)\prod_{s\in\Delta}\left(1+\frac{2z}{h_s}\right)=\left(\left(T;T\right)_{\infty}\left(T^2;T^2\right)_{\infty}^{z-1}\right)^{2z+1},
\end{equation}
where $z$ is any complex number.

\textbf{(b)} The second specialization is $e^{a_i}\to 1$ for $1\leq i\leq t-1$ and $e^{a_t}\to -1$.
 This is equivalent to taking $x_i=-1$ in Theorem \ref{prop:MacdotypeCv}. Moreover in \cite[Appendix $1$]{Mac}, $T$ is replaced with $T^2$. It is equivalent to letting $q\to -1$ in Theorem \ref{NOCV}. This implies that $u\to 1$. Moreover if $s$ is a box on the main diagonal $\Delta$ of a self-conjugate partition, then $h_s$ is odd. Therefore we derive the following result already proved in \cite[Theorem $3.24$]{MP}:
 \begin{equation}\label{eq:nospeccvmoins1}
\sum_{\lambda\in\ccsc}(-1)^{d_\lambda}T^{\lvert \lambda\rvert}\prod_{s\in\lambda}\left(1-\frac{2z\varepsilon_s}{h_s}\right)=\left(\frac{\left(T^2;T^2\right)_{\infty}^{z+1}}{\left(T;T\right)_{\infty}}\right)^{2z-1},
\end{equation}
where $z$ is any complex number.

\subsubsection{Type $A^{(2)}_{2t}$}
Let $t$ be an integer greater than $1$.
In this case, in \cite{Mac}, $a_0=1-2\varepsilon_1$, and $a_i=\varepsilon_i-\varepsilon_{i+1}$ for $1\leq i\leq t-1$ and $a_t=\varepsilon_t$. In this section, we set $u=q^{2t+1}$.

\textbf{(a)} The first specialization is $e^{a_i}\to 1$ for $1\leq i\leq t$. This is equivalent to taking $x_i=1$ in Theorem \ref{prop:Macdotypebc}.
By letting $q\rightarrow 1$ in Theorem \ref{NOBC}, using this time the fact that both sides are polynomials in $t$, one derives the new Nekrasov--Okounkov type formula
\begin{equation}\label{eq:nospecbc}
\sum_{\lambda\in\ccdd}T^{\lvert \lambda\rvert/2}\prod_{s\in\lambda\setminus\Delta}\left(1-\frac{(2z+1)\varepsilon_s}{h_s}\right)\prod_{s\in\Delta}\left(1+\frac{2z+1}{h_s}\right)=\left(\frac{\left(T;T\right)_{\infty}^{2z+3}}{\left(T^2;T^2\right)_{\infty}^{2}}\right)^{z},
\end{equation}
where $z$ is any complex number.

\textbf{(b)} The second specialization is $e^{a_i}\to 1$ for $0\leq i\leq t-1$.  This specialization is equivalent to taking $x_i\to T^{-1/2}$ in Theorem \ref{prop:Macdotypebc}. As for specializations \textbf{(c)} of type $B^{(1)}_{t}$ and \textbf{(b)} of type $A^{(2)}_{2t-1}$, one can derive a $q$-Nekrasov--Okounkov formula by substituting $x_i$ by $x_iT^{-1/2}$ in the Macdonald identity for type $A^{(2)}_{2t}$ of Proposition \ref{prop:RS}. The quadratic form to the exponent of $T$ changes and the subset of partitions naturally arising is $\ccsc_{(2t+1)}$. Using the same techniques, one gets analogues of Theorem \ref{thm:SC} and Lemma \ref{lem:signschursc} and one derives the following new Nekrasov--Okounkov type formula:
\begin{equation}\label{eq:nospecbcweird}
\sum_{\lambda\in\ccsc}T^{\lvert \lambda\rvert/2}\prod_{s\in\lambda}\left(1-\frac{(2z+1)\varepsilon_s}{h_s}\right)=\left(\left(T^{1/2};T^{1/2}\right)_{\infty}^{2}\left(T;T\right)_{\infty}^{2z-3}\left(T^{2};T^{2}\right)_{\infty}^{2}\right)^{z},
\end{equation}
where $z$ is any complex number.

\textbf{(c)} The third specialization is $e^{a_i}\to 1$ for $1\leq i\leq t-1$ and $e^{a_t}\to -1$. This is equivalent to $q\to-1$ and $u\to-1$ in Theorem \ref{NOBC}, and we derive the following new Nekrasov--Okounkov type formula:
\begin{equation}\label{eq:nospecbcmoins}
\sum_{\lambda\in\ccdd}(-1)^{d_\lambda}T^{\lvert \lambda\rvert/2}\prod_{s\in\lambda}\left(1-\frac{(2z+1)\varepsilon_s}{h_s}\right)=\left(T;T\right)_{\infty}^{2z^2-z},
\end{equation}
where $z$ is any complex number.
 
\textbf{(d)} The fourth specialization is $e^{a_0/2}\to -1$ and $e^{a_i}\to 1$ for $0\leq i\leq t-1$.  This specialization is equivalent to taking $x_i\to-T^{-1/2}$ in Theorem \ref{prop:Macdotypebc}. One could obtain this specialization by letting $q$ tend to $-1$ in the $q$-Nekrasov--Okounkov formula derived from the specialization \textbf{(b)} mentioned above. Nevertheless this specialization can also be obtained as remarked by Pétréolle \cite[Theorem $3.33$]{MP} by the shift $z\to z+1/2$ and $T\to T^{1/2}$ in \eqref{eq:nospeccvmoins1} and one gets the following Nekrasov--Okounkov type formula:
\begin{equation}\label{eq:nospecbcweirdmoins1}
\sum_{\lambda\in\ccsc}(-1)^{d_\lambda}T^{\lvert \lambda\rvert/2}\prod_{s\in\lambda}\left(1-\frac{(2z+1)\varepsilon_s}{h_s}\right)=\left(\frac{\left(T;T\right)_{\infty}^{2z+3}}{\left(T^{1/2};T^{1/2}\right)_{\infty}^{2}}\right)^{z},
\end{equation}
where $z$ is any complex number.

\subsubsection{Type $D^{(1)}_{t}$}
Let $t$ be an integer greater than $4$.
In this case, in \cite{Mac}, $a_0=1-\varepsilon_1-\varepsilon_2$, and $a_i=\varepsilon_i-\varepsilon_{i+1}$ for $1\leq i\leq t-1$ and $a_t=\varepsilon_{t-1}+\varepsilon_t$. We set $u=q^{2t-2}$.

The specialization  is $e^{a_i}\to 1$ for $1\leq i\leq t$. This is equivalent to taking $x_i=1$ in Theorem \ref{prop:Macdotyped}. By letting $q\rightarrow 1$ in Theorem \ref{NOd}, using this time the fact that both sides are polynomials in $t$, we have the new Nekrasov--Okounkov type formula given as \eqref{eq:nospecd} in the introduction.

\subsection*{Acknowledgements}
The author would like to thank the referees for extremely valuable feedback and suggestions, which have tremendously enhanced the quality of the paper. The author would also like to thank Jehanne Dousse, Benjamin Dupont, Marion Jeannin, Cédric Lecouvey, Philippe Nadeau, Nicolas Ressayre and Ole Warnaar for their help and the fruitful conversations they had. Any critical remark must be exclusively addressed to the author of this paper. A sagemath code to check the results of Section \ref{chap3:hooks} and Section \ref{subsecc:qnp} can be found on \href{https://www.idpoisson.fr/wahiche/indexEN.html}{the author's webpage} or can be asked by mail.

\appendix
\section{Proof of Lemma \ref{lem:lemtech}}\label{ap:sec1}

This section is devoted to prove the technical results of Lemma \ref{lem:lemtech}.
Equalities \eqref{techsimple} and \eqref{techsimpleneg} are immediately derived noting that their right-hand sides are telescopic products: indeed as pointed out by one of the reviewers, they correspond to the identity $(a;q)_t=(a;q)_\infty/(aq^{t};q)_\infty$ for $a=xq$, respectively $a=xq^{-t}$.
In order to prove \eqref{techstrictplus}, we will need to gather $i+j$ by values. Hence if we set $3\leq k\leq 2t-1$, we need to count how many $(i,j)\in\lbrace 1,\dots,t\rbrace^2$ such that $i<j$ and $i+j=k$ there are. For $i+j=k$, we have:
\begin{equation*}
\begin{array}{ccc}
\begin{cases}
1\leq i\leq t-1\\
i+1\leq j\leq t
\end{cases}
&\iff &
\begin{cases}
1\leq i\leq t-1\\
2i+1\leq k\leq i+t.
\end{cases}
\end{array}
\end{equation*}
Hence it is equivalent to
\begin{equation*}
\begin{array}{ccc}
\begin{cases}
t+1\leq k\leq 2t-1\\
k-t\leq i\leq \lfloor\frac{k-1}{2}\rfloor
\end{cases}
&\text{ or } &
\begin{cases}
3\leq k\leq t\\
1\leq i\leq \lfloor\frac{k-1}{2}\rfloor.
\end{cases}
\end{array}
\end{equation*}

\noindent Therefore gathering by values of $k$ in the left-hand side of \eqref{techstrictplus} yields:
\begin{multline*}
\prod_{k=3}^{t} \left(1-q^{k}x\right)^{\lfloor(k-1)/2\rfloor}\prod_{k=t+1}^{2t-1} \left(1-q^{k}x\right)^{\lfloor(k-1)/2\rfloor-k+t+1}\\=\prod_{k=1}^{2t-1} \left(1-q^{k}x\right)^{\lfloor(k-1)/2\rfloor}\prod_{k=t+1}^{2t-1} \left(1-q^{k}x\right)^{t+1-k}
\end{multline*}

\noindent Shifting $k\to k+t$ in the second product, the above equation is equivalent to
\begin{multline*}
\prod_{k=1}^{2t-1} \left(1-q^{k}x\right)^{\lfloor(k-1)/2\rfloor} \prod_{k=1}^{t-1} \left(1-q^{k+t}x\right)^{1-k}=\frac{\prod_{r\geq 1}\left(1-q^{r}x\right)^{\lfloor(r-1)/2\rfloor}}{\prod_{r\geq 2t}\left(1-q^{r}x\right)^{\lfloor(r-1)/2\rfloor}}\frac{\prod_{r\geq t} \left(1-q^{r+t}x\right)^{r-1}}{\prod_{r\geq 1} \left(1-q^{r+t}x\right)^{r-1}}\\
=\prod_{r\geq 1}\frac{\left(1-q^{r}x\right)^{\lfloor(r-1)/2\rfloor}}{\left(1-q^{r-1+2t}x\right)^{t-1+\lfloor r/2\rfloor}}\frac{ \left(1-q^{r-1+2t}x\right)^{t+r-2}}{\left(1-q^{r+t}x\right)^{r-1}}\\
=\prod_{r\geq 1} \frac{\left(1-q^{r}x\right)^{\lfloor(r-1)/2\rfloor}\left(1-q^{r-1+2t}x\right)^{r-1-\lfloor r/2\rfloor}}{\left(1-q^{r+t}x\right)^{r-1}},
\end{multline*}
which is the desired result by shifting $r\to r+2$ at the numerator and $r\to r+1$ at the denominator and using $\lfloor n\rfloor + _lceil n\rceil=n$ for any integer $n$.
%
%
The other equalities \eqref{techstrictmoins} and \eqref{techtypea} are proved similarly, which concludes the proof.





%
\bibliographystyle{amsalpha}
\bibliography{bilbio_init}
%

\end{document}